\documentclass[11pt,letterpaper,reqno]{amsart}

\usepackage{siunitx}
\usepackage[letterpaper,margin=.98in,bottom=1in,top=1in]{geometry}
\usepackage{mathrsfs}
\usepackage{amsmath,amsthm,amssymb}
\usepackage{hyperref}
\usepackage{enumitem}
\usepackage{xcolor}
\usepackage{graphicx}

\numberwithin{equation}{section}
\newtheorem{theorem}{Theorem}[section]
\newtheorem{lemma}[theorem]{Lemma}

\newtheorem{proposition}[theorem]{Proposition}
\theoremstyle{remark}
\newtheorem{remark}[theorem]{Remark}
\theoremstyle{definition}
\newtheorem{definition}[theorem]{Definition}

\newcommand{\fint}{\,-\mspace{-19.4mu}\int}
\newcommand{\R}{\mathbb{R}}
\newcommand{\N}{\mathbb{N}}
\newcommand{\C}{\mathbb{C}}
\newcommand\T{\mathbb T}
\newcommand\D{\mathbb D}

\renewcommand       {\Im}    {\operatorname{Im}} 
\DeclareMathOperator{\ran}   {\mathrm{ran}}
\DeclareMathOperator{\codim} {\mathrm{codim}}
\renewcommand       {\Re}    {\operatorname{Re}}  

\newcommand{\without}{\setminus}
\newcommand{\by}{\times}
\newcommand{\sub}{\subset}
\newcommand{\maps}{\colon}

\title{Multipole vortex patch equilibria for active scalar equations}
\author{Zineb Hassainia \and  Miles H. Wheeler}

\begin{document}

\begin{abstract}
  We study how a general steady configuration of finitely-many point vortices, with Newtonian interaction or generalized surface quasi-geostrophic interactions, can be desingularized into a steady configuration of vortex patches. The configurations can be uniformly rotating, uniformly translating, or completely stationary. Using a technique first introduced by Hmidi and Mateu \cite{Hmidi-Mateu} for vortex pairs, we reformulate the problem for the patch boundaries so that it no longer appears singular in the point-vortex limit. Provided the point vortex equilibrium is non-degenerate in a natural sense, solutions can then be constructed directly using the implicit function theorem, yielding asymptotics for the shape of the patch boundaries. As an application, we construct new families of asymmetric translating and rotating pairs, as well as stationary tripoles. We also show how the techniques can be adapted for highly symmetric configurations such as regular polygons, body-centered  polygons and nested regular polygons by integrating the appropriate symmetries into the formulation of the problem. 
\end{abstract}

\maketitle

\tableofcontents

\section{Introduction}

\subsection{Historical discussion}
In this note we consider the generalized surface quasi-geostrophic (gSQG)  equations,  which describe the evolution of the potential temperature $\omega$ through the transport equation
\begin{equation} \label{eqn:omega}
  \left\{ \begin{alignedat}{2}
    &\partial_{t}\omega+v\cdot\nabla\omega=0,\quad&&(t,x)\in\R_+\times\R^2, \\
    &v=\nabla^\perp\psi,\\
    &\psi=-(-\Delta)^{-1+\frac{\alpha}{2}}\omega,\\
    &\omega_{|t=0}=\omega_0.
  \end{alignedat} \right.
\end{equation}
Here $\nabla^\perp=(-\partial_2,\partial_1)$  while $\alpha \in [0,1)$ is a real parameter. The vector field $v$ is the flow velocity and the fractional Laplacian operator $(-\Delta)^{-1+\frac{\alpha}{2}}$ is of convolution type, defined by 
\begin{equation*}
-(-\Delta)^{-1+\frac{\alpha}{2}} \omega(x)=\int_{\R^2}K_\alpha(x-y)\omega(y)dy
\end{equation*}
 with
\begin{equation} \label{eqn:kalpha}
  K_\alpha(x):=\begin{cases}
    \displaystyle \frac{1}{2\pi}\ln|x|& \text{if }\alpha=0, \\[2ex]
    \displaystyle-\frac{C_\alpha}{2\pi}\frac{1}{|x|^\alpha} & \text{if }\alpha\in(0,1), \quad \text{with } \quad \displaystyle C_\alpha=\frac{\Gamma(\alpha/2)}{2^{1-\alpha}\Gamma(\frac{2-\alpha}{2})},
  \end{cases}
\end{equation}
where $\Gamma$ is the gamma function. This model was proposed by C\'ordoba et al.\@ \cite{C-F-M-R} as an interpolation between Euler equations and the surface quasi-geostrophic model, which correspond to $\alpha=0$ and $\alpha=1$, respectively. 

The  main purpose of this note is to show the existence of new families of periodic global solutions of \eqref{eqn:omega} in the vortex patch  setting, namely when $\omega(\,\cdot\,,t)$ is the characteristic function of finite collection of bounded domains. Such patterns are a special class of Yudovich solutions where $\omega(\,\cdot\,,t)$ is merely  bounded and integrable. Yudovich solutions are known to be unique and to exist globally in time in the case of Euler equations \cite{Y}, but for $\alpha>0$ the situation is more delicate because the velocity field $v$ is in general not Lipschitz. Nonetheless, when the initial datum has a patch structure, one can locally construct a unique solution which remains a patch. The motion of the boundary of the patch is governed by so-called contour dynamics equations; see \cite{Gancedo, Rodrigo}. It is worth mentioning that, while the boundary's regularity is globally preserved for $\alpha=0$ \cite{C,CB}, for $\alpha>0$ numerical evidence \cite{C-F-M-R} suggests singularity formation in finite time.
  
There are very few explicit solutions to the gSQG and Euler equations. The only known explicit simply-connected vortex patch solutions are the Rankine vortex, which is stationary, and the Kirchhoff ellipses \cite{K} for the Euler equation, which are rotating. Nevertheless, a family of uniformly  rotating patches with $m$-fold symmetry, called V-states, was numerically computed by Deem and Zabusky \cite{DZ}. Later, Burbea \cite{Bu} gave an analytical proof of their existence, based on a conformal mapping parametrization and local bifurcation theory. Recently, Burbea's branches of solutions were extended to global ones \cite{HMW}. The regularity and the convexity of the V-states  have been investigated in \cite{HMV,HMW,CCG1}. Similar research has been carried out for the gSQG  equations: The construction of simply connected V-states was established in \cite{HH,CCG}, and their boundary regularity was discussed in \cite{CCG1}. 

We point out that there is a large literature dealing with rotating vortex patches and related problems. For instance, we mention  the existence results of rotating patches close to Kirchhoff’s ellipses \cite{HM2,CCG1}, multiply-connected patches \cite{HMV2,DHMV,HM3, DHH,R,Gom}, patches in bounded domains \cite{DHHM}, non-trivial rotating smooth solutions \cite{CCG2} and rotating vortices with non-uniform densities \cite{GHS}. The radial symmetry properties of stationary and uniformly-rotating solutions was studied in a series of works \cite{F,Hmidi,gomez2019symmetry}.

All of the above analytical results treat connected patches. The first numerical works revealing the existence of translating symmetric pairs of simply connected patches for Euler equation are due to Deem and Zabusky \cite{DZ} and Pierrehumbert \cite{P}. Similar studies were performed  by Saffman and Szeto in \cite{SS} for the symmetric co-rotating vortex pairs and by Dritschel \cite{D} for asymmetric pairs. Later, Turkington gave in \cite{T} an analytical proof using variational arguments, where he considered an unbounded fluid domain with $N$ symmetrically arranged vortex patches rotating about the origin. Implementing the same approach, Keady \cite{Keady}   proved the existence  of translating pairs of symmetric patches and Wan \cite{Wan} studied the  existence and stability of desingularizations of a general system of rotating point vortices. Very recently, Godard-Cadillac, Gravejat and Smets \cite{GGS} extended Turkington's result to the gSQG equations, while Ao, D\'avila, Del Pino, Musso and Wei \cite{ADDMW} have obtained related families of smooth solutions via gluing techniques.  See \cite{N,No,O,S,New-fam} for additional references on multiply connected patches.

The variational arguments \cite{T,Keady,GGS} do not give much information about the shape of the vortex patches, or about the uniqueness of solutions. In \cite{Hmidi-Mateu}, however, Hmidi and Mateu gave a direct proof showing the existence of co-rotating and counter-rotating vortex pairs, using an elegant desingularization of the contour dynamics equations and an application of the implicit function theorem. The same technique was implemented  for the desingularization of the asymmetric pairs \cite{HH2}, K\'arm\'an street \cite{G-Kar} and the vortex polygon \cite{G}. See \cite{cqzz:doubly} for related results where point vortices are instead desingularized into doubly-connected patches. We mention that, using more sophisticated Nash--Moser techniques, G\'omez-Serrano, Park and Shi \cite{Gomez-Jaemin-Jia2021} have very recently constructed stationary configurations of multi-layered patches with finite kinetic energy. Also, Garc\'ia and Haziot \cite{gh:global} have combined ideas from \cite{Hmidi-Mateu} and \cite{HMW} to prove a global bifurcation result for co-rotating and counter-rotating pairs. 

In this note we show how the technique in \cite{Hmidi-Mateu} can be extended to arbitrary configurations of finitely-many point vortices. For a general configuration, the problem reduces via Lyapunov--Schmidt to a finite-dimensional nonlinear equation. Under an natural non-degeneracy assumption on the point vortex configuration alone, one can instead simply apply a modified version of the implicit function theorem. For highly symmetric configurations, we may also simplify the problem by reformulating it in spaces which take these symmetries into account.

\subsection{Statement of the general result}

Recall that the gSQG point vortex model for $N$ interacting vortices in the complex plane $\mathbb{C}$ is given by the Hamiltonian system
\begin{align}\label{ode-sys0}
\frac{d}{dt}z_{j}(t)=\frac{i\widehat{C}_\alpha }{2}\sum_{\substack{k=1, k\neq j}}^{N}\gamma_k   \frac{z_{j}(t)-z_{k}(t)}{|z_{j}(t)-z_{k}(t)|^{\alpha+2}}, 
\quad j=1,\ldots,N,
\end{align}
where $z_1(t),\ldots,z_N(t)$ are the point vortex locations, $\pi\gamma_1,\ldots, \pi\gamma_N\in \R\setminus\{0\}$ are the circulations and 
\begin{equation} \label{eqn:kalpha2}
\widehat{C}_\alpha :=\alpha  C_\alpha=\frac{2^{\alpha}\Gamma(1+\alpha/2)}{\Gamma(1-\alpha/2)}.
\end{equation} 
The case  $\alpha=0$ corresponds to the  classical point vortex Eulerian interaction. A general review about the $N$-vortex problem and vortex statics can be found in \cite{Aref} for the Newtonian interaction and \cite{Ros} for gSQG interactions.
We are concerned with periodic solutions for which the configuration of vortices is instantaneously moving as a rigid body, so that
\begin{align}
  \label{eqn:ode}
 \frac{d}{dt} z_j(t) = iU + i\Omega z_j(t),
\end{align}
where here $U \in \R$ is the constant linear velocity and $\Omega \in \R$ is the constant angular velocity. Such solutions are known as \emph{relative equilibria} or \emph{vortex crystals}.  Explicitly solving \eqref{eqn:ode}, we see that whenever $\Omega \ne 0$ we can shift coordinates so that $U=0$. Thus there is no loss of generality in restricting our attention to equilibria which are either \emph{rotating} with $\Omega \ne 0$ and $V=0$, \emph{translating} with $U \ne 0$ and $\Omega = 0$, or \emph{stationary} with $U=\Omega=0$.

Setting $z_j=x_j+iy_j:=z_j(0)$, \eqref{ode-sys0} reduces to the algebraic system  
\begin{align}\label{alg-sysP}
  \mathcal P_j^\alpha(\lambda)
  &=
  \Omega z_{j}+U-  \frac{\widehat{C}_\alpha}{2}\sum_{\substack{k=1, k\neq j}}^{N} \gamma_k \frac{z_j-z_{k}}{|z_j-z_{k}|^{\alpha+2}}=0,\quad j=1,\ldots,N,
\end{align}
where here $\lambda=(x_1,\ldots,x_N,y_1,\ldots,y_N, \gamma_1,\ldots,\gamma_N,\Omega, U)$. Taking real and imaginary parts, this defines a mapping $\mathcal P^\alpha(\lambda)$ with values in $\R^{2N}$. 
\begin{definition}\label{def:non-deg}
  We call a rigidly rotating or translating solution $\lambda^*$ of \eqref{alg-sysP} \emph{non-degenerate} if, after a reordering of the entries of $\lambda$, we can write 
  \begin{align}
    \label{non-deg-codim1}
    \lambda=(\lambda_1,\lambda_2)\quad \textnormal{with} \quad \lambda_1
    \in \R^{2N-1} \quad \textnormal{and}  \quad \codim \ran  D_{\lambda_1}\mathcal{P}^\alpha(\lambda^*) = 1.
  \end{align} 
  We call a stationary  solution $\lambda^*$ of \eqref{alg-sysP} (with $\Omega=U=0$) \emph{non-degenerate} if,
  after a reordering of the entries of $\lambda$, we can write 
  \begin{align}
    \label{non-deg-codim3}
    \lambda=(\lambda_1,\lambda_2)\quad \textnormal{with} \quad \lambda_1\in \R^{2N-3} \quad \textnormal{and}  \quad  \codim\ran D_{\lambda_1}\mathcal{P}^\alpha(\lambda^*) = 3.
  \end{align} 
\end{definition}
 
\begin{figure}
  \centering
  \includegraphics[scale=1.1]{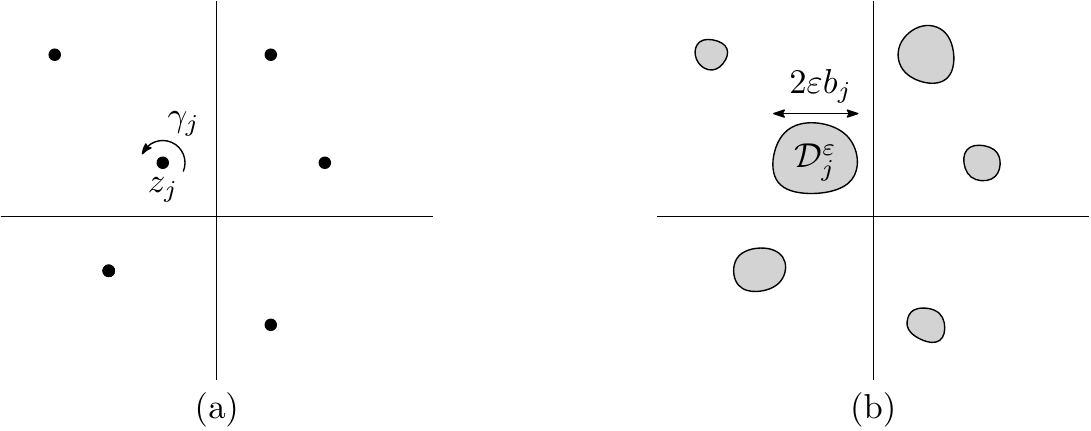}
  \caption{(a) A solution of the point vortex system \eqref{alg-sysP}. The vortices are located at the points $z_j(0):=z_j$ and have circulations $\pi\gamma_j$. (b) A desingularization into vortex patches. The vortex at $z_j$ has become a small nearly-circular patch $\mathcal D^\varepsilon_j$. To leading order in the small parameter $\varepsilon$, the radius of the patch is $\varepsilon b_j$, and the net circulation is $\pi\gamma_j$.
  \label{fig:general}
  }
\end{figure}

Informally stated, our first result is the following; see Theorems~\ref{prop:ift} and \ref{theorem-Phi-stationary} for a precise version.
\begin{theorem}\label{thm:general}
  Let $\alpha \in [0,1)$. Then any non-degenerate solution $\lambda$ of \eqref{alg-sysP} can be desingularized into a family of vortex patch equilibria depending smoothly on a small parameter $\varepsilon > 0$ measuring the size of the patches.
\end{theorem}
See Figure~\ref{fig:general} for an illustration.
\begin{remark}\label{rk:general}
  As we shall see later in Theorems~\ref{prop:ift} and \ref{theorem-Phi-stationary}, the vortex patches in Theorem~\ref{thm:general} are small $C^{1+\beta}$ perturbations of the unit disk, whose boundaries are given by conformal parametrizations which can be explicitly expanded to any order in the small parameter $\varepsilon$.
  In particular, the conformal parametrizations of the boundaries  $\phi_j\colon \mathbb{T}\to \partial \mathcal{O}_j^\varepsilon$ have the explicit Fourier asymptotic expansions
  $$
  \phi_j(\varepsilon,w)=w+(\varepsilon {b_j})^{\alpha+2} \Xi_\alpha\sum_{k=1, k\neq j}^N\frac{\gamma_k}{\gamma_j} \frac{(\overline{z_k}-\overline{z_j})^2}{ |z_k -z_j|^{\alpha+4}}\,\overline{w}+o(\varepsilon^{2+\alpha})  , \quad \Xi_\alpha:=\frac{(\alpha+2)\Gamma(1-\frac\alpha2)\Gamma(3-\frac\alpha2)}{4\Gamma(2-\alpha)}.
  $$
  Furthermore, with slight modifications in  the proof we can show that the boundary of each vortex patches  belongs to $C^{n+\beta}$ for any fixed $n\in \mathbb{N}$. The range of $\varepsilon$ would be not uniform with respect to $n$ but would shrink to zero as $n$ goes to infinity. However, using the approach developed in \cite{CCG1}, we expect that the conformal mappings possess holomorphic extensions outside of a small disc and thus that the boundaries are analytic. In the Euler case, the analyticity of sufficiently smooth patch boundaries could alternatively be proved using the same technique as in \cite{HMW}.
\end{remark}
\begin{remark}\label{rem:b}
  The leading-order ratios between the sizes of the patches can be specified a priori. Moreover, the range of $\varepsilon$ is uniform as some of these ratios are sent to zero, allowing us to recover solutions involving a combination of point vortices and vortex patches. See Remark~\ref{rem:bift} for more details.
\end{remark}
\begin{remark}
  While the proof is valid for $\alpha\in [0, 1)$, we expect that a similar result can be proved for 
  $\alpha\in [1, 2)$ using the spaces introduced in \cite{CCG}; see \cite{cqzz:alpha}.
\end{remark}

\subsection{Applications} \label{sec:applications}
There are many point vortex equilibria satisfying the non-degeneracy assumption in Theorem~\ref{thm:general}. We shall give several examples where this assumption can be easily checked and the resulting vortex patch solutions are, to the best of our knowledge, new. 

The most elementary solutions to \eqref{alg-sysP} are co-rotating and counter-rotating vortex pairs. A family of asymmetric co-rotating pairs is given by
\begin{equation}\label{v-pair-rot}
  \lambda^*:=(x_1^*,x_2^*,y_1^*,y_2^*;\gamma_1^*,\gamma_2^*;\Omega^*, U^*)=\Big(d,-\mathtt{c} d,0,0;\mathtt{c}\gamma,\gamma;\frac{\gamma \widehat{C}_\alpha}{2d^{\alpha+2}(1+\mathtt{c})^{\alpha+1}},0\Big),
\end{equation} 
where $\gamma\in\mathbb{R} \without \{0\}$, $d > 0$ and $0 < |\mathtt{c}|<1$; see Figure~\ref{fig:pairs}(a). In the time-dependent problem, the two vortices steadily rotate about the origin with angular velocity $\Omega^*$. Counter-rotating pairs given by
\begin{equation}\label{v-pair-trans}
\lambda^*:=(x_1^*,x_2^*,y_1^*,y_2^*;\gamma_1^*,\gamma_2^*;\Omega^*,U^*)=\Big(d,-d,0,0;-\gamma,\gamma;0,\frac{\gamma \widehat{C}_\alpha}{2^{\alpha+2}d^{\alpha+1}}\Big),
\end{equation}
instead steadily translate along the $y$-axis; see Figure~\ref{fig:pairs}(b). We also consider asymmetric stationary tripoles of the form
\begin{equation}\label{stationary-tripole}
  \begin{aligned}
    \lambda^*&:=(x_1^*,x_2^*,x_3^*,y_1^*,y_2^*,y_3^*;\gamma_1^*,\gamma_2^*,\gamma_3^*;\Omega^*,U^*)\\ &=\Big(1,0,-\mathtt{a},0,0,0;\gamma ,-\gamma\big(\tfrac{\mathtt{a}}{\mathtt{a}+1}\big)^{\alpha+1},\gamma \mathtt{a}^{\alpha+1};0,0\Big),
  \end{aligned}
\end{equation}
where $\mathtt{a} \in (0,1)$; see Figure~\ref{fig:pairs}(c). Note that all of these above configurations are invariant under reflections about the $x$-axis. Furthermore, horizontal translations of \eqref{v-pair-trans} and rotations of \eqref{v-pair-rot} are also solutions.
\begin{figure}
  \centering
  \includegraphics[scale=1.1]{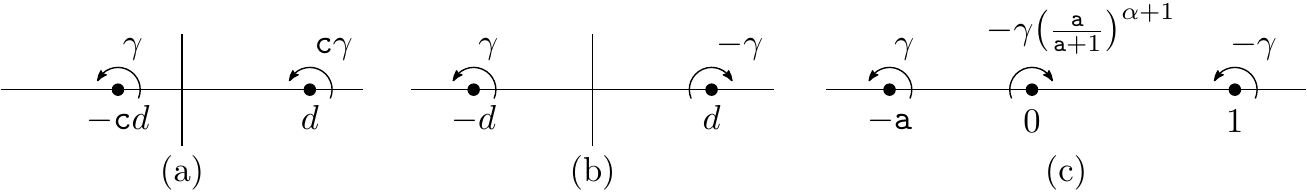}
  \caption{(a) An asymmetric co-rotating vortex pair given by \eqref{v-pair-rot}. (b) A symmetric counter-rotating vortex pair given by \eqref{v-pair-trans}. (c) A asymmetric point vortex tripole given by \eqref{stationary-tripole}.
  \label{fig:pairs}
  }
\end{figure}

As we shall see in Section \ref{sec:asym}, these configurations are non-degenerate in the sense of Definition~\ref{def:non-deg}. Using Theorem~\ref{thm:general}, they can therefore be desingularized into steady vortex patch equilibria. As no two vortices in \eqref{v-pair-rot} or \eqref{stationary-tripole} can be identified with one another, the same is true of the corresponding vortex patches. While the two vortices in \eqref{v-pair-trans} \emph{can} be identified, this symmetry is broken if we require the leading-order ratio between the sizes of the patches to be different from $1$. For the pairs, this extends the desingularization result of \cite{HH2}, obtained in the Eulerian case $\alpha=0$, to gSQG equations \eqref{eqn:omega} with $\alpha\in (0,1)$. To the best of our knowledge, the asymmetric tripole patch solutions are new both for the Euler and gSQG equations. Furthermore, this appears to be the first existence proof for stationary solutions to the gSQG equations involving multiple patches; see \cite{Gomez-Jaemin-Jia2021} for stationary solutions to the Euler equations with multiple multi-layered patches and \cite{Gom} for stationary doubly connected solutions to the gSQG equations.

\begin{theorem}\label{thm:informal-pair} 
  Let $\alpha\in[0,1)$ and $b_1,b_2,b_3\in (0,\infty)$, and let $\gamma,d,\mathtt{c},\mathtt{a}$ be as above.
  Then, the following results hold true.
  \begin{enumerate}[label=\rm(\roman*)]
 
  \item For any $\varepsilon>0$ sufficiently small, there are  two strictly convex domains $\mathcal{O}^\varepsilon_1,\mathcal{O}^\varepsilon_2$, $1$-fold symmetric,  $C^{1+\beta}$ perturbations of the unit disc,  and  real numbers $x_1(\varepsilon)=d+o(\varepsilon)$, $x_2(\varepsilon)=-\mathtt{c}d+o(\varepsilon)$ such that 
  \begin{equation*}
    \omega_{0}^\varepsilon=\frac{\gamma}{\varepsilon^2b_1^2}\chi_{\mathcal{D}_1^\varepsilon}+\frac{\mathtt{c}\gamma}{\varepsilon^2b_2^2}\chi_{\mathcal{D}_2^{\varepsilon}}\quad\textnormal{with}\quad  \mathcal{D}_1^\varepsilon:=\varepsilon b_1 \mathcal{O}_1^\varepsilon+x_1(\varepsilon),\quad  \mathcal{D}_2^{\varepsilon}:=\varepsilon b_2  \mathcal{O}_2^\varepsilon+x_2(\varepsilon),
  \end{equation*}
  generates a co-rotating vortex pair for \eqref{eqn:omega} with angular velocity $\Omega^*={\frac 12 \gamma \widehat{C}_\alpha d^{-\alpha-2}(1+\mathtt{c})^{-\alpha-1}}$.

  \item  For any  $\varepsilon>0$ sufficiently small, there are  two strictly convex domains $\mathcal{O}^\varepsilon_1,\mathcal{O}^\varepsilon_2$, $1$-fold symmetric,   $C^{1+\beta}$ perturbations of the unit disc, and  real numbers $x_2(\varepsilon)=-d+o(\varepsilon)$, $\gamma_1(\varepsilon)=-\gamma+o(\varepsilon)$ such that   \begin{equation*}
    \omega_{0}^\varepsilon=\frac{\gamma_1(\varepsilon)}{\varepsilon^2b_1^2}\chi_{\mathcal{D}_1^\varepsilon}+\frac{\gamma}{\varepsilon^2b_2^2}\chi_{\mathcal{D}_2^{\varepsilon}}\quad\textnormal{with}\quad  \mathcal{D}_1^\varepsilon:=\varepsilon b_1 \mathcal{O}_1^\varepsilon+d,\quad  \mathcal{D}_2^{\varepsilon}:=\varepsilon b_2  \mathcal{O}_2^\varepsilon+x_2(\varepsilon),
  \end{equation*}
    generates a counter-rotating vortex pair for \eqref{eqn:omega} with speed $U^*=\gamma \widehat{C}_\alpha 2^{-\alpha-2}d^{-\alpha-1}$.
  \item For any $\varepsilon>0$ sufficiently small, there are  three strictly convex domains $\mathcal{O}^\varepsilon_1,\mathcal{O}^\varepsilon_2,\mathcal{O}^\varepsilon_3$, $1$-fold symmetric,  $C^{1+\beta}$ perturbations of the unit disc,  and two  real numbers $x_3(\varepsilon)=-\mathtt{a}+o(\varepsilon)$, $\gamma_2(\varepsilon)=-\gamma\big(\tfrac{\mathtt{a}}{\mathtt{a}+1}\big)^{\alpha+1}+o(\varepsilon)$ such that 
  \begin{align*}
  &  \omega_{0}^\varepsilon=\frac{\gamma}{\varepsilon^2b_1^2}\chi_{\mathcal{D}_1^\varepsilon}+\frac{\gamma_2(\varepsilon)}{\varepsilon^2b_2^2}\chi_{\mathcal{D}_2^{\varepsilon}}+\frac{\gamma \mathtt{a}^{\alpha+1}}{\varepsilon^2b_3^2}\chi_{\mathcal{D}_3^\varepsilon}\\ & \textnormal{with}\quad  \mathcal{D}_1^\varepsilon:=\varepsilon b_1 \mathcal{O}_1^\varepsilon+1,\quad  \mathcal{D}_2^{\varepsilon}:=\varepsilon b_2  \mathcal{O}_2^\varepsilon,\quad \mathcal{D}_3^{\varepsilon}:=\varepsilon b_3  \mathcal{O}_3^\varepsilon+x_3(\varepsilon),
  \end{align*}
  generates a stationary vortex tripole for \eqref{eqn:omega}.

  \end{enumerate}
  \end{theorem}
\begin{remark}\label{rem:bpair}
  By sending some of $b_1,b_2,b_3$ to zero, we can recover configurations involving a mixture of vortex patches and point vortices; see Remark~\ref{rem:bift}.
\end{remark}
\begin{remark}
  The reflection symmetry property can be checked using the boundary equations, the uniqueness of the constructed curve of solutions and invariance under reflections about the $x$-axis of the point vortex configuration; see Section~\ref{sec:asym}.
\end{remark}

While the above examples are asymmetric and have only two or three vortices, there are other well-known point vortex equilibria which are highly symmetric and have many vortices. When seeking similarly symmetric desingularizations of such equilibria, it convenient to integrate these additional symmetries into the statement of the problem. In particular, the relevant non-degeneracy conditions on the point vortex equilibria can be much simpler to verify than Definition~\ref{def:non-deg}.

As a concrete example, we consider two concentric regular $m$-gons with a vortex at each vertex, and  assume that the vortices of a same polygon have the same  vorticity $\gamma_1 \in\R\setminus\{0\}$ or $\gamma_2 \in\R$. We place, in addition, a point vortex at the center of the regular $m$-gons with intensity $\gamma_0\in \R$.   
More specifically, we are concerned with  the system of point vortices
\begin{equation}\label{q00}
\omega_0^0(z)=\pi\gamma_0\delta_{0}(z)+\sum_{j=1}^{2}\pi\gamma_j\sum_{k=0}^{m-1}\delta_{z_{jk}(0)}(z),
\quad  z_{jk}(0):=
  \begin{cases}
    d_1e^{\frac{2k\pi i}{m}} &\, \textnormal{if }\, j=1, \\
    d_2e^{\frac{(2k+\vartheta)\pi i}{m}} &\, \textnormal{if }\, j=2,
    \, \vartheta\in\{0,1\}.
  \end{cases}
\end{equation}
 The value $\vartheta=0$ corresponds to the configuration where the vertices of the polygons are radially aligned with each other and  $\vartheta=1$ refers to the case where the vertices are out-of-phase by an angle $\pi/m$. If we assume that \eqref{q00} performs a perpetual  uniform rotation, then  
  we may easily check that the system of $2m+1$ equations in  \eqref{alg-sysP} can be reduced to a system of two real equations:
\begin{equation}\label{syst-pj2}
 \Omega d_j-\frac{\widehat{C}_\alpha}{2d_j^{1+\alpha}}\bigg({\gamma_0}+{\gamma_j}\sum_{k=1}^{m-1}\frac{1-e^{\frac{2k\pi i }{m}} }{\big|1 -e^{\frac{2k\pi i }{m}} \big|^{2+\alpha}}+\gamma_{3-j}\sum_{k=0}^{m-1}\frac{1 -e^{\frac{(2k+\vartheta)\pi i }{m}}\frac{d_{3-j}}{d_{j}} }{\big|1 -e^{\frac{(2k+\vartheta)\pi i }{m}}\frac{d_{3-j}}{d_{j}} \big|^{2+\alpha}}\bigg)=0, \, j=1,2.
\end{equation}
 Observe that the last system 
is linear in $\Omega$ and $\gamma_2$, and, thus,  under a simple non-degeneracy condition 
we can  explicitly solve the system \eqref{syst-pj2} for  $\Omega$ and $\gamma_2\neq 0$.

\begin{figure}
  \centering
  \includegraphics[scale=1.1]{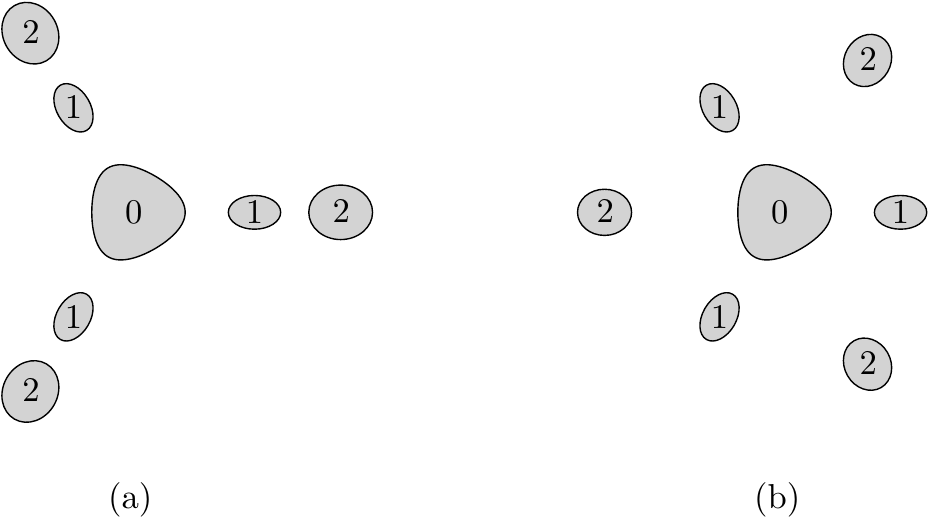}
  \caption{Sketch of the vortex patch solutions in Theorem~\ref{thm:informal-1polygon} with $m=3$. The point vortex of circulation $\pi\gamma_0$ at the origin has been desingularized into a patch with $m$-fold symmetry, while the vortices with circulations $\pi\gamma_1,\pi\gamma_2$ at the vertices of two regular $m$-gons have been desingularized into vortex patches with $1$-fold symmetry. (a) shows the case aligned case $\vartheta=0$, while (b) shows the staggered case $\vartheta=1$.
  \label{fig:poly}
  }
\end{figure}
For the sake of clarity we shall give an elementary statement about the desingularization of these configurations; for a complete statement see Theorem~\ref{thm:polygon}.
\begin{theorem}\label{thm:informal-1polygon} Let $\alpha \in [0,1)$, $\vartheta\in\{0,1\}$, $b_1,b_2\in (0,1)$, $\gamma_0,\gamma_1\in \R\setminus\{0\}$, $d_1,d_2\in  (0,\infty)$ and  let   $(\Omega^*,\gamma_2^*)\in (\R\setminus\{0\})^2$ be a solution of \eqref{syst-pj2} satisfying the non-degeneracy conditions \eqref{det-j-poly} and \eqref{non-deg}. 
 Then, for any $\varepsilon>0$ sufficiently small,  there are  three strictly convex domains $\mathcal{O}^\varepsilon_0, \mathcal{O}^\varepsilon_1, \mathcal{O}^\varepsilon_2$,   $C^{1+\beta}$ perturbations of the unit disc, and  a real number $\gamma_2=\gamma_2(\varepsilon)$  such that 
 \begin{equation}\label{theta0nest}
\omega_{0,\varepsilon}^\alpha=\frac{\gamma_0}{\varepsilon^2b_0^2}\chi_{ \varepsilon b_0 \mathcal{O}_0^\varepsilon}+\sum_{j=1}^{2}\frac{\gamma_j}{\varepsilon^2b_j^2}\sum_{k=0}^{m-1}\chi_{\mathcal{D}_{jk}^{\varepsilon}}
 \quad\textnormal{with} \quad \mathcal{D}_{jk}^\varepsilon :=
   \begin{cases}
e^{\frac{2k\pi i}{m}}(\varepsilon b_1  \mathcal{O}_2^\varepsilon+d_1) &\, \textnormal{if }\, j=1,   \\ e^{\frac{(2k+\vartheta)\pi i}{m}}(\varepsilon b_2  \mathcal{O}_2^\varepsilon+d_2)& \, \textnormal{if }\,  j=2,
   \end{cases}
\end{equation}
 generates a rotating solution for \eqref{eqn:omega} with some constant angular velocity $\Omega(\varepsilon)$.
Moreover $\mathcal{O}_0^\varepsilon$ is $m$-fold symmetric and $\mathcal{O}_1^\varepsilon$, $\mathcal{O}_2^\varepsilon$ are $1$-fold symmetric.
\end{theorem}
See Figure~\ref{fig:poly} for an illustration.

\begin{remark}
  The complete statement of Theorem~\ref{thm:informal-1polygon} given in Theorem~\ref{thm:polygon} explicitly computes the asymptotic behavior of the conformal parametrizations $\phi_j^\varepsilon\colon\mathbb{T}\to \partial \mathcal{O}_j^\varepsilon$.
\end{remark}

\begin{remark}\label{rem:bpoly}
  The parameter $\varepsilon>0$ can be chosen uniformly as any of the parameters $b_0$, $b_1$ and $b_2$ tend to $0$, and therefore we may recover the point vortex-vortex patch configurations discussed in \cite{Crowdy, New-fam}; see Remark~\ref{rem:bift}.
\end{remark}

\begin{remark}
  The proof can be easily adapted to   the rotating vortex polygon (with $\gamma_0=\gamma_2=0$ in  \eqref{theta0nest})  and which has been studied in \cite{ADDMW,GGS, T,G}. This remains equally true for  the body-centered  polygonal configurations  ($\gamma_2=0$), treated in \cite{Wan}, as well as for  the nested polygons without a central patch  ($\gamma_0= 0$). The latter solutions were first observed numerically in \cite{New-fam}.  
\end{remark}

\subsection{Idea of the proof}
We shall briefly explain the basic  ideas behind Theorem~\ref{thm:general} for the Euler equations in the rotating case; a similar strategy is followed for the gSQG equations and for traveling or stationary patches. We seek simply connected bounded domains $\mathcal{O}_j^\varepsilon$ such that the initial datum  
\begin{equation*}
\omega_{0}^\varepsilon(z)=\sum_{j=1}^N\frac{\gamma_j}{\varepsilon^2b _j^2}\chi_{\mathcal{D}_j^\varepsilon}(z) \quad\textnormal{with}\quad  \mathcal{D}_{j}^\varepsilon := \varepsilon  b_j\mathcal{O}_j^\varepsilon+z_j
\end{equation*}
performs a uniform rotation around the center of mass of the system, taken to be the origin, with an angular velocity $\Omega$. Here the parameters $b_j$ will allow us to specify the relative sizes of the patches; see Remark~\ref{rem:b}.
After moving to the rotating  frame, the boundaries of the system are subject to the following stationary system, see for instance \cite[Page 1896]{DHMV}, 
\begin{equation}\label{eq:rot}
\Re\big\{\gamma_{j}\big(\overline{v^\varepsilon(z)}+i\Omega\, \overline{z}\big)\vec{n}\big\}=0\quad \textnormal{for all}\quad z\in\displaystyle\partial \mathcal{D}_j^\varepsilon, 
\end{equation}
where $\vec{n}$ is the exterior unit normal vector to the boundary at the point $z$. In virtue of the Biot--Savart law and  Green’s theorem, we may write
\begin{align*}
\overline{v^\varepsilon(z)} &=\frac{1}{4\pi}\sum_{k=1}^N\frac{\gamma_k}{\varepsilon^2 b_k^2}\int_{\partial \mathcal{D}_k^\varepsilon}\frac{\overline{\xi}-\overline{z}}{\xi-z}d\xi
\end{align*}
for all $z\in \mathbb{C}$. 

Following Hmidi and Mateu \cite{Hmidi-Mateu}, we reformulate \eqref{eq:rot}  in terms of the conformal parametrizations of the boundaries  $\phi_j\colon\mathbb{T}\to \partial \mathcal{O}_j^\varepsilon$, which we assume have the form
$$
\phi_j(w)= w+\varepsilon b_j f_j(w).
$$
In other words, we shall look for domains $\mathcal{O}_j^\varepsilon$  which are small perturbations of the unit disc with an amplitude of order $\varepsilon b_j$. While the resulting problem initially appears to have terms which are singular in $\varepsilon$, as in \cite{Hmidi-Mateu} there is a cancellation --- essentially due to the symmetry of the disk --- which eliminates these terms. This leads to a nonlinear system
\begin{equation*}
\mathcal{G}_j(\varepsilon,f;\lambda)=0, \quad   j=1,\ldots,N
\end{equation*}
for the perturbations $f:=(f_1,\ldots,f_N)$ of the patch boundaries, where the nonlinear operator $\mathcal{G}=(\mathcal{G}_1,\ldots,\mathcal{G}_N)\colon(-\varepsilon_0,\varepsilon_0)\times B_0\times \Lambda \to \mathcal{W}_0$ is well-defined and of class $C^1$. 
Here $\Lambda$  in a small neighborhood of $\lambda^*$, which is the solution to the point vortex system \eqref{alg-sysP} with $\alpha=0$,  $B_0$ is the unit ball in $\mathcal{V}_0$ and 
\begin{equation*}
\begin{split}
&\mathcal{V}_0 :=\underbrace{V_0\times\cdots \times V_0}_{\text{$N$ times}}, \quad
\mathcal{W}_0 :=\underbrace{W_0\times\cdots \times W_0}_{\text{$N$ times}}, \\
&V_0:=
        \Big\{f\in  C^{1+\beta}(\mathbb{T}): f(w)=\sum_{n\geq 1}
{a_n}\,\overline{w}^n,\;   a_n\in \mathbb{C}\Big\},\\
 &  W_0:=
        \Big\{g\in  C^{\beta}(\mathbb{T}): g(w)= \sum_{n\neq 0}
c_n\,{w}^n, \; c_{-n}=\overline{c_n}\in \mathbb{C}\Big\}
\end{split}
\end{equation*}
for some fixed H\"older exponent $\beta \in (0,1)$.
 Moreover, for $(\varepsilon,f;\lambda)=(0,0;\lambda)$ we find
\begin{equation} \label{eq:pvconnection}
  \mathcal{G}_j(0,0;\lambda)(w)=\Im \big\{ \mathcal P_j^0(\lambda) \overline w \big\},
\end{equation}
so that $\mathcal G(0,0;\lambda) = 0$ is equivalent to \eqref{alg-sysP}.

Intending to apply the implicit function theorem, we next linearize $\mathcal{G}$ about the point vortex solution $(0,0;\lambda^*)$. The linearized operator with respect to the patch boundaries $f$ is 
\begin{equation*}
  D_{f}\mathcal{G}_j(0,0;\lambda^*)h(w)=
  {\gamma_j}\, \Im \big\{h'_j(w)\big\},
\end{equation*}
which has a trivial kernel and a range with finite codimension $2N$. To deal with this deficiency, we also linearize with respect to the point vortex parameters $\lambda$ using \eqref{eq:pvconnection}. Assuming that $\lambda^*$ is non-degenerate in the sense of Definition~\ref{def:non-deg}, we deduce that the linearized operator $D_{(f;\lambda_1)}\mathcal G(0,0;\lambda^*)$ has trivial kernel and a range with codimension 1. As this final deficiency is caused by a nonlinear identity 
\begin{align*}
  \sum_{j=1}^N\frac{\gamma_j}{i\pi}\int_{\mathbb{T}}\mathcal{G}^\alpha_j(\varepsilon,f;\lambda)(w) \Big(\varepsilon b_j|w+\varepsilon b_j f_j(w)|^2+\varepsilon b_j\Re\big[\overline{z_j}f_j(w)\big]+\Re[\overline{z_j}w]\Big)\overline{w}dw=0
\end{align*}
satisfied by the functional $\mathcal G$, we can conclude by applying a modified version of the implicit function theorem (Lemma~\ref{abstract-lemma}).

\subsection{Notation}
Let us end this part by summarizing some notation to be used in the paper. We will denote the unit disc by $\mathbb{D}$ and its boundary by $\mathbb{T}$. For continuous functions $f\colon\mathbb{T}\to \mathbb{C}$ we introduce the notation
$$
\fint_{\mathbb{T}} f(\tau)\, d\tau:=\frac{1}{2\pi i}\int_{\mathbb{T}} f(\tau)\, d\tau,
$$
where $d\tau$ stands for complex integration. For any $x\in \mathbb{R}$ and $n\in\mathbb{N}$, we use the notation $(x)_n$ to denote the Pochhammer symbol defined by
$$
(x)_n:=
\begin{cases}
  1 &\, \textnormal{if }\, n=0,        \\
  x(x+1)\cdots (x+n-1)  &\, \textnormal{if }\, n\geq 1.
\end{cases}
$$
 Finally,  we use the notation $\delta_{ij}$ to denote the Kronecker delta  defined by
$$
\delta_{ij}:= 
\begin{cases}
  1 &\, \textnormal{if }\, i=j,        \\
  0 &\, \textnormal{if }\, i\neq j.
\end{cases}
$$

\subsection{Outline of the paper} 
In Section~\ref{sec:general} we consider completely general vortex equilibria, without any symmetry assumptions, show how the problem is desingularized and prove Theorem~\ref{thm:general}.  Section~\ref{sec:asym} is devoted to the proof of  Theorem~\ref{thm:informal-pair}. Finally, in Section~\ref{sec:nest-poly}, we prove Theorem~\ref{thm:informal-1polygon} by imposing suitable rotation and reflection symmetries.

\section{Desingularization of  general vortex equilibria}\label{sec:general}
In this section we consider a general configuration of finitely many point vortices in uniform rotation or translation, and show how these vortices can be desingularized into small vortex patches. Using the approach developed in \cite{Hmidi-Mateu}, we first write down the  contour dynamics  equations governing  the $N$ steady vortex patches, and then find the suitable function spaces where the problem is well-posed. Finally, we prove Theorem~\ref{thm:general} using a simple extension of the implicit function theorem.

\smallskip

Consider $N$  bounded simply connected domains  $\mathcal{O}_j^\varepsilon,$ $j=1,\ldots,N$,   containing  the origin and contained in the ball $B(0,2)$. Given $b_j\in (0,\infty)$,  $z_j\in \mathbb{C}$ and  $\varepsilon\in(0,\varepsilon_0)$, we define the domains
\begin{equation}\label{Dj}
\displaystyle \mathcal{D}_{j}^\varepsilon:= \varepsilon   b_j\mathcal{O}_j^\varepsilon+z_j,
\end{equation}
where $\varepsilon_0 > 0$ is chosen small enough that the sets $\mathcal{D}_{j}^\varepsilon$ are pairwise disjoint,
\begin{equation}\label{intersection}
\overline{\mathcal{D}_{j}^\varepsilon}\cap \overline{\mathcal{D}_{k}^\varepsilon}=\emptyset , \quad j\neq k.
\end{equation}
Consider the initial vorticity  
\begin{equation}\label{intial-vort}
\omega_{0}^\varepsilon(z)=\frac{1}{\varepsilon^2}\sum_{j=1}^N\frac{\gamma_j}{ b_j^2}\chi_{\mathcal{D}_j^\varepsilon}(z).
\end{equation}
Note that, if $\varepsilon\to 0$ and $|\mathcal{O}_j^\varepsilon|\to|\mathbb{D}|$ in \eqref{intial-vort}, we find the point vortex distribution 
\begin{equation*}
\omega_{0}^0(z)=\pi\sum_{j=1}^N{\gamma_j}\delta_{z_j}(z), 
\end{equation*}
whose evolution is described by \eqref{ode-sys0}. 

\subsection{Integral identities for the stream function} \label{sec:identities}
We shall give in this subsection some identities related to the (non-relative) stream function,
\begin{equation}\label{exp psi}
  \forall z\in \mathbb{C}\mapsto\psi^\varepsilon(z)=\frac{1}{\varepsilon^2}\sum_{k=1}^N\frac{\gamma_k}{b_k^2}\int_{\mathcal{D}_k^\varepsilon}K_\alpha(z-\xi)dA(\xi),
\end{equation}
associated to the vortex patch \eqref{intial-vort}, where $K_\alpha$ is defined in \eqref{eqn:kalpha}.  
As we shall see in the following subsections, these identities will be useful to explain the degeneracy  of the functional defining the V-states. While they can be derived using the symmetries and variational structure of the problem, we shall give here a simple proof using the structure of the kernel $K_\alpha$. 
\begin{lemma}\label{lem:idens}
  For all $\varepsilon \in (0,\varepsilon_0)$, the stream function 
  \eqref{exp psi} satisfies the following identities:
  \begin{enumerate}[label=\rm(\roman*)]
  \item $\displaystyle\sum_{j=1}^N\frac{\gamma_j}{\varepsilon^2 b_j^2 }\int_{\partial \mathcal{D}_j^\varepsilon}\psi^\varepsilon(z)dz=0$,
  \item $\displaystyle\Re\Big[\sum_{j=1}^N\frac{\gamma_j}{\varepsilon^2 b_j^2 }\int_{\partial \mathcal{D}_j^\varepsilon}\overline{z}\psi^\varepsilon(z)dz\Big]=0$.
  \end{enumerate}
\end{lemma}
\begin{proof}
Applying the complex version of Green's theorem, 
\begin{equation}\label{Stokes Thm}
	2i\int_{D}\partial_{\overline{\xi}}f(\xi,\overline{\xi})dA(\xi)=\int_{\partial D}f(\xi,\overline{\xi})d\xi, 
\end{equation}
to \eqref{exp psi} we find
$$
\psi^\varepsilon(z)=\frac{1}{2i}\sum_{k=1}^N\frac{\gamma_k}{\varepsilon^2 b_k^2}\int_{\partial \mathcal{D}_k^\varepsilon} \widehat K_\alpha(z-\xi)d\xi\quad{\rm with}\quad  \widehat K_\alpha(\xi)=\begin{cases}
    \displaystyle \frac{1}{2\pi}\overline{\xi}\big(\ln|\xi|-\tfrac12\big)& \text{if }\alpha=0, \\[2ex]
    \displaystyle-\frac{C_\alpha}{2\pi(1-\frac\alpha2)}\frac{\overline{\xi}}{|\xi|^\alpha} & \text{if }\alpha\in(0,1). 
  \end{cases}
$$
It follows that
$$
\sum_{j=1}^N\frac{\gamma_j}{\varepsilon^2 b_j^2 }\int_{\partial \mathcal{D}_j^\varepsilon}\psi^\varepsilon(z)dz=\frac{1}{2i}\sum_{k=1}^N\sum_{j=1}^N \frac{\gamma_k\gamma_j}{\varepsilon^4 b_j^2b_k^2 }\int_{\partial \mathcal{D}_k^\varepsilon}\int_{\partial \mathcal{D}_j^\varepsilon} \widehat  K_\alpha(z-\xi)d\xi dz.
$$
The integrand $\widehat  K_\alpha(z-\xi)$ changes sign when the roles of $z$ and $\xi$ are reversed, and so {\rm(i)} follows. 
 
Next we apply formula \eqref{Stokes Thm} to $\bar z \psi^\varepsilon$, which yields
\begin{equation*}
\int_{\partial \mathcal{D}_j^\varepsilon}\overline{z}\psi^\varepsilon(z)dz =2i\int_{ \mathcal{D}_j^\varepsilon}\psi^\varepsilon(z)dA(z)+2i\int_{ \mathcal{D}_j^\varepsilon}\overline{z}\partial_{\overline z}\psi^\varepsilon(z)dA(z).
\end{equation*}
Since $\psi^\varepsilon(z)$ is real, the last identity implies
\begin{equation}\label{id0}
\Re\Big[\sum_{j=1}^N\frac{\gamma_j}{\varepsilon^2 b_j^2 }\int_{\partial \mathcal{D}_j^\varepsilon}\overline{z}\psi^\varepsilon(z)dz\Big]=2\Im\Big[\sum_{j=1}^N\frac{\gamma_j}{\varepsilon^2 b_j^2 }\int_{ \mathcal{D}_j^\varepsilon}z\partial_z\psi^\varepsilon(z)dA(z)\Big].
\end{equation}
Differentiating  \eqref{exp psi} with respect to $z$ we get 
\begin{align*}
\partial_z\psi^\varepsilon(z)&=\sum_{k=1}^N \frac{\gamma_k}{\varepsilon^2 b_k^2 }\int_{\mathcal{D}_k^\varepsilon}\widetilde K_\alpha(z-\xi)dA(z)\quad{\rm with}\quad  \widetilde K_\alpha(\xi)=\begin{cases}
    \displaystyle \frac{1}{4\pi}\frac{1}{\xi}& \text{if }\alpha=0, \\[2ex]
    \displaystyle\frac{\alpha C_\alpha}{4\pi}\frac{\overline{\xi}}{|\xi|^{\alpha+2}} & \text{if }\alpha\in(0,1). 
  \end{cases}
\end{align*}
Inserting this into \eqref{id0} and interchanging the roles of $\xi$ and  $z$ we obtain
\begin{align*}
  \Re\Big[\sum_{j=1}^N\frac{\gamma_j}{\varepsilon^2 b_j^2 }\int_{\partial \mathcal{D}_j^\varepsilon}\overline{z}\psi^\varepsilon(z)dz\Big]&=2\Im\bigg[\sum_{k=1}^N\sum_{j=1}^N \frac{\gamma_k\gamma_j}{\varepsilon^4 b_j^2b_k^2 }\int_{\mathcal{D}_k^\varepsilon}\int_{ \mathcal{D}_j^\varepsilon}z\widetilde K_\alpha(z-\xi)dA(\xi) dA(z)\bigg]\\ &=\Im\bigg[\sum_{k=1}^N\sum_{j=1}^N \frac{\gamma_k\gamma_j}{\varepsilon^4 b_j^2b_k^2 }\int_{\mathcal{D}_k^\varepsilon}\int_{ \mathcal{D}_j^\varepsilon}\big(z\widetilde K_\alpha(z-\xi)+\xi  \widetilde K_\alpha(\xi-z)\big)dA(\xi) dA(z)\bigg].
\end{align*}
The desired identity {\rm(ii)} now follows from the fact that the integrand
$$
z  \widetilde K_\alpha(z-\xi)+\xi  \widetilde K_\alpha(\xi-z)=\begin{cases}
    \displaystyle \frac{1}{4\pi}& \text{if }\alpha=0, \\[2ex]
    \displaystyle\frac{\alpha C_\alpha}{4\pi}\frac{1}{|z-\xi|^{\alpha}} & \text{if }\alpha\in(0,1)
  \end{cases}
$$
is purely real, and the proof is complete.
\end{proof}

\subsection{Boundary equations}\label{sec:formulation}

First suppose that $\omega_{0}^\varepsilon$ gives rise to  $N$ rotating  patches  for the model \eqref{eqn:omega}  about the centroid of the system, assumed to be the origin, 
 with an angular velocity $\Omega$.  More precisely, we are looking for a  solution $\omega^\varepsilon(t)$ of \eqref{eqn:omega} of  the form
$$
\omega^\varepsilon(t,z)=\omega_{0}^\varepsilon\big(e^{-it\Omega}z\big).
$$
Inserting this expression into \eqref{eqn:omega} we obtain
$$
\big(v^\varepsilon(z)-i\Omega z\big)\cdot \nabla \omega_{0}^\varepsilon(z)=0,
$$
with $v^\varepsilon$ is the velocity field associated to $\omega_{0}^\varepsilon$. Using the patch structure, we conclude that \cite[Page 1896]{DHMV}
\begin{equation}
 \label{eqn:boundry00-rot}
\Re\big\{\gamma_{j}\big(\overline{v^\varepsilon(z)}+i\Omega\, \overline{z}\big)\vec{n}\big\}=0\quad \textnormal{for all}\quad z\in\displaystyle\partial \mathcal{D}_j^\varepsilon, \quad j=1,\ldots,N,
\end{equation}
where $\vec{n}$ is the exterior unit normal vector to the boundary at the point $z$.

If instead $\omega_0^\varepsilon$ translates vertically with uniform velocity $U$, that is
$$
\omega^\varepsilon(t,z)=\omega_{0}^\varepsilon\big(z-itU\big),
$$
the analogue of  \eqref{eqn:boundry00-rot} is 
\begin{equation}\label{eqn:boundry00-trans}
\Re\big\{\gamma_{j}\big(\overline{v^\varepsilon(z)}+iU\big)\vec{n}\big\}=0\quad \textnormal{for all}\quad z\in\displaystyle\partial \mathcal{D}_j^\varepsilon, \quad j=1,\ldots, N. 
\end{equation}
For the sake of  abbreviation and simplicity we shall unify \eqref{eqn:boundry00-rot} and \eqref{eqn:boundry00-trans} as follows 
\begin{equation}
 \label{eqn:boundry00}
\Re\big\{\gamma_{j}\big(\overline{v^\varepsilon(z)}+iU+i\Omega\, \overline{z}\big)\vec{n}\big\}=0\quad \textnormal{for all}\quad z\in\displaystyle\partial \mathcal{D}_j^\varepsilon, \quad j=1,\ldots, N, 
\end{equation}
and assume that either  $\Omega$ or $U$ vanishes.
 \subsubsection{Euler equation}  \label{subsec:euler}
 In view of the Biot--Savart law one has
\begin{equation}\label{Biot-savart}
  \overline{v^\varepsilon(z)} =-2i\partial_z \psi^\varepsilon(z) = -\frac{i}{2\pi}\sum_{k=1}^N\frac{\gamma_k}{\varepsilon^2  b_k^2}\int_{\mathcal{D}_k^\varepsilon}\frac{dA(\zeta)}{z-\zeta}
\end{equation}
for all $z\in\mathbb{C}$. By the complex form of Green's theorem \eqref{Stokes Thm},
we may replace the integral over $\mathcal{D}_k$ in \eqref{Biot-savart} with an integral along $\partial \mathcal{D}_k$:
\begin{align*}
\overline{v^\varepsilon(z)} &=\frac{i}{2}\sum_{k=1}^N\frac{\gamma_k}{\varepsilon^2  b_k^2}\fint_{\partial \mathcal{D}_k^\varepsilon}\frac{\overline{\xi}-\overline{z}}{\xi-z}d\xi.
\end{align*}
Inserting the  last identity  into   \eqref{eqn:boundry00} leads to
\begin{align}\label{eq2E0}
& \gamma_j\, \Re\big\{\big(\Omega\,\overline{z}+U+V^\varepsilon(z)\big)z'\big\}=0,\quad \forall z\in \partial \mathcal{D}_j^\varepsilon, \quad j=1,\ldots, N,
\end{align}
where  $z'$ denotes a tangent vector to the boundary  at the point $z$ and
$$
V^\varepsilon(z):=\sum_{k=1}^N\frac{\gamma_k}{2\varepsilon^2  b_k^2}\fint_{\partial \mathcal{D}_k^\varepsilon}\frac{\overline{\xi}-\overline{z}}{\xi-z}d\xi.
$$  
 In view of \eqref{Dj}, a suitable change of variables gives
\begin{align*}
V^\varepsilon(z)=\sum_{k=1}^N\frac{\gamma_k}{2\varepsilon  b_k}\fint_{\partial \mathcal{O}_k^\varepsilon}\frac{\varepsilon  b_k\overline{\xi}+{z_k}-\overline{z}}{\varepsilon  b_k\xi+z_k-z}d\xi.
\end{align*} 
Observe, from \eqref{intersection}, that for any $j\neq k$ and $z\in \partial \mathcal{O}_j^\varepsilon$ one has  
$$\varepsilon  b_j z+z_j+z_k\not\in \varepsilon  b_k\overline{{\mathcal{O}}_k^\varepsilon}.
$$
Thus, by the residue theorem, for every  $z\in \partial \mathcal{O}_j^\varepsilon$, we may write
\begin{align*}
 V^\varepsilon\big(\varepsilon  b_j{z}+z_j\big)
&=\sum_{k=1}^N\frac{\gamma_k}{2\varepsilon  b_k }\fint_{\partial \mathcal{O}_k^\varepsilon}\frac{\varepsilon  b_k\overline{\xi}-{z_k}-\varepsilon  b_j\overline{z}-{z_j}}{\varepsilon  b_k\xi+z_k-\varepsilon  b_j{z}-{z_j}}d\xi\\
& =\frac{\gamma_j }{2\varepsilon  b_j }\fint_{\partial \mathcal{O}_j^\varepsilon}\frac{\overline{\xi}-\overline{z}}{\xi- z}d\xi+\sum_{\substack{k=1\\ k\neq j}}^{N}\frac{\gamma_k}{2}\fint_{\partial \mathcal{O}_k^\varepsilon}\frac{\overline{\xi}}{\varepsilon  b_k\xi+z_k-\varepsilon  b_j{z}-{z_j}}d\xi. 
\end{align*}
Replacing $z$ by $\varepsilon  b_j{z}+z_j$  in \eqref{eq2E0}
and using the  last identity we get
\begin{equation}\label{boundry D11}
\begin{split}
\gamma_j\,\Re\Big\{& \Big(\Omega \big(\varepsilon  b_j\overline{z}+\overline{z_j}\big)+U+\frac{\gamma_j }{2\varepsilon  b_j }\fint_{\partial \mathcal{O}_j^\varepsilon}\frac{\overline{\xi}-\overline{z}}{\xi- z}d\xi\\&+\sum_{\substack{k=1\\ k\neq j}}^{N}\frac{\gamma_k}{2}\fint_{\partial \mathcal{O}_k^\varepsilon}\frac{\overline{\xi}}{ \varepsilon  b_k\xi+z_k-\varepsilon  b_j{z}-{z_j}}d\xi \Big)z'\Big\}=0, \quad \forall z\in \partial \mathcal{O}_j^\varepsilon.
\end{split}
\end{equation}

We shall look  for  domains $\mathcal{O}_j^\varepsilon$,  which are perturbations of the unit disc with an amplitude of order $\varepsilon  b_j$. More precisely, we shall 
consider $\phi_j\colon\mathbb{C}\backslash \overline{\mathbb{D}}\to \mathbb{C}\backslash \overline{\mathcal{O}_j^\varepsilon}$  the unique conformal map    with the expansion 
\begin{align}\label{conf0}
  \phi_j (w)= w+\varepsilon  b_j f_j(w) \quad\textnormal{with}\quad f_j(w)=\sum_{m=1}^\infty\frac{a_m^j}{w^{m}},\quad a_m^j\in \mathbb{C}.
\end{align}
By the Kellogg--Warschawski theorem \cite[Theorem 3.6]{P},  since the boundary  $\partial \mathcal{O}_j^\varepsilon$ is assumed to be a smooth Jordan Curve,  $\phi_j$
extends to a smooth mapping $\mathbb{C}\backslash \mathbb{D}\to \mathbb{C}\backslash \mathcal{O}_j^\varepsilon$, and its trace, that we shall also denote by $\phi_j$, is a smooth parametrization of $\partial\mathcal{O}_j^\varepsilon$.
 Thus, making the change of variable $z=\phi_j(w)$, $ z' = iw\phi_j'(w)$  in  \eqref{boundry D11},   we obtain
\begin{align} \label{sys1}
  \Im \bigg\{\gamma_j
  \bigg(&
  \Omega\big(\varepsilon  b_j \overline{\phi_j(w)} +\overline{ z_j}\big)+U+ \frac{\gamma_j}{2\varepsilon  b_j} \fint_\T 
  \frac{\overline{\phi_j(\tau)}-\overline{\phi_j(w)}}{\phi_j(\tau)-\phi_j(w)}\, \phi_j'(\tau)d\tau 
 \notag \\&   +\sum_{k=1}^N\frac{\gamma_k}{2}  \fint_\T \frac{
  \overline{\phi_k(\tau)}\phi_k'(\tau)}
  { 
\varepsilon  b_k \phi_k(\tau)+z_k - \varepsilon  b_j \phi_j(w) - z_j}d\tau \bigg)
  w\phi_j'(w)
  \bigg\}=0
\end{align}
for any $w\in\mathbb{T}$ and  $j=1,\ldots,N$. 

To desingularize this system in $\varepsilon$, we follow the ideas of \cite{Hmidi-Mateu} and write, by virtue of \eqref{conf0},
\begin{align*}
\frac{1}{2\varepsilon  b_j}\fint_\T
  \frac{\overline{\phi_j(\tau)}-\overline{\phi_j(w)}}{\phi_j(\tau)-\phi_j(w)}\, \phi_j'(\tau)d\tau & =\frac{1}{2\varepsilon  b_j}\fint_{\mathbb{T}}\frac{\overline{w}-\overline{\tau}+\varepsilon  b_j\big(\overline{f_j(\tau)}-\overline{f_j(w)}\big)}{w-\tau+\varepsilon  b_j\big(f_j(\tau)-f_j(w)\big)}\Big[1+\varepsilon  b_j f_j'(\tau)\Big]d\tau
\\ &=\frac12\fint_{\mathbb{T}}\frac{\overline{w}-\overline{\tau}+\varepsilon  b_j\big(\overline{f_j(\tau)}-\overline{f_j(w)}\big)}{w-\tau+\varepsilon  b_j\big(f_j(\tau)-f_j(w)\big)}f_j'(\tau)d\tau-\frac{1}{2\varepsilon  b_j}\fint_{\mathbb{T}}\frac{\overline{w}-\overline{\tau}}{w-\tau}d\tau
\\ &\qquad  +\frac12\fint_{\mathbb{T}}\frac{(w-\tau)\big(\overline{f_j(\tau)}-\overline{f_j(w)}\big)-(\overline{w}-\overline{\tau})\big(f_j(\tau)-f_j(w)\big)}{(w-\tau)\big(w-\tau+\varepsilon  b_j\big(f_j(\tau)-f_j(w)\big)\big)}d\tau
\\& := \overline{\mathcal{I}^0}[\varepsilon, f_j](w)-\frac{1}{2\varepsilon  b_j}\fint_{\mathbb{T}}\frac{\overline{w}-\overline{\tau}}{w-\tau}d\tau.
\end{align*}
From the obvious identity
$$
\fint_{\mathbb{T}}\frac{\overline{w}-\overline{\tau}}{w-\tau}d\tau=\overline{w},
$$
  the expression of $\phi_j$ in \eqref{conf0}, and  the symmetry of the disc we  can get rid of  the singular term  from the full nonlinearity,
\begin{align*}
\Im\Big\{\frac{1}{2\varepsilon  b_j}\Big(\fint_\T
  \frac{\overline{\phi_j(\tau)}-\overline{\phi_j(w)}}{\phi_j(\tau)-\phi_j(w)}\, \phi_j'(\tau)d\tau\Big) w\phi'_j(w)\Big\}&=\Im\Big\{\overline{\mathcal{I}^0}[\varepsilon, f_j](w)w\Big(1+\varepsilon f'_j(w)\Big)-\frac12 f'_j(w)\Big\}.
\end{align*}
Inserting the last equation into \eqref{sys1}, we conclude that
\begin{equation} \label{rotn+100}
\begin{split}
\gamma_j\mathcal{G}_j^0(\varepsilon,f;\lambda)(w):= -\gamma_j&\Im \Big\{
  \Big(
  \Omega\big(\varepsilon  b_j\overline{w}+\varepsilon^2  b_j^2 \overline{f_j(w)}- \overline{z_j}\big)+U+{\gamma_j}\, \overline{\mathcal{I}^0}[\varepsilon, f_j](w)
  \\&\quad +\sum_{k=1,k\neq j}^N{\gamma_k} \ \overline{\mathcal{J}^0_{kj}}[\varepsilon,f_k, f_j](w)
    \Big)
  w\big(1+\varepsilon b_j f'_j(w)\big)
  -\frac{\gamma_j}{2}  f'_j(w)\Big\}=0
  \end{split}
\end{equation}
for all $w\in \T$ and $j=1,\ldots,N$, where $f=(f_1,\ldots,f_N)$, 
\begin{align}\label{lambda}
  \lambda=(x_1,\ldots,x_N,y_1,\ldots,y_N, \gamma_1,\ldots,\gamma_N,\Omega, U)
\end{align}
are the point vortex parameters and 
\begin{align}\label{I00}
\overline{\mathcal{I}^0}[\varepsilon, f_j](w)&:=\frac12 \fint_{\mathbb{T}}\frac{\overline{w}-\overline{\tau}+\varepsilon  b_j\big(\overline{f_j({\tau})}-\overline{f_j({w})}\big)}{w-\tau+\varepsilon  b_j\big(f_j(\tau)-f_j(w)\big)}f_j'(\tau)d\tau\notag
\\ &\qquad+\frac12\fint_{\mathbb{T}}\frac{2i\Im\big\{(w-\tau)\big(\overline{f_j({\tau})}-\overline{f_j(\overline{w})}\big)\big\}}{(w-\tau)\big(w-\tau+\varepsilon  b_j f_j(\tau)-\varepsilon  b_j f_j(w)\big)}d\tau,\\
\overline{\mathcal{J}_{kj}^0}[\varepsilon,f_k,f_j](w)&:=\frac12\fint_\T \frac{
  \big(\overline{\tau}+\varepsilon  b_k \overline{f_k({\tau})}\big) \big(1+\varepsilon  b_k f_k'({\tau})\big)}
  {\varepsilon  b_k  {\tau}+\varepsilon^2  b_k^2 f_k({\tau}) +z_k-\varepsilon  b_j\big({w}+\varepsilon  b_j f_j({w})\big)-z_j}\, d\tau.\label{J00}
\end{align}
Here we have  used the notation $\overline{\mathcal{I}^0}$ and $\overline{\mathcal{J}_{kj}^0}$  with a complex conjugate in order to unify the notation with the functions  ${\mathcal{I}^\alpha}$,  ${\mathcal{J}_{kj}^\alpha}$ and ${\mathcal{G}_j^\alpha}$  that we shall introduce in next subsection for the gSQG equations. Furthermore, both sides of \eqref{rotn+100} are multiplied by $\gamma_j$ to ensure that the system is valid even if we set some of the $\gamma_j$  equal to zero.

\subsubsection{gSQG equations} \label{subsec:sqg}
 The velocity can be recovered from the boundary as follows
\begin{equation}\label{vgsqg}
{v^\varepsilon(z)} =\frac{C_\alpha}{2\pi}\sum_{k=1}^N\frac{\gamma_k}{\varepsilon^2  b_k^2}\int_{\partial\mathcal{D}_k^\varepsilon}\frac{d\xi}{|z-\xi|^\alpha}
\end{equation}
for all $z\in\mathbb{C}$, see for instance \cite{HH}. In view of \eqref{Dj}, suitable change of variables gives
\begin{align*}
v^\varepsilon(z)&=i \sum_{k=1}^N\frac{\gamma_k C_\alpha}{\varepsilon  b_k}\fint_{\partial \mathcal{O}_k^\varepsilon}\frac{1}{|\varepsilon  b_k\xi+z_k-z|^\alpha}d\xi.
\end{align*}
Inserting the  last identity  into \eqref{eqn:boundry00} and taking a complex conjugate inside the real part   leads to
\begin{align*}
&\Re\Big\{\gamma_j\Big(\Omega{z}+U-\sum_{k=1}^N \frac{\gamma_k C_\alpha}{\varepsilon  b_k}\fint_{\partial \mathcal{O}_k^\varepsilon}\frac{1}{|\varepsilon  b_k\xi+z_k-z|^\alpha}d\xi\Big)\overline{z'}\Big\}=0\quad \forall z\in \partial \mathcal{D}_j^\varepsilon, \quad j=1,\ldots,N, 
\end{align*}
where  $z'$ denotes a tangent vector to the boundary  at the point $z$. 
Replacing $z$ by $\varepsilon  b_j{z}+z_j$  in the last system  gives
\begin{equation}\label{boundry D11gsqg}
\begin{split}
&\Re\Big\{\gamma_j \Big(\Omega \big(\varepsilon  b_j{z}+{z_j}\big)+U-\frac{\gamma_j C_{\alpha}}{\varepsilon|\varepsilon|^{\alpha}  b_j^{1+\alpha} }\fint_{\partial \mathcal{O}_j^\varepsilon}\frac{1}{|\xi- z|^\alpha}d\xi\\&\qquad\qquad-\sum_{k=1, k\neq j}^N\frac{\gamma_k C_\alpha}{\varepsilon  b_k}\fint_{\partial \mathcal{O}_k^\varepsilon}\frac{1}{|\varepsilon  b_k\xi+z_k-\varepsilon  b_j\xi-z_j|^\alpha}d\xi\Big)\overline{z'}\Big\}=0\quad \forall  z\in \partial \mathcal{O}_j^\varepsilon.
\end{split}
\end{equation}
We shall    look for conformal parametrizations $\phi_j\colon\T\to \partial \mathcal{O}_j^\varepsilon$  
having the expansions 
\begin{align}\label{conf-alpha}
  \phi_j (w)= w+\varepsilon|\varepsilon|^{\alpha}  b_j^{1+\alpha}  f_j(w) \quad\textnormal{with}\quad f_j(w)=\sum_{m=1}^\infty\frac{a_m^j}{w^{m}},\quad a_m^j\in \mathbb{C},\quad j=1,\ldots,N.
\end{align}
Here, the coefficient  $|\varepsilon|^{\alpha}$ in the definition of the conformal mapping   $\phi_j$ comes from the singularity of the gSQG kernel.
For every $w \in \mathbb{T}$, the  tangent vector is given by $ z' = iw\phi_j'(w)$
 and therefore \eqref{boundry D11gsqg}  becomes 
\begin{equation} \label{sys1sqg0}
\begin{split}
  &\Im \Big\{\gamma_j
  \Big(
  \Omega\big(\varepsilon  b_j {\phi_j(w)} +{ z_j}\big)+U-\frac{\gamma_j C_{\alpha}}{\varepsilon|\varepsilon|^{\alpha}  b_j^{1+\alpha} } \fint_\T 
  \frac{\phi_j'(\tau)}{|\phi_j(\tau)-\phi_j(w)|^\alpha}\, d\tau 
 \\&   \qquad\quad-\sum_{k=1, k\neq j}^N\frac{\gamma_k C_\alpha}{\varepsilon  b_k}  \fint_\T \frac{
 \phi_k'(\tau)}
  { |
\varepsilon  b_k \phi_k(\tau)+z_k - \varepsilon  b_j \phi_j(w) - z_j|^\alpha}d\tau \Big)
 \overline{ w}\overline{\phi_j'(w)}
  \Big\}=0.
  \end{split}
\end{equation}
In order to desingularize the system \eqref{sys1sqg0} we shall use the following Taylor formula,
\begin{equation}\label{taylor0}
\frac{1}{|A+B|^\alpha}=\frac{1}{|A|^\alpha}-\alpha\int_0^1\frac{\Re(A\overline{B})+t|B|^2}{|A+tB|^{2+\alpha}}dt
\end{equation}
which is true for any complex numbers $A, B$ such that $|B|<|A|$. Taking $A=z_k-z_j$ and $B=b_k \phi_k(\tau) -   b_j \phi_j(w) $, one may write
\begin{align*}
  &\frac{1}{ |
  \varepsilon  b_k \phi_k(\tau)+z_k - \varepsilon  b_j \phi_j(w) - z_j|^\alpha}
  =\frac{1} { |z_k-z_j|^\alpha}
  \\&\qquad-\alpha \varepsilon\int_0^1\frac{
    \Re\big[\big(z_k-z_j\big)\big( 
  b_k \overline{\phi_k(\tau)} -   b_j \overline{\phi_j(w)}\big)\big]+t \varepsilon  |  b_k \phi_k(\tau) -   b_j \phi_j(w) |^2}
  { |
\varepsilon  b_k t\phi_k(\tau)+z_k - \varepsilon  b_j t\phi_j(w) + z_j|^{\alpha+2}}dt.
\end{align*}
It follows that
\begin{equation}\label{Tjsqg0}
\begin{split}
&\frac{ 1}{\varepsilon}  \fint_\T \frac{
 \phi_k'(\tau)}
  { |
\varepsilon  b_k \phi_k(\tau)+z_k - \varepsilon  b_j \phi_j(w) -d_j|^\alpha}d\tau=
\\ &\qquad- \alpha \fint_\T\int_0^1\frac{
\Re\big[\big( z_k-z_j\big)\big(
  b_k \overline{\phi_k(\tau)} -   b_j \overline{\phi_j(w)}\big)\big]}
  { |\varepsilon  b_k t\phi_k(\tau)+z_k- \varepsilon  b_j t\phi_j(w) -z_j|^{\alpha+2}} \phi_k'(\tau)dt\, d\tau 
\\ &\qquad- \alpha \int_0^1\fint_\T\frac{t \varepsilon |  b_k \phi_k(\tau) -   b_j \phi_j(w) |^2} { |\varepsilon  b_k t\phi_k(\tau)+z_k- \varepsilon  b_j t\phi_j(w) -z_j|^{\alpha+2}} \phi_k'(\tau)dt\, d\tau .
\end{split}
\end{equation}
On the other hand, from \eqref{conf-alpha}, one has
\begin{align*}
\frac{C_{\alpha}}{\varepsilon|\varepsilon|^{\alpha}} \fint_\T 
  \frac{\phi_j'(\tau)}{|\phi_j(\tau)-\phi_j(w)|^\alpha}\, d\tau
  &=\frac{C_{\alpha}}{\varepsilon|\varepsilon|^{\alpha} } \fint_\T 
  \frac{d\tau}{|\tau-w|^\alpha}+ C_{\alpha}\fint_\T 
  \frac{ b_j^{1+\alpha} f_j'(\tau)d\tau}{|w-\tau+t\varepsilon|\varepsilon|^\alpha  b_j^{1+\alpha}\big(f_j(\tau)-f_j(w)\big)|^\alpha}
\\ &\quad+\frac{C_{\alpha}}{\varepsilon|\varepsilon|^{\alpha}} \fint_\T \Big(
  \frac{1}{|w-\tau+t\varepsilon|\varepsilon|^\alpha  b_j^{1+\alpha}\big(f_j(\tau)-f_j(w)\big)|^\alpha}-  \frac{1}{|\tau-w|^\alpha}\Big)\, d\tau
 .
\end{align*}
Using the identity \cite[Page 337]{HH},
\begin{equation}\label{mu}
C_{\alpha}\fint_\T\frac{d\tau}{|\tau-w|^\alpha}=\frac{\alpha C_{\alpha}\Gamma(1-\alpha)}{(2-\alpha)\Gamma^2(1-\frac\alpha2)}w
 =:\mu_\alpha w ,
\end{equation}
and applying the formula \eqref{taylor0} with $A=\tau-w$ and $B=\varepsilon|\varepsilon|^\alpha  b_j^{1+\alpha}\big(f_j(\tau)-f_j(w)\big)$, we find 
\begin{equation*}
\begin{split}
\frac{C_{\alpha}}{\varepsilon|\varepsilon|^{\alpha}} \fint_\T 
  \frac{\phi_j'(\tau)}{|\phi_j(\tau)-\phi_j(w)|^\alpha}\, d\tau &=\frac{\mu_{\alpha}}{\varepsilon|\varepsilon|^{\alpha} }w+C_{\alpha} b_j^{1+\alpha}\fint_\T 
  \frac{f_j'(\tau)}{|w-\tau+t\varepsilon|\varepsilon|^\alpha  b_j^{1+\alpha}\big(f_j(\tau)-f_j(w)\big)|^\alpha}\, d\tau 
   \\ &\quad-\alpha C_{\alpha}  b_j^{1+\alpha} \fint_\T\int_0^1 
  \frac{\Re\Big(\big(f_j(\tau)-f_j(w)\big)\big(
 \overline{\tau} - \overline{w}\big)\Big)}{|w-\tau+t\varepsilon|\varepsilon|^\alpha  b_j^{1+\alpha}\big(f_j(\tau)-f_j(w)\big)|^{2+\alpha}}d\tau dt
 \\ &\quad-\alpha C_{\alpha}  \varepsilon|\varepsilon|^\alpha  b_j^{1+\alpha}\fint_\T\int_0^1 
  \frac{t|f_j(\tau)-f_j(w)|^2}{|w-\tau+t\varepsilon|\varepsilon|^\alpha  b_j^{1+\alpha}\big(f_j(\tau)-f_j(w)\big)|^{2+\alpha}}d\tau dt
   \\ &:=\frac{\mu_{\alpha}}{\varepsilon|\varepsilon|^{\alpha}  }w- b_j^{1+\alpha} \mathcal{I}^\alpha[\varepsilon, f_j](w).
\end{split}
\end{equation*}
As in the Euler case, by \eqref{conf-alpha} and the symmetry of the disc the singular term disappears from the nonlinearity, 
\begin{equation}\label{Qjsqg0}
\begin{split}
 & \Im \bigg\{
 \frac{C_{\alpha}}{\varepsilon|\varepsilon|^{\alpha}  b_j^{1+\alpha} } \fint_\T 
  \frac{\phi_j'(\tau)}{|\phi_j(\tau)-\phi_j(w)|^\alpha}\, d\tau\, \overline{ w}\,\overline{\phi_j'(w)}\bigg\}=\\ &\qquad\Im \Big\{\mu_\alpha \overline{f_j'(w)}-\mathcal{I}^\alpha[\varepsilon, f_j](w)\overline{ w}\big(1+\varepsilon|\varepsilon|^{\alpha} b_j^{1+\alpha} \overline{f_j'(w)} \big)\Big\}.
\end{split}
\end{equation}
Inserting \eqref{Tjsqg0} and  \eqref{Qjsqg0} into \eqref{sys1sqg0} we get
\begin{equation}\label{Fjsqgalpha0}
\begin{split}
 \gamma_j\mathcal{G}^\alpha_j(\varepsilon,f;\lambda)(w)&:=\gamma_j \Im \Big\{
  \Big(
  \Omega\big(\varepsilon  b_j w+\varepsilon^2|\varepsilon|^\alpha  b_j^{2+\alpha}  f_j(w) +{ z_j}\big)+U
 +\gamma_j \,\mathcal{I}^\alpha[\varepsilon, f_j](w) 
 \\
 & \quad +\sum_{k=1, k\neq j}^N\gamma_k \mathcal{J}^\alpha_k[\varepsilon,f_k,f_j](w)  \Big)
 \overline{ w}\big(1+\varepsilon|\varepsilon|^\alpha b_j^{1+\alpha} \overline{f_j'(w)} \big)
  -\mu_\alpha\, \gamma_j \overline{f_j'(w)}\Big\}=0
\end{split}
\end{equation}
for all  $w\in \T$ and $j=1,\ldots,N$, where $f=(f_1,\ldots,f_N)$, $\lambda$ denotes the point vortex parameters \eqref{lambda} and 
\begin{align}\label{Ialpha0}
\mathcal{I}^\alpha[\varepsilon, f_j](w)&:= -C_{\alpha} \fint_\T 
  \frac{f_j'(\tau)}{|\tau-w+\varepsilon|\varepsilon|^\alpha b_j^{1+\alpha}\big( f_j(\tau)-f_j(w)\big)|^\alpha}\, d\tau\nonumber
  \\
   &\;+\alpha C_{\alpha} \fint_\T\int_0^1 
  \frac{\Re\big[\big(f_j(\tau)-f_j(w)\big)\big(
 \overline{\tau} - \overline{w}\big)\big]+\varepsilon|\varepsilon|^\alpha t|f_j(\tau)-f_j(w)|^2}{|w-\tau+t\varepsilon|\varepsilon|^\alpha  b_j^{1+\alpha}\big(f_j(\tau)-f_j(w)\big)|^{2+\alpha}}dt\, d\tau,
  \\
\mathcal{J}^\alpha_k[\varepsilon,f_k,f_j;\lambda](w)&:= \frac{\alpha C_\alpha}{  b_k}\bigg[ \fint_\T \int_0^1\frac{
\Re\big[(z_k-z_j)\big(
  b_k \overline{\phi_k(\tau)} -   b_j \overline{\phi_j(w)}\big)\big]}
  { |\varepsilon  b_k t\phi_k(\tau)+z_k- \varepsilon  b_j t\phi_j(w) -z_j|^{\alpha+2}} \phi_k'(\tau)dt\, d\tau \nonumber 
\\ 
&\qquad + \fint_\T \int_0^1\frac{
t \varepsilon | 
  b_k \phi_k(\tau) -   b_j \phi_j(w) |^2\phi_k'(\tau) }
  { |\varepsilon  b_k t\phi_k(\tau)+z_k- \varepsilon  b_j t\phi_j(w) -z_j|^{\alpha+2}} dt\, d\tau\bigg] .\label{Jalpha0}
\end{align}

\subsection{Regularity and linearization of the functional \texorpdfstring{$\mathcal{G}^\alpha$}{G\^{}alpha}}\label{sec:linearize}
For any 
$\beta\in(0,1)$, we denote by $C^\beta(\mathbb{T}) $  the space of continuous functions $f\colon\T\to \mathbb{C}$ such that
$$
\Vert f\Vert_{C^\beta(\mathbb{T})}:= \Vert f\Vert_{L^\infty(\mathbb{T})}+\sup_{x\neq y\in \mathbb{T}}\frac{\vert f(x)-f(y)\vert}{\vert x-y\vert^\beta}<\infty.
$$
For any integer $n$, the space $C^{n+\beta}(\mathbb{T})$ stands for the set of functions $f$ of class $C^n$ whose $n$-th order derivatives are H\"older continuous  with exponent $\beta$. It is equipped with the usual  norm,
$$
\Vert f\Vert_{C^{n+\beta}(\mathbb{T})}:= \Vert f\Vert_{L^\infty(\mathbb{T})}+\Big\Vert \frac{d^n f}{dw^n}\Big\Vert_{C^\beta(\mathbb{T})}.
$$ 
We henceforth fix some $\beta \in (0,1)$, and for
$\alpha\in [0,1)$ define
$$
\eta_\alpha:=
\begin{cases}
  1-\alpha\, & \mbox{if } 0<\alpha<1,\\
  \beta\, & \mbox{if } \alpha=0.
\end{cases}
$$  

Consider the Banach spaces
\begin{equation}\label{Valpha0}
\begin{split}
&\mathcal{V}_0^\alpha :=\underbrace{V_0^\alpha\times\cdots \times V_0^\alpha}_{\text{$N$ times}}, \quad
\mathcal{W}_0^\alpha :=\underbrace{W^\alpha_0\times\cdots \times W_0^\alpha}_{\text{$N$ times}}, \quad\textnormal{and}\quad  
\widetilde{\mathcal{W}}_0^\alpha:=\underbrace{\widetilde{W}^\alpha_0\times\cdots \times \widetilde{W}_0^\alpha}_{\text{$N$ times}}
\\
\textnormal{with}\qquad  &V_0^\alpha:=
        \Big\{f\in  C^{1+{\eta_\alpha}}(\mathbb{T}): f(w)=\sum_{n\geq 1}
{a_n}\,\overline{w}^n,\;   a_n\in \mathbb{C}\Big\},\\
 &  W_0^\alpha:=
        \Big\{g\in  C^{{\eta_\alpha}}(\mathbb{T}): g(w)=\sum_{n\neq 0}
c_n\,{w}^n, \; c_{-n}=\overline{c_n}\in \mathbb{C}\Big\}, \\
& \widetilde{W}^\alpha_0:=\Big\{g\in  W^\alpha: 
c_{1}=0\Big\}.
\end{split}
\end{equation}
Here the subscript $0$ refers to the fact that no a priori symmetry assumption is made on the patch boundaries. The restriction on the Fourier coefficients of functions in $V^\alpha_0$ guarantees that these functions can be extended to holomorphic functions on $\C \without \overline{\D}$. The restriction on the Fourier coefficients of functions in $W_0^\alpha$, on the other hand, simply says that these functions are real-valued.

We denote by $B_0^\alpha$ the unit ball in $\mathcal{V}_0^\alpha$,
$$
B_0^\alpha:=\big\{f\in \mathcal{V}^\alpha: \|f\|_{C^{\eta_\alpha+1}(\mathbb{T})}< 1\big\}
$$
and by $\mathcal{B}_0^\alpha$ the unit ball in $\mathcal{W}_0^\alpha$,
$$
\mathcal{B}_0^\alpha:=\big\{f\in \mathcal{W}^\alpha: \|f\|_{C^{\eta_\alpha}(\mathbb{T})}< 1\big\}.
$$
We shall  unify the expression of the function $\mathcal{G}^\alpha_j$ in \eqref{Fjsqgalpha0} and  the  function $\mathcal{G}^0_j$ in  \eqref{rotn+100} as follows: For every $f\in \mathcal{V}_0^\alpha$ and  $w\in \T$, 
\begin{equation}\label{exp:Galpha}
\begin{split}
\mathcal{G}^\alpha_{j}&(\varepsilon,f;\lambda)(w):=\Im \Big\{
 \Big( \Omega\big(\varepsilon  b_j w+\varepsilon^2|\varepsilon|^\alpha  b_j^{2+\alpha}  f_j(w)+{ z_j}\big)+U\Big)
 \overline{ w}\big(1+\varepsilon|\varepsilon|^\alpha b_j^{1+\alpha} \overline{f_j'(w)} \big)
\\
&-\mu_\alpha\, \gamma_j \overline{f_j'(w)}+
\Big( \gamma_j \,\mathcal{I}^\alpha[\varepsilon, f_j](w)+ \sum_{k=1, k\neq j}^N\gamma_k \mathcal{J}^\alpha_k[\varepsilon,f_k,f_j](w) \Big)\overline{ w}\big(1+\varepsilon|\varepsilon|^\alpha b_j^{1+\alpha} \overline{f_j'(w)} \big)\Big\},
\end{split}
\end{equation}
 where $\mu_\alpha$ is defined in  \eqref{mu}, and
 $\mathcal{I}^\alpha[\varepsilon, f_j](w)$ and $\mathcal{J}^\alpha_k[\varepsilon,f_k,f_j](w)$ are given by \eqref{Ialpha0}--\eqref{Jalpha0} if $\alpha\in (0,1)$ and by  \eqref{I00}--\eqref{J00} if $\alpha=0$. We then define the nonlinear operator 
\begin{equation*}
\mathcal{G}^\alpha(\varepsilon,f;\lambda):=\big(\mathcal{G}^\alpha_1(\varepsilon,f;\lambda),\ldots,\mathcal{G}^\alpha_N(\varepsilon,f;\lambda)\big).\end{equation*}

\begin{proposition}\label{proposition2}
  Let $\alpha\in[0,1)$ and let $\lambda^*$ solve \eqref{alg-sysP}.
  \begin{enumerate}[label=\rm(\roman*)]
  \item There exists $\varepsilon_0 > 0$ and a small neighborhood $\Lambda$ of $\lambda^*$ such that $\mathcal{G}^\alpha$ can be extended to a $C^1$ mapping $(-\varepsilon_0,\varepsilon_0)\times B_0^\alpha\times \Lambda \to \mathcal{W}_0^\alpha$.  
  \item  For all $\lambda \in \Lambda$ one has
    \begin{equation*} 
      \mathcal{G}_j^\alpha (0,0;\lambda)(w)=
       \Im \big\{ \mathcal{P}_j^\alpha(\lambda) \overline{w}\big\}
    \end{equation*}
  for $w \in \T$, where $\mathcal{P}_j^\alpha(\lambda)$ is given by \eqref{alg-sysP}.

\item The Frech\'et derivative of $\mathcal{G}^\alpha$ with respect to $f$ at $(0,0;\lambda)$ is given by
  \begin{align*}
    D_{f}\mathcal{G}^\alpha(0,0;\lambda)h(w)=&\sum_{n\geq 1}M_n^\alpha\begin{pmatrix}
      {\gamma_1}\, \Im\big\{a_n^1\,w^{n+1}\big\} \\
      \vdots
      \\
      {\gamma_N}\, \Im\big\{a_n^N\,w^{n+1}\big\}
    \end{pmatrix},
    \notag
  \end{align*}
  where here $h=(h_1,\ldots,h_N)\in \mathcal{V}_0^\alpha$ with
  $
  h_j(w):=\displaystyle\sum_{n\geq 1}a_n^j\overline{w}^{n},
  $
  and
  \begin{equation}\label{lambdan}
    M_n^\alpha:=
        \frac{\Gamma(1+\frac\alpha2) \Gamma(1-\alpha)}{2^{2-\alpha}\Gamma^3(1-\frac\alpha2)}\Big(\frac{2(n+1)}{1-\frac\alpha2}-\frac{(1+\frac\alpha2)_n}{(1-\frac\alpha2)_n}-\frac{(1+\frac\alpha2)_{n+1}}{(1-\frac\alpha2)_{n+1}}\Big) .
  \end{equation}
\item  For any $\lambda \in \Lambda$, the linear operator 
  $D_{f}\mathcal{G}^\alpha (0,0;\lambda)\colon \mathcal{V}_0^\alpha \to  \widetilde{\mathcal{W}}_0^\alpha$ is an isomorphism.
  \end{enumerate}
\end{proposition}

\begin{proof}

Note that the term
\begin{equation*}
\Im \Big\{
 \Big( \Omega\big(\varepsilon  b_j w+\varepsilon^2|\varepsilon|^\alpha  b_j^{2+\alpha}  f_j(w)+{ z_j}\big)+U+ \gamma_j \,\mathcal{I}^\alpha[\varepsilon, f_j](w)\Big)
 \overline{ w}\big(1+\varepsilon|\varepsilon|^\alpha b_j^{1+\alpha} \overline{f_j'(w)} \big)
  -\mu_\alpha\, \gamma_j \overline{f_j'(w)}\Big\},
\end{equation*}
 appears identically in the study of the vortex pairs in \cite{Hmidi-Mateu}, where it was shown that it is $C^1$ in $\varepsilon$ and $f$. The spaces used in \cite{Hmidi-Mateu} are $V^\alpha_0$  with the reflection symmetry  $\overline{f_j(w)}=f_j(\overline{w})$ and $W_0^\alpha$ with the symmetry  $g_j(\overline{w})=-g_j(w)$, but the study of the regularity  can be easily generalized  to our case without further difficulties. The terms 
  \begin{equation*}
  \sum_{k=1, k\neq j}^N \Im \Big\{
\gamma_k \mathcal{J}^\alpha_k[\varepsilon,f_k,f_j;\lambda](w)  \overline{ w}\big(1+\varepsilon|\varepsilon|^\alpha b_j^{1+\alpha} \overline{f_j'(w)} \big)\Big\}
\end{equation*}   describe the
interaction between the boundaries of the patches, which are disjoint provided $\varepsilon$ is sufficiently small. Therefore the  kernel involved is sufficiently  smooth and can be treated  in a very classical way.

\vspace{0.2cm}

Next, we shall prove  {\rm(ii)} and {\rm (iii)} in the case $\alpha=0$.  Substituting $\varepsilon=0$ in \eqref{I00} and \eqref{J00} gives
\begin{align*}
\overline{\mathcal{J}_{kj}^0}[0,f_k,f_j;\lambda](w)&=-\frac12  \frac{
1}
  {z_j-z_k},\\
\overline{\mathcal{I}^0}[0, f_j](w)&=\frac12\fint_{\mathbb{T}}\frac{\overline{w}-\overline{\tau}}{w-\tau}f_j'(\tau)d\tau +\fint_{\mathbb{T}}\frac{i\Im\big\{(w-\tau)\big(\overline{f_j({\tau})}-\overline{f_j({w})}\big)\big\}}{(w-\tau)^2}d\tau
=0,
\end{align*}
where we have used  the residue theorem in the last identity.
Then, from \eqref{rotn+100}, we get
\begin{align} \label{rotn+1}
\mathcal{G}_j^0 (0,f;\lambda)(w)=-\Im \Big\{
  \Big(&
 -\Omega \overline{z_j}+U-\frac12\sum_{k=1,k\neq j}^N\, \frac{
\gamma_k}
  {z_j- z_k}
    \Big)
  w
  -\frac{\gamma_j}{2} f'_j(w)\Big\}
\end{align}
and hence
\begin{align*} 
\mathcal{G}_j^0 (0,0;\lambda)(w)&= -\Im \Big\{
  \Big(
 -\Omega \overline{z_j}+U-\frac12\sum_{k=1,k\neq j}^N\, \frac{
\gamma_k}
  {z_j- z_k}
    \Big)
  w\Big\}\\ &\stackrel{ \eqref{alg-sysP}}{=} -\Im \big\{
 \overline{\mathcal{P}^0_j(\lambda)}
  w\big\},
  \end{align*}
which shows {\rm(ii)}.
Differentiating \eqref{rotn+1} with respect to $f$ gives
\begin{align*} 
\partial_{f_k}\mathcal{G}_j^0(0,0;\lambda)h_j(w)=\delta_{jk}\frac{\gamma_j}{2}\Im \big\{h'_j(w)\big\},
\end{align*}
implying {\rm(iii)} in the case $\alpha=0$.

\vspace{0.2cm}

To prove {\rm(ii)} and {\rm(iii)} in the case $\alpha\in(0,1)$, we substitute  $\varepsilon=0$ in \eqref{Ialpha0} and \eqref{Jalpha0} and obtain
\begin{equation*}
\begin{split}
\mathcal{I}^\alpha[0, f_j](w)&=- C_{\alpha} \fint_\T 
  \frac{f_j'(\tau)}{|\tau-w|^\alpha}\, d\tau+\alpha C_{\alpha} \fint_\T\int_0^1 
  \frac{\Re\big[\big(f_j(\tau)-f_j(w)\big)\big(
 \overline{\tau} - \overline{w}\big)\big]}{|w+\tau|^{2+\alpha}}d\tau dt,\\
\mathcal{J}^\alpha_k[0,f_k,f_j;\lambda](w)
 &= \frac{\alpha C_\alpha}{2} \frac{
z_k-z_j}
  { |z_k -z_j|^{\alpha+2}}.
\end{split}
\end{equation*}
Thus, by \eqref{Fjsqgalpha0},  one has
\begin{equation}\label{Fjsqgalpha0eps0}
\begin{split}
 \mathcal{G}^\alpha_j(0,f;\lambda)(w)= \Im \Big\{&
  \Big(
  \Omega{ z_j} +U
+\frac{\alpha C_\alpha}{2} \sum_{{k=1, k\neq j}}^N \gamma_k\frac{
z_k-z_j}
  { |z_k -z_j|^{\alpha+2}}  \Big)
 \overline{ w}
 \\ &+\gamma_j\, \overline{ w} \, \mathcal{I}^\alpha[0, f_j](w) -\mu_\alpha\, \gamma_j \overline{f_j'(w)}\Big\}.
\end{split}
\end{equation}
It follows that 
\begin{equation*}
\begin{split}
 \mathcal{G}^\alpha_j(0,0;\lambda)(w)&= \Im \Big\{
  \Big(
  \Omega{ z_j}+U-\frac{\alpha C_\alpha}{2} \sum_{k=1, k\neq j}^N \gamma_k\frac{
z_j-z_k}
  { |z_k -z_j|^{\alpha+2}}  \Big)
 \overline{ w}
\Big\}.
\end{split}
\end{equation*}
Comparing the last expression with \eqref{alg-sysP} concludes {\rm(ii)}. 
Next, differentiating \eqref{Fjsqgalpha0eps0} with respect to $f$ gives
\begin{align} \label{dgj-sqg0}
& \partial_{f_k}\mathcal{G}_j^\alpha(0,0;\lambda)h_j(w)=\delta_{jk} \gamma_j\Im \Big\{
   \mathcal{I}^\alpha[ 0,h_j](w)
 \overline{ w}-\mu_\alpha \overline{h_j'(w)}\Big\}.
\end{align}
The last expression was explicitly computed  in \cite[Pages 726--728]{Hmidi-Mateu}  and takes the form
\begin{equation*}
 \partial_{f_j}\mathcal{G}_j^\alpha(0,0;\lambda)h_j(w)= \gamma_j\sum_{n\geq 1}\frac{\alpha C_\alpha \Gamma(1-\alpha)}{4\Gamma^2(1-\frac\alpha2)}\Big(\frac{2(n+1)}{1-\frac\alpha2}-\frac{(1+\frac\alpha2)_n}{(1-\frac\alpha2)_n}-\frac{(1+\frac\alpha2)_{n+1}}{(1-\frac\alpha2)_{n+1}}\Big) \Im \big\{ a_n^j w^{n+1}
  \big\},
\end{equation*}
getting the announced result.

\vspace{0.2cm}

 The proof of {\rm(iv)}  is elementary in the case $\alpha=0$, since 
\begin{align*}
D_{f}\mathcal{G}^\alpha(0,0;\lambda)h(w)=&\sum_{n\geq 1}\frac12\begin{pmatrix}
{\gamma_1}\, \Im\big\{h_1'(w)\big\} \\
\vdots
\\
{\gamma_N}\, \Im\big\{h_N'(w)\big\}
\end{pmatrix}.
\notag
\end{align*}
For the case $\alpha\in(0,1)$, we reproduce similar arguments to \cite{Hmidi-Mateu}. In particular, we observe from \eqref{dgj-sqg0}  that  $D_{f}\mathcal{G}_j^\alpha(0,0;\lambda)\colon   \mathcal{V}_0^\alpha\to  \widetilde{\mathcal{W}}_0^\alpha$ is a compact perturbation of a Fredholm operator of index zero, since  $ \mathcal{I}^\alpha\colon   \mathcal{V}_0^\alpha\to   \widetilde{\mathcal{W} }_0^\alpha$ is smoothing.   Therefore $\partial_{f}\mathcal{G}_j^\alpha(0,0;\lambda)\colon  \mathcal{V}_0^\alpha\to   \widetilde{\mathcal{W}}_0^\alpha$ is Fredholm with index zero.
To check that is has a trivial kernel we can argue as in \cite[page 728]{Hmidi-Mateu}.
This concludes the proof of the proposition.
\end{proof}

\begin{remark}\label{remark-23}
  Since $\widetilde{\mathcal W}_0^\alpha \subset \mathcal W_0^\alpha$ has codimension $2N$, Proposition~\ref{proposition2}(iv) implies that $D_f \mathcal G^\alpha(0,0;\lambda^*)$ is Fredholm with finite index $-2N$. As we will see, even for non-generate equilibria the full linearized operator $D_{(f;\lambda)} \mathcal G^\alpha(0,0;\lambda^*)$ will have a range with nonzero codimension. While this rules out an immediate application of the implicit function theorem, one can still, in a small neighborhood of $(0,0;\lambda^*)$, the nonlinear problem $\mathcal G^\alpha(\varepsilon,f;\lambda) = 0$ will always admit a Lyapunov--Schmidt reduction to an equation in finite dimensions; see for instance \cite{kielhofer}. Of course, studying this reduced equation may still be quite challenging, and may in particular involve evaluating further Fr\'echet derivatives of $\mathcal G^\alpha$. 
\end{remark}

\subsection{Integral identities for the functional \texorpdfstring{$\mathcal{G}^\alpha$}{G\^{}alpha}}\label{sec:Gidentities}
In this short section we prove several integral identities for the functional $\mathcal{G}^\alpha_j$ which follow from Lemma~\ref{lem:idens}. 

First, we show how that functional $\mathcal{G}^\alpha_j$ defined in \eqref{Fjsqgalpha0} can be written in terms of the relative stream function 
\begin{equation}\label{relative-stream-function}
\Psi^\varepsilon(z):=-\frac12\Omega|z|^2-\frac12U\overline{z}+\psi^\varepsilon(z)
\end{equation}
restricted to points on the boundary $\varepsilon b_j\phi_j(w)+z_j\in \partial \mathcal{D}_j^\varepsilon$. Writing $w\in \mathbb{T}$ as $w=e^{i\theta}$, we claim that
\begin{align}\label{psiUG}
  \partial_\theta\Psi^\varepsilon(\varepsilon b_j\phi_j(w)+z_j)&=-\varepsilon b_j\, \mathcal{G}^\alpha_j(\varepsilon,f;\lambda)(w).
\end{align}
To see this, we use $\partial_\theta \phi_j(w)= iw\phi_j'(w)$ and $v^\varepsilon(z) =2i\partial_{\overline{z}}\psi^\varepsilon(z)$ to rewrite
\begin{align*}
  \partial_\theta\Psi^\varepsilon(\varepsilon b_j\phi_j(w)+z_j)&=2\varepsilon b_j\, \Re\Big\{\partial_{\overline{z}}\Psi^\varepsilon\big(\varepsilon b_j\phi_j(w)+z_j\big)\partial_\theta \overline{\phi_j(w)}\Big\}
  \\ &=-\varepsilon b_j\, \Re\Big\{2i\partial_{\overline{z}}\Psi^\varepsilon\big(\varepsilon b_j\phi_j(w)+z_j\big)\overline{w}\overline{\phi'_j(w)}\Big\}\\ 
  &=-\varepsilon b_j\, \Re\Big\{\Big(-i\Omega\big(\varepsilon b_j\phi_j(w)+z_j\big)-iU+v^\varepsilon\big(\varepsilon b_j\phi_j(w)+z_j\big)\Big)\overline{w}\overline{\phi'_j(w)}\Big\}\\ 
  &=-\varepsilon b_j\,\Im\Big\{\Big(\Omega\big(\varepsilon b_j\phi_j(w)+z_j\big)+U+iv^\varepsilon\big(\varepsilon b_j\phi_j(w)+z_j\big)\Big) \overline{w}\overline{\phi'_j(w)}\Big\}\\ 
  &=-\varepsilon b_j\, \mathcal{G}^\alpha_j(\varepsilon,f;\lambda)(w).
\end{align*}
We can now prove the following lemma.
\begin{lemma}\label{coro:idG}
  Let $(\varepsilon,f,\lambda)\in(-\varepsilon_0,\varepsilon_0)\times B_0^\alpha\times \Lambda$.
  Then, the following identities hold:
  \begin{enumerate}[label=\rm(\roman*)]
  \item If $\Omega=0$ then
    \begin{align}
      \label{eqn:Gident:trans}
      \begin{aligned}
        \Re\Big[\sum_{j=1}^N\frac{\gamma_j}{\pi}\int_{\mathbb{T}}\mathcal{G}^\alpha_j(\varepsilon,f;\lambda)(w)\big(1+\varepsilon b_jf_j(w)\overline{w}\big)dw\Big]&={U}\sum_{j=1}^N  \gamma_j\Big(1+\varepsilon^2b_j^2 \fint_{\mathbb{T}}\overline{f_j(w)}f'_j(w)dw\Big),\\
        \Im\Big[\sum_{j=1}^N\frac{\gamma_j}{\pi}\int_{\mathbb{T}}\mathcal{G}^\alpha_j(\varepsilon,f;\lambda)(w)\big(1+\varepsilon b_jf_j(w)\overline{w}\big)dw\Big]&=0.
      \end{aligned}
    \end{align}
  \item If $U=0$ then
    \begin{align}
      \label{eqn:Gident:rot}
      \sum_{j=1}^N\frac{\gamma_j}{i\pi}\int_{\mathbb{T}}\mathcal{G}^\alpha_j(\varepsilon,f;\lambda)(w) \Big(\varepsilon b_j|w+\varepsilon b_j f_j(w)|^2+\varepsilon b_j\Re\big[\overline{z_j}f_j(w)\big]+\Re[\overline{z_j}w]\Big)\overline{w}dw=0.
    \end{align}
  \end{enumerate}
\end{lemma}
\begin{proof}
  By continuity, it is enough to consider $\varepsilon \ne 0$, and thanks to the $\varepsilon \mapsto -\varepsilon$ symmetry in \eqref{g-eps}, we can further restrict to $\varepsilon > 0$.

  We shall first prove {\rm (i)}.
  Averaging \eqref{relative-stream-function} over the boundaries, with $\Omega=0$, then summing and using Lemma~\ref{lem:idens} yields
  $$
  \sum_{j=1}^N\frac{\gamma_j}{\varepsilon^2 b_j^2}\int_{\partial \mathcal{D}_j^\varepsilon}\Psi^\varepsilon(z)dz=
  -\frac{U}{2}\sum_{j=1}^N\frac{\gamma_j}{\varepsilon^2 b_j^2}\int_{\partial \mathcal{D}_j^\varepsilon}\overline{z}dz.
  $$
  In view of  \eqref{Dj} and \eqref{conf0}, by making  simple changes of variables we get
  \begin{equation}\label{id-psi-mass}
    \sum_{j=1}^N\frac{\gamma_j}{\varepsilon b_j} \int_{ \mathbb{T}}\Psi^\varepsilon\big({\varepsilon b_j}\overline{\phi_j(w)}+\overline{z_j}\big)\phi'_j(w)dw=- \frac{U}{2}\sum_{j=1}^N\frac{\gamma_j}{\varepsilon b_j}\int_{ \mathbb{T}}\big({\varepsilon b_j}\overline{\phi_j(w)}+\overline{z_j}\big)\phi'_j(w)dw.
  \end{equation}
  Writing $w\in \mathbb{T}$ as $w=e^{i\theta}$ then integrating by parts and using \eqref{psiUG},  the left hand side of \eqref{id-psi-mass} becomes
  \begin{align}
    \nonumber \sum_{j=1}^N\frac{\gamma_j}{\varepsilon b_j} \int_{ \mathbb{T}}\Psi^\varepsilon\big({\varepsilon b_j}\overline{\phi_j(w)}+\overline{z_j}\big)\phi'_j(w)dw&=\sum_{j=1}^N \frac{\gamma_j}{\varepsilon b_j}\int_{0}^{2\pi}\Psi^\varepsilon(\varepsilon b_j\phi_j(w)+z_j)\partial_\theta\phi_j(w)d\theta
    \\ \nonumber &=-\sum_{j=1}^N\frac{\gamma_j}{\varepsilon b_j} \int_{0}^{2\pi} \partial_\theta \Psi^\varepsilon(\varepsilon b_j\phi_j(w)+z_j)\phi_j(w)d\theta
    \\ &=-i\sum_{j=1}^N\gamma_j\int_{\mathbb{T}}\mathcal{G}^\alpha_j(\varepsilon,f;\lambda)(w)\phi_j(w)\overline{w}dw.\label{Id-int-psi}
  \end{align}
  Combining \eqref{id-psi-mass}--\eqref{Id-int-psi} and using the  simple  identity 
  \begin{align*}
    \frac{1}{\varepsilon b_j}\int_{ \mathbb{T}}\big({\varepsilon b_j}\overline{\phi_j(w)}+\overline{z_j}\big)\phi'_j(w)dw
    =2\pi i+\varepsilon^2 b^2 \int_{ \mathbb{T}}\overline{f_j(w)}f'_j(w)dw
  \end{align*} 
  we conclude that
  \begin{align*}
    \sum_{j=1}^N\frac{\gamma_j}{\pi}\int_{\mathbb{T}}\mathcal{G}^\alpha_j(\varepsilon,f;\lambda)(w)\big(1+\varepsilon b_jf_j(w)\overline{w}\big)dw&={U}\sum_{j=1}^N  \gamma_j\Big(1+\varepsilon^2b_j^2 \fint_{\mathbb{T}}\overline{f_j(w)}f'_j(w)dw\Big).
  \end{align*}
  Then {\rm (i)} follows  from  the identity
  $$
  \fint_{\mathbb{T}}\overline{f_j(w)}f'_j(w)dw-\overline{\fint_{\mathbb{T}}\overline{f_j(w)}f'_j(w)dw}=\frac{1}{2\pi i}\fint_0^{2\pi}\partial_\theta|{f_j(w)}|^2d\theta=0.
  $$ 

  As for {\rm (ii)}, integrating \eqref{relative-stream-function} with $U=0$ over the boundaries then summing and taking the real part we obtain, by virtue of Lemma~\ref{lem:idens},
  $$
  \Re\Big[\sum_{j=1}^N\frac{\gamma_j}{\varepsilon^2 b_j^2}\int_{\partial \mathcal{D}_j^\varepsilon}\overline{z}\Psi^\varepsilon(z)dz\Big]=
  -\frac{\Omega}{2}\, \Re\Big[\sum_{j=1}^N\frac{\gamma_j}{\varepsilon^2 b_j^2}\int_{\partial \mathcal{D}_j^\varepsilon}\overline{z}|{z}|^2dz\Big].
  $$
  According to \eqref{Stokes Thm}, the quantity 
  $$
  \sum_{j=1}^N\frac{\gamma_j}{\varepsilon^2 b_j^2}\int_{\partial \mathcal{D}_j^\varepsilon}\overline{z}|{z}|^2dz=4i\sum_{j=1}^N\frac{\gamma_j}{\varepsilon^2 b_j^2}\int_{ \mathcal{D}_j^\varepsilon}|{z}|^2dA(z)
  $$
  is purely imaginary. It follows that
  $$
  \Re\Big[\sum_{j=1}^N\frac{\gamma_j}{\varepsilon^2 b_j^2}\int_{\partial \mathcal{D}_j^\varepsilon}\overline{z}\Psi^\varepsilon(z)dz\Big]=0.$$
  Then, in view of  \eqref{Dj} and  \eqref{conf0}, by  making  simple changes of variables we get
  \begin{equation}\label{id-psi-mass22}
    \begin{aligned}
      \sum_{j=1}^N& \frac{\gamma_j}{\varepsilon b_j} \mathcal{A}_j:=\sum_{j=1}^N& \frac{\gamma_j}{\varepsilon b_j}  \Re\Big[ \int_{\mathbb{T}}\big(\varepsilon b_j\overline{\phi_j(w)}+\overline{z_j}\big)\Psi^\varepsilon\big(\varepsilon b_j\phi_j(w)+z_j\big)\phi'_j(w)dw\Big]=0.
    \end{aligned}
  \end{equation}
  Writing $w\in \mathbb{T}$  as $w=e^{i\theta}$   then integrating by parts and using \eqref{psiUG} we find
  \begin{align}
    \mathcal{A}_j &= \nonumber  \int_{0}^{2\pi}\Re\Big[\big(\varepsilon b_j\overline{\phi_j(w)}+\overline{z_j}\big)\partial_\theta\phi_j(w)\Big]\Psi^\varepsilon\big(\varepsilon b_j\phi_j(w)+z_j\big) d\theta\\ &=
    \nonumber  \int_{0}^{2\pi}\partial_\theta\Big(\varepsilon b_j|\phi_j(w)|^2+\Re\big[\overline{z_j}\phi_j(w)\big]\Big)\Psi^\varepsilon\big(\varepsilon b_j\phi_j(w)+z_j\big) d\theta \\&= \nonumber 
    - \int_{0}^{2\pi} \Big(\varepsilon b_j|\phi_j(w)|^2+\Re\big[\overline{z_j}\phi_j(w)\big]\Big)\partial_\theta\Big[\Psi^\varepsilon\big(\varepsilon b_j\phi_j(w)+z_j\big)\Big]d\theta
     \\ &=  -i\varepsilon b_j\int_{\mathbb{T}}\Big(\varepsilon b_j|\phi_j(w)|^2+\Re\big\{\overline{z_j}\phi_j(w)\big\}\Big)\mathcal{G}^\alpha_j(\varepsilon,f;\lambda)(w) \overline{w}dw.\label{id-psi221}
  \end{align}
  Combining \eqref{id-psi-mass22} and \eqref{id-psi221} gives the desired result.
\end{proof}

\begin{remark}\label{remark-pv-identity}
  Substituting $(\varepsilon,f)=(0,0)$  in the identities of Lemma~\ref{coro:idG} and using Proposition~\ref{proposition2}{\rm(ii)} we obtain well known identities for point vortices: If $\Omega=0$ we have
  \begin{align*}
    \sum_{j=1}^N {\gamma_j} \mathcal{P}_j^\alpha(\lambda) &={U}\sum_{j=1}^N  \gamma_j
  \end{align*}
  and if $U=0$ we get
  \begin{align*}
    \Im \Big\{ \sum_{j=1}^N{\gamma_j} \mathcal{P}_j^\alpha(\lambda) \overline{z_j} \Big\}=0.
  \end{align*}
  See, for instance, \cite{oneil:stationary}. These identities can be seen as coming from the symmetries of the Hamiltonian system \eqref{ode-sys0} under translations and rotations.
\end{remark}

\subsection{An abstract lemma}\label{sec:abstract}
From Proposition~\ref{proposition2}, we know that the operator $D_f \mathcal G^\alpha(0,0;\lambda^*)$ is not onto. As discussed in Remark~\ref{remark-23}, the linearized operator $D_{(f;\lambda)} \mathcal G^\alpha(0,0;\lambda^*)$ is also not onto. Thankfully, we will be able to ``explain'' this latter degeneracy using the nonlinear identities in Lemma~\ref{coro:idG}. More precisely, for non-degenerate equilibria, we will be able to apply the following mild generalization of the usual implicit function theorem.
\begin{lemma}\label{abstract-lemma}
  Let $G  \colon U\times V \to Z$ and $F \colon Z\times U\times V \to\mathbb{R}^n$ be $C^1$ mappings satisfying 
  \begin{align}
    \label{eqn:G0}
    G(0,0)&=0,\\
    \label{eqn:F0}
    F(0,x,y)&=0,\\
    \label{eqn:FG}
    F(G(x,y),x,y)&=0,
  \end{align}
  for all $(x,y) \in U \by V$, where $X,Y,Z$ are Banach spaces and $U \subset X$ and $V \subset Y$ are open sets containing the origin. If the linearizations of these mappings at the origin satisfy 
  \begin{align}
    \label{eqn:kerG}
    \ker D_y G(0,0)&=\{ 0\},\\
    \label{eqn:ranG}
    \codim \ran D_y G(0,0) &= n,\\
    \label{eqn:ranF}
    \ran D_z F(0,0,0) &= \mathbb{R}^n,
  \end{align}
  then there exists a neighborhood $\tilde U\times \tilde V$ of the origin in $U \by V$ and a $C^1$ mapping $g \maps  \tilde U \to \tilde V$ such that
  \begin{align}
    \notag
    g(0)&=0,\\
    \label{eqn:g}
    G(x,g(x))&=0\quad \textnormal{for all }x\in \tilde U.
  \end{align}
  Moreover, every solution of $G(x,y)=0$ in $\tilde{U}\times \tilde{V}$ is of the form $(x,g(x))$, and the operator $D_x g(0)$ is uniquely determined by the equation
  \begin{align*}
    D_x G(0,0) + D_y G(0,0) D_x g(0) &= 0
  \end{align*}
  obtained by implicitly differentiating \eqref{eqn:g}.
  \begin{proof}
    We first claim that the linearizations of $G$ and $F$ at the origins in $U \by V$ and $Z \by U \by V$ satisfy
    \begin{align}
      \label{eqn:exact}
      \ran D_y G(0) = \ker D_z F(0).
    \end{align}
    To see this, we differentiate \eqref{eqn:F0} and \eqref{eqn:FG} with respect to $x$ and $y$ to find
    \begin{align*}
      D_x F(0) = 0,
      \quad 
      D_y F(0) = 0,
      \quad 
      D_z F(0) D_x G(0) + D_x F(0) = 0,
      \quad 
      D_z F(0) D_y G(0) + D_y F(0) = 0.
    \end{align*}
    This in particular implies that $D_z F(0) D_yG(0)=0$, and hence $\ran D_y G(0) \subset \ker D_z F(0)$. As \eqref{eqn:ranF} and \eqref{eqn:ranG} force $\codim \ker D_z F(0) = \codim \ran D_y F(0) = n$, the only possibility is that the two spaces are equal and \eqref{eqn:exact} holds.

    Again using \eqref{eqn:ranF}, we know that there is an $n$-dimensional subspace $Z_1 \sub Z$ such that $D_z F(0)$ restricts to an invertible map $Z_1 \to \R^n$. Moreover, by \eqref{eqn:exact} we have 
    \begin{align}
      \label{eqn:splitting}
      Z = Z_1 \oplus \ker D_z F(0) = Z_1 \oplus \ran G_y(0).
    \end{align}
    Consider the augmented mapping
    \begin{align*}
      H \colon Z_1 \by U \by V \to Z,
      \qquad 
      H(z_1,x,y)=G(x,y)-z_1.
    \end{align*}
    The linearized operator $D_{(z_1,y)} H(0) \colon Z_1 \by Y \to Z$, given by
    \begin{align}
      \label{eqn:DH}
      D_{(z_1,y)} H(0) 
      \begin{pmatrix}
         z_1 \\  y 
      \end{pmatrix}
      = 
      D_y G(0) y - z_1,
    \end{align}
    is invertible. Indeed, it is onto by \eqref{eqn:splitting}. If $(z_1, y)$ lies in its kernel, applying $D_z F(0)$ to \eqref{eqn:DH} and using the definition of $Z_1$ \eqref{eqn:exact} yields $z_1 = 0$, at which point $y = 0$ follows from \eqref{eqn:kerG}. Thus, by the implicit function theorem there is a neighborhood $\tilde Z_1 \times \tilde U \times \tilde V$ of the origin in $Z_1 \times U \times V$ and $C^1$ mappings $(h,g) \maps \tilde U \to \tilde Z_1 \times \tilde V$ satisfying
    \begin{align}
      \notag
      (h(0),g(0)) &= 0, \\
      \label{eqn:h}
      H(h(x),x,g(x)) = G(x,g(x)) -h(x) &= 0 \quad \textnormal{for all }x\in \tilde U.
    \end{align}
    Moreover, all solutions of $H(z,x,y)=0$ in $\tilde Z_1 \times \tilde U \times \tilde V$ are of the form $(h(x),x,g(x))$, and the linearizations of $(h,g)$ at the origin are uniquely determined by
    \begin{align*}
      0 &= D_{z_1} H(0) D_x h(0) + D_x H(0) + D_y H(0) D_x g(0)\\
      &= -D_x h(0) + D_x G(0) + D_y G(0) D_x g(0).
    \end{align*}

    The proof will therefore be complete if we can show that $h \equiv 0$, for which will need the full force of the nonlinear conditions \eqref{eqn:F0} and \eqref{eqn:FG} and not just their linearizations at the origin. Consider the restriction
    \begin{align*}
      F_1 = F\big|_{Z_1 \times U \times V} 
      \colon Z_1 \times U \times V \longrightarrow \mathbb R^n.
    \end{align*}
    By our choice of $Z_1$, the linearized operator $D_{z_1} F_1(0) = D_z F(0)|_{Z_1}$ is invertible. Applying the implicit function theorem and using \eqref{eqn:F0} we conclude that, possibly after shrinking $\tilde Z_1 \times \tilde U \times \tilde V$, all solutions of $F_1(z_1,x,y)=0$ in $\tilde Z_1 \times \tilde U \times \tilde V$ are of the form $(0,x,y)$. By \eqref{eqn:FG} and \eqref{eqn:h} we have
    \begin{align*}
      0 = F(G(x,g(x)),x,g(x))
      = F(h(x),x,g(x))
      = F_1(h(x),x,g(x)),
    \end{align*}
    for all $x \in \tilde U$, and so this forces $h \equiv 0$ as desired.
  \end{proof}
\end{lemma}

\subsection{Existence of  vortex patch equilibria}\label{sec:proof}
We shall give in this subsection a detailed statement of Theorem \ref{thm:general}, as well as a proof based on Lemma~\ref{abstract-lemma}. The stationary case is more degenerate than the rigid motion case due to the additional symmetries in the problem, and so we will treat it separately.

For rigidly rotating or translating vortex patch solutions, our main result is the following.
\begin{theorem} \label{prop:ift}
Let $\alpha\in[0,1)$ and let $\lambda^*$ be a non-degenerate solution, in the sense of Definition~\ref{def:non-deg}{\rm (i)}, to the $N$-vortex problem \eqref{alg-sysP}, with one of $\Omega,U$ nonzero.
Then the following hold true.
\begin{enumerate}[label=\rm(\roman*)]
\item There exists $\varepsilon_1>0$ and a unique $C^1$ function $(f,\lambda_1)\colon (-\varepsilon_1,\varepsilon_1)\longrightarrow B_1^\alpha \times \R^{2N-1}$ satisfying
\begin{equation}\label{sol_g-alpha}
\mathcal{G}^\alpha \big(\varepsilon,f(\varepsilon); \lambda_1(\varepsilon), \lambda_2^*\big)=0,
\end{equation}
with $\lambda_1(\varepsilon)=\lambda_1^* +o(\varepsilon)$ and 
 $$
 f_j(\varepsilon,w)=\varepsilon {b_j}\Xi_\alpha\sum_{k=1, k\neq j}^N\frac{\gamma_k}{\gamma_j} \frac{(\overline{z_k}-\overline{z_j})^2}{ |z_k -z_j|^{\alpha+4}} \overline w
+o(\varepsilon), \quad \Xi_\alpha:=\frac{(\alpha+2)\Gamma(1-\frac\alpha2)\Gamma(3-\frac\alpha2)}{4\Gamma(2-\alpha)}.
  $$ 
\item These solutions enjoy the symmetries
  \begin{align*}
    f(-\varepsilon)(w) = f(\varepsilon)(-w),
    \qquad 
    \lambda_1(-\varepsilon) = \lambda_1(\varepsilon).
  \end{align*}
\item For all  $\varepsilon\in (-\varepsilon_1,\varepsilon_1)\setminus\{0\}$  the domains $\mathcal{O}_j^\varepsilon$, whose boundaries are given by the conformal parametrizations $\phi_j^\varepsilon=\operatorname{Id}+\varepsilon|\varepsilon|^\alpha b_j^{1+\alpha} f_j\colon\mathbb{T}\to \partial \mathcal{O}_j^\varepsilon$,   are
 strictly convex.

\end{enumerate}
\end{theorem}
\begin{proof}
In view of Proposition~\ref{proposition2}, for any $h\in \mathcal{V}_0^\alpha$ and $\dot\lambda_1\in \R^{2N-1}$, we have 
\begin{equation}\label{diff-g-gen}
D_{(f;\lambda_1)}\mathcal{G}^\alpha(0,0;\lambda^*)\begin{pmatrix} h\\ \dot\lambda_1\end{pmatrix}(w)=
 D_{f}\mathcal{G}^\alpha(0,0;\lambda^*)h(w)+ \Im \big\{
 D_{\lambda_1}\mathcal{P}_j^\alpha(\lambda^*)\dot\lambda_1
   \overline{w}\big\},
\end{equation}
where $D_{f}\mathcal{G}^\alpha (0,0;\lambda^*)$ is an isomorphism from  $ \mathcal{V}_0^\alpha$ to $  \widetilde{\mathcal{W}}_0^\alpha$. From the hypothesis on the matrix $D_{\lambda_1}\mathcal{P}^\alpha(\lambda^*)$, the second linear operator $D_{\lambda_1}\mathcal G^\alpha(0,0;\lambda^*)$ on the right hand side has a trivial kernel and 
$$\ran [D_{\lambda_1}\mathcal G^\alpha(0,0;\lambda^*)]\subset \mathbb{W}:=\big\{w\mapsto \Im[c_1 w] : c_1\in \mathbb{C}^N\big\}$$ is codimension $1$.
Moreover, it is easy to see that
\begin{equation}\label{www}
{\mathcal{W}}_0^\alpha=\widetilde{\mathcal{W}}_0^\alpha \oplus \mathbb{W}. 
\end{equation}
Thus, one has
\begin{equation}\label{ker-ran-DG}
\codim\ran D_{(f;\lambda_1)}\mathcal{G}^\alpha(0,0;\lambda^*) = 1 \quad \textnormal{and} \quad  \ker D_{(f;\lambda_1)}\mathcal{G}^\alpha(0,0;\lambda^*)=\{ 0\}.
\end{equation} 
In the case of pure translation ($\Omega=0$ and $U\neq 0$) we set
 \begin{equation}\label{PHI1}
\Phi(\varepsilon,f, g;\lambda):={\rm Im}\Big\{\sum_{j=1}^N\frac{\gamma_j}{\pi} \int_{\mathbb{T}}g_j(w)\big(1+\varepsilon b_j \overline{w} f_j(w)\big)dw\Big\}
\end{equation}
while in the case of pure rotation ($\Omega \neq 0$ and $U=0$) we instead set
\begin{align}\label{PHI2}
\Phi(\varepsilon,f, g;\lambda):={\rm Im}\Big\{\sum_{j=1}^N\frac{\gamma_j}{\pi}\int_{\mathbb{T}}g_j(w) \Big(\varepsilon b_j|w+\varepsilon b_j f_j(w)|^2+\varepsilon b_j\Re[\overline{z_j}f_j(w)]+\Re[\overline{z_j}w]\Big)\overline{w}dw\Big\}
\end{align}
with $g=(g_1,\ldots, g_N)\in \mathcal{W}^\alpha_0$. It is clear that the mapping  $\Phi \colon(-\varepsilon_0,\varepsilon_0)\times B_0^\alpha \times \mathcal{B}_0^\alpha\times \Lambda\to \mathbb{R}$ is $C^1$ and that for all $(\varepsilon,f,\lambda) \in (-\varepsilon_0,\varepsilon_0)\times B_0^\alpha\times\Lambda$ one has
\begin{align} \label{condition 1}
\Phi\big(\varepsilon,f,  0;\lambda\big)&=0.
 \end{align} 
 Moreover, by \eqref{PHI1}--\eqref{PHI2} and  Lemma~\ref{coro:idG}{\rm (i)}, we have
\begin{align} \label{condition 2}
\Phi\big(\varepsilon,f,  \mathcal{G}^\alpha(\varepsilon,f;\lambda);\lambda\big)&=0.
 \end{align}
By differentiating  \eqref{PHI1}--\eqref{PHI2} with respect to $g$ in the direction $\tilde g=(\tilde g_1,\ldots,\tilde g_N)\in \mathcal W_0^\alpha$, 
\begin{equation}\label{expang}
\tilde g_j(w)=\sum_{n\neq 0}
c_{n,j}\,{w}^n, \; c_{-n,j}=\overline{c_{n,j}}\in \mathbb{C}, \; j=1,\ldots, N,
\end{equation}
 we get, for $\Omega =0$ and $U\neq 0$, 
\begin{align*}
 D_{g}\Phi (0,0,0;\lambda)\tilde g&={\rm Im}\Big\{\sum_{j=1}^N\frac{\gamma_j}{\pi} \int_{\mathbb{T}}\tilde g_j(w)dw\Big\}=-2\sum_{j=1}^N{\gamma_j} {\rm Re}[ c_{1,j}]
 \end{align*}
 and, for  $\Omega \neq 0$ and $U=0$, 
 \begin{align*}
 D_{g}\Phi (0,0,0;\lambda)\tilde g &={\rm Im}\Big\{\sum_{j=1}^N\frac{\gamma_j}{\pi}\int_{\mathbb{T}}\tilde g_j(w) \Re[\overline{z_j}w]\overline{w}dw\Big\}=2\sum_{j=1}^N{\gamma_j}\Re[{z_j}c_{1,j}].
\end{align*}
In either case, we easily check that
\begin{align} \label{condition 3}
\ran\big[ D_{g}\Phi (0,0,0;\lambda^*)\big]&=\mathbb{R}.
 \end{align} 
Consequently, the existence and uniqueness in {\rm(i)} follow from \eqref{ker-ran-DG}--\eqref{condition 3} and Lemma~\ref{abstract-lemma}. 
 
Next, differentiating \eqref{sol_g-alpha} with respect to $\varepsilon$ at the point $(0,0; \lambda^*)$  we get
\begin{align}\label{diff-g-gen-2}
D_{(f;\lambda_1)}\mathcal{G}^\alpha(0,0;\lambda^*)\partial_\varepsilon \big( f(\varepsilon);\lambda(\varepsilon)\big)\Big|_{\varepsilon=0}=-\partial_\varepsilon \mathcal{G}^\alpha \big(0,0; \lambda^*\big).
\end{align}
In view of  \eqref{Fjsqgalpha0}, for all $\alpha\in(0,1)$  we have
\begin{align}\label{f0dif-eps}
\partial_\varepsilon\mathcal{G}^\alpha_j(0,0;\lambda^*)(w)  &= \frac{\alpha}{2}\left(\frac{\alpha}{2}+1\right)  C_\alpha \sum_{k=1, k\neq j}^N{\gamma_k} b_j\Im \Big\{
   \frac{(z_k-z_j)^2}{ |z_k -z_j|^{\alpha+4}}  \overline{ w}^2 \Big\}
\end{align}
and, by \eqref{rotn+100}, for $\alpha=0$ we get
\begin{align}\label{f0dif-eps2}
  \nonumber\partial_\varepsilon\mathcal{G}^\alpha_j(0,0;\lambda^*)(w)
    &= -\frac12 \sum_{k=1, k\neq j}^N{\gamma_k} b_j\Im \Big\{
    \frac{{ w}^2}{(z_k-z_j)^2}  \Big\}.
\end{align}
Thus, for all $\alpha\in[0,1)$ we have   
$$\partial_\varepsilon\mathcal{G}^\alpha_j(0,0;\lambda^*)\in\widetilde{\mathcal{W}}_0^\alpha.$$ Since the linear operator 
$D_{f}\mathcal{G}^\alpha (0,0;\lambda^*)\colon \mathcal{V}_0^\alpha \to  \widetilde{\mathcal{W}}_0^\alpha$ is an isomorphism and, by hypothesis, the kernel of the operator   $D_{\lambda_1}\mathcal{P}_j^\alpha(\lambda^*)$ is trivial, combining \eqref{diff-g-gen}, \eqref{diff-g-gen-2}, \eqref{f0dif-eps} and Proposition~\ref{proposition2}{\rm (iii)}  we conclude  that
$$
\partial_\varepsilon  \lambda(\varepsilon)\big|_{\varepsilon=0}=0\quad \textnormal{and} \quad  \partial_\varepsilon  f_j(\varepsilon)\big|_{\varepsilon=0}(w)=\begin{cases}
   \displaystyle \big(\tfrac{\alpha}{2}+1\big) \frac{\alpha C_\alpha}{2M_1^\alpha} \sum_{k=1, k\neq j}^N\frac{\gamma_k}{\gamma_j} \frac{{b_j}(\overline{z_k}-\overline{z_j})^2}{ |z_k -z_j|^{\alpha+4}} \overline w & \text{if } \alpha\in(0,1), \\ 
      \displaystyle \sum_{k=1, k\neq j}^N\frac{\gamma_k}{\gamma_j} \frac{{b_j}}{(\overline{z_k}-\overline{z_j})^2} \overline w & \text{if } \alpha=0. 
  \end{cases} 
$$
Finally, straightforward computations yield
\begin{equation}\label{gam1hatg}
\frac{\alpha {C_\alpha}}{M_1^\alpha}=\frac{\Gamma(1-\frac\alpha2)\Gamma(3-\frac\alpha2)}{\Gamma(2-\alpha)},
\end{equation}
completing the proof of {\rm (i)}.

\vspace{0.2cm}

  By the uniqueness in {\rm(i)}, in order to prove {\rm(ii)} it suffices to show that
 \begin{equation}\label{g-eps}
\begin{split}
 &\mathcal{G}^\alpha_j(\varepsilon,f;\lambda)(-w)= \mathcal{G}^\alpha_j(-\varepsilon,\tilde{f};\lambda)(w),
\end{split}
\end{equation}
where $\tilde f(w)=f(-w)$. From \eqref{exp:Galpha} one has
\begin{equation}\label{g-eps2}
\begin{split}
 &\mathcal{G}^\alpha_j(-\varepsilon,\tilde{f};\lambda)(w)= \Im \Big\{
 \Big(
  \Omega\big(-\varepsilon b_j w+\varepsilon^2|\varepsilon|^\alpha b_j^{2+\alpha}  \tilde{f}_j(w)+z_j \big)+U
 \\
 & +\gamma_j \,\mathcal{I}^\alpha[-\varepsilon, \tilde{f}_j](w)  + \sum_{k=1, k\neq j}^N\gamma_k \mathcal{J}^\alpha_k[\varepsilon,f_k,f_j](w) \Big)
 \overline{ w}\big(1-\varepsilon|\varepsilon|^\alpha b_1^{1+\alpha} \overline{\tilde{f}_j'(w)} \big)
  -\mu_\alpha\, \gamma_j \overline{\tilde{f}_j'(w)}\Big\}.
\end{split}
\end{equation}
Since $\tilde{f}'_j(w) = -f_j'(-w)$, by \eqref{Ialpha0}  we have
\begin{equation*}
\begin{split}
\mathcal{I}^\alpha[-\varepsilon, \tilde{f}_j](w)&:= C_{\alpha} \fint_\T 
  \frac{f_j'(-\tau)}{|\tau-w-\varepsilon|\varepsilon|^\alpha b_j^{1+\alpha}\big( f_j(-\tau)-f_j(-w)\big)|^\alpha}\, d\tau
  \\
   &\qquad +\alpha C_{\alpha} \fint_\T\int_0^1 
  \frac{\Re\big[\big(f_j(-\tau)-f_j(-w)\big)(
 \overline{\tau} - \overline{w})\big]-\varepsilon|\varepsilon|^\alpha t|f_j(-\tau)-f_j(-w)|^2}{|\tau-w-t\varepsilon|\varepsilon|^\alpha b_j^{1+\alpha}\big(f_j(-\tau)+f_j(-w)\big)|^{2+\alpha}}dt\, d\tau.
\end{split}
\end{equation*}
Making the change of variable $\tau\mapsto -\tau$ we get
\begin{equation*}
\begin{split}
\mathcal{I}^\alpha[-\varepsilon, \tilde{f}_j](w)&= -C_{\alpha} \fint_\T 
  \frac{f_j'(\tau)}{|\tau+w+\varepsilon|\varepsilon|^\alpha b_j^{1+\alpha}\big( f_j(\tau)-f_j(-w)\big)|^\alpha}\, d\tau
  \\
   &\qquad+\alpha C_{\alpha} \fint_\T\int_0^1 
  \frac{\Re\big[\big(f_j(\tau)-f_j(-w)\big)(
 \overline{\tau} + \overline{w})\big]+\varepsilon|\varepsilon|^\alpha t|f_j(\tau)-f_j(-w)|^2 }{|w+\tau+t\varepsilon|\varepsilon|^\alpha b_j^{1+\alpha}\big(f_j(\tau)+f_j(-w)\big)|^{2+\alpha}}dt\, d\tau 
 \\
  &=\mathcal{I}^\alpha[\varepsilon, {f}_j](-w).
\end{split}
\end{equation*}
In a similar way we can check that
$$
\mathcal{J}^\alpha_{k}[-\varepsilon,\tilde{f}_{k},\tilde{f}_j;\lambda](w) =\mathcal{J}^\alpha_{k}[\varepsilon,{f}_{k},{f}_j;\lambda](-w). 
$$
Inserting the two last identities into \eqref{g-eps2} and using the fact that $\tilde{f}_j(w) = f_j(-w)$ yields
\eqref{g-eps} as desired.
 
\vspace{0.2cm}

As mentioned in Remark~\ref{rk:general}, with only minor modifications the above proof still holds when $V_0^\alpha$ and $W_0^\alpha$ in \eqref{Valpha0} are replaced by spaces with higher H\"older regularity $C^{n+1+\eta_\alpha}$ and $C^{n+\eta_\alpha}$ for any fixed $n \in \N$, at the cost of possibly shrinking $\varepsilon_1$. In particular, by uniqueness we may assume that $f_j$ and hence $\phi_j$ are $C^2$, which allows us to prove the convexity of the domains $\mathcal{O}_j^\varepsilon$ by following the same argument as in \cite{Hmidi-Mateu}. 

Recall that the curvature can be expressed, in terms of the conformal mapping, by the formula
$$
\kappa(w) = \frac{1}{|\phi_j^\varepsilon(w)|}
\Re\Big(1+w\frac{\phi_j''(w)}{\phi_j'(w)} \Big).
$$
As 
$
\phi_j(w) = w + \varepsilon|\varepsilon|^\alpha b_j^{1+\alpha}
f_j(w)
$,
we easily verify that
$$
1+w\frac{\phi_j''(w)}{\phi_j'(w)}=1+O(\varepsilon|\varepsilon|^\alpha),
$$
uniformly in $w$.
Thus the curvature is strictly positive and therefore the
domain $\mathcal{O}_j^\varepsilon$ is strictly convex.
\end{proof}

Now, we treat the stationary case, where $\Omega=U=0$. We have the following result.
\begin{theorem}\label{theorem-Phi-stationary}
  Let $\alpha\in[0,1)$ and  let $\lambda^*$ be a non-degenerate solution, in the sense of Definition~\ref{def:non-deg}{\rm (ii)},  to the $N$-vortex problem \eqref{alg-sysP} with $\Omega=U=0$. 
  Then the conclusions of Theorem~\ref{prop:ift} hold, except that now $\lambda_1(\varepsilon)$ takes values in $\R^{2N-3}$ rather than $\R^{2N-1}$.
\end{theorem}
\begin{proof} We shall only give the proof of the existence and uniqueness of {\rm(i)}.
The proof of the asymptotic expansion and  {\rm(ii)}--{\rm(iii)} follow the same lines of Theorem~\ref{prop:ift}. 

From Proposition~\ref{proposition2}, the hypothesis on the matrix $D_{\lambda_1}\mathcal{P}^\alpha(\lambda^*)$, \eqref{diff-g-gen} and \eqref{www} we conclude that
\begin{equation}\label{ker-ran-DG2}
\codim\ran D_{(f;\lambda_1)}\mathcal{G}^\alpha(0,0;\lambda^*) = 3 \quad \textnormal{and} \quad  \ker D_{(f;\lambda_1)}\mathcal{G}^\alpha(0,0;\lambda^*)=\{ 0\}.
\end{equation} 
For all $(\varepsilon,f,\lambda) \in (-\varepsilon_0,\varepsilon_0)\times B_0^\alpha\times\Lambda$, we set
\begin{align*}
\widetilde\Phi(\varepsilon,f, g;\lambda):=\sum_{j=1}^N\frac{\gamma_j}{\pi}\begin{pmatrix}
{\rm Re}\big\{ \int_{\mathbb{T}}g_j(w)\big(1+\varepsilon b_j \overline{w} f_j(w)\big)dw\big\}
\\
{\rm Im}\big\{ \int_{\mathbb{T}}g_j(w)\big(1+\varepsilon b_j \overline{w} f_j(w)\big)dw\big\}
\\{\rm Im}\big\{\int_{\mathbb{T}}g_j(w) \big(\varepsilon b_j|w+\varepsilon b_j f_j(w)|^2+\varepsilon b_j\Re\big[\overline{z_j}f_j(w)\big]+\Re[\overline{z_j}w]\big)\overline{w}dw\big\}
\end{pmatrix}.
\end{align*}
The mapping  $\widetilde\Phi \colon(-\varepsilon_0,\varepsilon_0)\times B_0^\alpha \times \mathcal{B}_0^\alpha\times \Lambda\to \mathbb{R}^3$ is $C^1$ and satisfies 
\begin{align} \label{condition 12}
\widetilde \Phi\big(\varepsilon,f,  0;\lambda\big)&=0.
 \end{align} 
 Moreover, by   Lemma~\ref{coro:idG}, we have
\begin{align} \label{condition 22}
\widetilde\Phi\big(\varepsilon,f,  \mathcal{G}^\alpha(\varepsilon,f;\lambda);\lambda\big)&=0.
 \end{align}
Differentiating  $\widetilde\Phi$ with respect to $g$ in the direction $\tilde g=(\tilde g_1,\ldots,\tilde g_N)\in \mathcal W_0^\alpha$ in \eqref{expang}
gives
\begin{align*}
 D_{g}\widetilde\Phi (0,0,0;\lambda)\tilde g&={\rm Im}\Big\{\sum_{j=1}^N\frac{\gamma_j}{\pi} \int_{\mathbb{T}}\tilde g_j(w)dw\Big\}=-2\sum_{j=1}^N{\gamma_j} \begin{pmatrix} {\rm Im}[ c_{1,j}]\\ {\rm Re}[ c_{1,j}]\\ y_j {\rm Im}[ c_{1,j}]-x_j\Re[c_{1,j}]\end{pmatrix}.
 \end{align*}
We can easily check that
\begin{align} \label{condition 32}
\ran\big[ D_{g}\widetilde\Phi (0,0,0;\lambda^*)\big]&=\mathbb{R}^3.
 \end{align} 
 Thus,  from \eqref{ker-ran-DG2}--\eqref{condition 32} and using Lemma~\ref{abstract-lemma} we conclude the desired result.
\end{proof}

\begin{remark}\label{rem:bift}
  In this section we have suppressed the dependence of $\mathcal G^\alpha$ on the parameters $b_j \in (0,\infty)$. Just as with $\varepsilon$, one can check that $\mathcal G^\alpha$ is in fact $C^1$ in $b_j$. This is true even for $b_j = 0$, corresponding to the case where $j$-th point vortex is not desingularized into a vortex patch but instead remains a point vortex. Applying Lemma~\ref{abstract-lemma} as in the proof of Theorem~\ref{prop:ift}, one obtains families of solutions made up of a combination of point vortices and small vortex patches. The same can be done in the examples of Sections~\ref{sec:asym} and \ref{sec:nest-poly} below.
\end{remark}

\section{Examples of asymmetric vortex equilibria}\label{sec:asym}
 
In this section we shall give some explicit examples of point vortex solutions to the $N$-vortex problem \eqref{alg-sysP} satisfying the non-degeneracy condition in Definition~\ref{def:non-deg}. In particular, we prove Theorem \ref{thm:informal-pair} by simply applying Theorems~\ref{prop:ift} and \ref{theorem-Phi-stationary}.

\vspace{0.2cm}

\paragraph{\bf Asymmetric  co-rotating pairs.} Set $N=2$  and consider the rotating solution $\lambda^*$ given by \eqref{v-pair-rot}  
to the $N$-vortex problem \eqref{alg-sysP}. 
The differential of the mapping 
$${\mathcal{P}}^{\alpha}:=\bigl(\Re[\mathcal{P}_1^\alpha],\Im[\mathcal{P}_1^\alpha],\Re[\mathcal{P}_2^\alpha],\Im[\mathcal{P}_2^\alpha]\big)$$ with respect to $\lambda_1=(x_1,x_2,y_2)$ at the point $\lambda^*$ is 
\begin{equation*}
D_{\lambda_1}{\mathcal{P}}^{\alpha}(\lambda^*)\begin{pmatrix}
\dot x_1\\
\dot x_2\\
\dot y_2
\end{pmatrix}=\frac{\gamma \widehat{C}_\alpha}{2d^{\alpha+2}(1+\mathtt{c})^{\alpha+2}}
\begin{pmatrix}
2+\mathtt{c}+\alpha & -(\alpha+1)  & 0\\
0 & 0 & 1\\
-\mathtt{c}(\alpha+1)&1+\mathtt{c}(2+\alpha)  & 0 \\
0 & 0 & 1  
\end{pmatrix}\begin{pmatrix}
\dot x_1\\
\dot x_2\\
\dot y_2
\end{pmatrix}.
\end{equation*}
Eliminating the second row we get a matrix with Jacobian determinant $(\alpha+2)(1+\mathtt{c})^2$, which is nonzero if $\mathtt{c}\neq -1$. Thus, this matrix has rank 3, which implies that the kernel is trivial and the image has codimension $1$. Therefore, Theorem~\ref{prop:ift} applies yielding the existence of $\varepsilon_1>0$ and a unique $C^1$ function $(f;\lambda_1)=(f_1,f_2;x_1,x_2,y_2)\colon (-\varepsilon_1,\varepsilon_1)\longrightarrow B_1^\alpha \times \R^{3}$ satisfying
\begin{equation}\label{sol_g-alpha-pair}
\mathcal{G}^\alpha \big(\varepsilon,f(\varepsilon); \lambda_1(\varepsilon), \lambda_2^*\big)=0.
\end{equation}
  It remains only to check the reflection symmetry property. From \eqref{exp:Galpha}, one has
\begin{equation*}
\begin{split}
\mathcal{G}^\alpha_{j}(\varepsilon,f;&\lambda_1,\lambda_2)(\overline{w}):=-\Im \Big\{
 \Big( \Omega\big(\varepsilon  b_j w+\varepsilon^2|\varepsilon|^\alpha  b_j^{2+\alpha} \overline{ f_j(\overline{w})}+\overline{ z_j}\big)\Big)
 \overline{ w}\big(1+\varepsilon|\varepsilon|^\alpha b_j^{1+\alpha} {f_j'(\overline{w})} \big)
\\
&-\mu_\alpha\, \gamma_j {f_j'(\overline{w})}+
\Big( \gamma_j \,\overline{\mathcal{I}^\alpha[\varepsilon, f_j](\overline{w})}+\gamma_{3-j} \overline{\mathcal{J}^\alpha_{3-j}[\varepsilon,f_k,f_j](\overline{w})} \Big)\overline{ w}\big(1+\varepsilon|\varepsilon|^\alpha b_j^{1+\alpha} {f_j'(\overline{w})} \big)\Big\}.
\end{split}
\end{equation*}
Set 
\begin{equation}\label{ref-sym}
\tilde{f_j}(w):=\overline{f_j(\overline{w})}\quad\textnormal{and}\quad \tilde{\lambda}_1=(x_1,x_2,-y_2). 
\end{equation}
Since $\tilde{f}'_j(w) = \overline{f_j'(\overline{w})}$ then we can easily check, from \eqref{I00}--\eqref{J00}  and   \eqref{Ialpha0}--\eqref{Jalpha0}, that 
$$
\overline{\mathcal{I}^\alpha[\varepsilon, f_j](\overline{w})}=\mathcal{I}^\alpha[\varepsilon, \tilde f_j](w)\quad\textnormal{and}\quad \overline{\mathcal{J}^\alpha_{3-j}[\varepsilon,f_{3-j},f_j;\lambda_1,\lambda_2](\overline{w}) }=\mathcal{J}^\alpha_{3-j}[\varepsilon,\tilde f_{3-j},\tilde f_j;\tilde{\lambda}_1,\lambda_2](w). 
$$
It follows that
\begin{equation}\label{ref-sym2-pair}
\mathcal{G}_j^\alpha (\varepsilon,f;\lambda_1,\lambda_2)(\overline{w})=-\mathcal{G}_j^\alpha (\varepsilon,\tilde f, \tilde \lambda_1,\lambda_2)({w}). 
\end{equation}
Since $\lambda_1^*(0)=\tilde \lambda_1^*(0)$, by uniqueness of the solution of \eqref{sol_g-alpha-pair} we conclude that
$$
\overline{f_j(w)}={f_j(\overline{w})}\quad\textnormal{and}\quad y_2=0,
$$
which implies  that the Fourier coefficients of $f_j\in V_0^\alpha$ are real
and the domain associated to the conformal mapping $\phi_j=\operatorname{Id}+\varepsilon|\varepsilon|^\alpha b_j^{\alpha+1} f_j$   is symmetric with respect to the real axis.

\vspace{0.2cm}

\paragraph{\bf Asymmetric  counter-rotating pairs.} Set  $N=2$ and consider the translating solution $\lambda^*$ given by  \eqref{v-pair-trans}
to the $N$-vortex problem \eqref{alg-sysP}.
The differential of the mapping 
$${\mathcal{P}}^{\alpha}:=\bigl(\Re[\mathcal{P}_1^\alpha],\Im[\mathcal{P}_1^\alpha],\Re[\mathcal{P}_2^\alpha],\Im[\mathcal{P}_2^\alpha]\big)$$ with respect to $\lambda_1=(x_2,y_2,\gamma_1)$ is 
\begin{equation*}
D_{\lambda_1}{\mathcal{P}}^{\alpha}(\lambda^*)\begin{pmatrix}
\dot x_2\\
\dot y_2\\
\dot \gamma_1
\end{pmatrix}=\frac{\widehat{C}_\alpha}{2^{\alpha+3}}\frac{1}{d^{\alpha+2}}\
\begin{pmatrix}
- \gamma(\alpha+1) &0 & 0\\
0 & \gamma & 0\\
 -\gamma(\alpha+1) & 0 & 2d \\
0 & \gamma & 0  
\end{pmatrix}\begin{pmatrix}
\dot x_2\\
\dot y_2\\
\dot \gamma_1
\end{pmatrix}
\end{equation*}
By eliminating the last line we get a matrix with Jacobian determinant $2\gamma d(\alpha+1)$, which is nonzero. Therefore, this matrix has rank 3, which implies that the kernel is trivial and the image has codimension one. Hence, the existence of counter-rotating vortex patch pair follows from Theorem~\ref{prop:ift}. The reflection symmetry property can be checked  similarly to the co-rotating case.

\vspace{0.2cm}

\paragraph{\bf Stationary tripole} Set $N=3$ in \eqref{alg-sysP} and  consider the stationary tripole  $\lambda^*$ given by \eqref{stationary-tripole}. The differential of the mapping $${\mathcal{P}}^{\alpha}:=\bigl(\Re[\mathcal{P}_1^\alpha],\Im[\mathcal{P}_1^\alpha],\Re[\mathcal{P}_2^\alpha],\Im[\mathcal{P}_2^\alpha],\Re[\mathcal{P}_3^\alpha],\Im[\mathcal{P}_3^\alpha]\big)$$ 
with respect to $\lambda_1=(x_3,y_3,\gamma_2)$ at the point $\lambda^*$ is
\begin{equation*}
D_{\lambda_1}{\mathcal{P}}^{\alpha}(\lambda^*)\begin{pmatrix}
\dot x_3\\
\dot y_3\\
\dot \gamma_2
\end{pmatrix}=\frac{\gamma \widehat{C}_\alpha}{2}\
\begin{pmatrix}
-(\alpha+1)\mathtt{a}^{\alpha+1}{(\mathtt{a}+1)^{-\alpha-2}} &0 & -1/\gamma\\
0 & \mathtt{a}^{\alpha+1}{(\mathtt{a}+1)^{-\alpha-2}}& 0\\
-(\alpha+1)\mathtt{a}^{-1} & 0 & 0 \\
0 & \mathtt{a}^{-1} & 0  \\
-(\alpha+1)\mathtt{a}^{-1}{(\mathtt{a}+1)^{-\alpha-2}}  & 0 & {\mathtt{a}^{-\alpha-1}}/\gamma\\
0 & \mathtt{a}^{-1}{(\mathtt{a}+1)^{-\alpha-2}} & 0\\
\end{pmatrix}\begin{pmatrix}
\dot x_3\\
\dot y_3\\
\dot \gamma_2
\end{pmatrix},
\end{equation*}
which has rank 3. Thus, Theorem~\ref{theorem-Phi-stationary} guarantees the existence of stationary vortex patch tripole. The reflection symmetry property with respect to the real axis  can be checked  similarly to the co-rotating pairs.

\section{Nested polygonal vortex patch equilibria}\label{sec:nest-poly}
In this section we shall construct $2m+1$ multipolar vortex equilibria in
which a central patch is surrounded by $2m$  satellite
patches centered at the vertices of two nested regular $m$-gons. The vertices of the polygons are either radially aligned with each other or out of phase by an angle $\pi/m$, and the patches on each polygon are identical with the same strength; see Figure~\ref{fig:poly}.  

More precisely, we shall desingularize the  
following system of point vortices  
\begin{equation}\label{q0}
\omega^0_0(z)=\pi\Big(\gamma_0\delta_{z_0(0)}(z)+\gamma_1\sum_{k=0}^{m-1}\delta_{z_{1k}(0)}(z)+\gamma_2\sum_{k=0}^{m-1}\delta_{z_{2k}(0)}(z)\Big)
\end{equation}
with 
\begin{equation}\label{z}
  z_0(0)=0\quad\textnormal{and}\quad  z_{jk}(0):=
  \begin{cases}
    d_1e^{\frac{2k\pi i}{m}} &\, \textnormal{if }\, j=1 \, \textnormal{and }\, 0\leq k\leq m-1,\\
    d_2e^{\frac{(2k+\vartheta)\pi i}{m}} &\, \textnormal{if }\, j=2 \, \textnormal{and }\, 0\leq k\leq m-1,
  \end{cases}
\end{equation}
  where  $\gamma_0,\gamma_1,\gamma_2\in \R\setminus\{0\}$,  $d_2>d_1>0$ and $\vartheta=0$ corresponds to  the aligned configuration    while  $\vartheta=1$ refers to  the staggered configuration.  
  Assuming that $z_0(t)=z_0(0)$ and  $z_{jk}(t) = e^{i\Omega t}z_{jk}(0)$, one may easily check that the system of $2m+1$ equations in  \eqref{alg-sysP} can be reduced to  
\begin{equation}\label{gG0}
\begin{split}
\gamma_j{\mathcal{F}}_j^{\alpha}(\lambda)&=0,\quad j=0,1,2,
\end{split}
\end{equation}
where $\lambda=(\Omega,\gamma_2)$ and 
\begin{equation}\label{p-gG2-012}
\begin{split}
{\mathcal{F}}_0^{\alpha}(\lambda)&:=\frac{\widehat{C}_\alpha}{2}\Big(\frac{\gamma_1}{d_1^{1+\alpha}}+\frac{\gamma_2}{d_2^{1+\alpha}}e^{\frac{\vartheta\pi i}{m}}\Big)\sum_{k=0}^{m-1}e^{\frac{2k\pi i}{m}},
\\
{\mathcal{F}}_1^{\alpha}(\lambda)&:=\Omega d_1-\frac{\widehat{C}_\alpha}{2d_1^{1+\alpha}}\Big({\gamma_0}+{\gamma_1}\sum_{k=1}^{m-1}\frac{1-e^{\frac{2k\pi i }{m}} }{\big|1 -e^{\frac{2k\pi i }{m}} \big|^{2+\alpha}}+\gamma_2\sum_{k=0}^{m-1}\frac{1 -e^{\frac{(2k+\vartheta)\pi i }{m}} d  }{\big|1 -e^{\frac{(2k+\vartheta)\pi i }{m}} d  \big|^{2+\alpha}}\Big),\quad 
\\
{\mathcal{F}}_2^{\alpha}(\lambda)&:=\Omega d_2-\frac{\widehat{C}_\alpha}{2d_2^{1+\alpha}}\Big({\gamma_0}+\gamma_1\sum_{k=0}^{m-1}\frac{1 -e^{\frac{(2k-\vartheta)\pi i }{m}}  d^{-1} }{\big|1 -e^{\frac{(2k-\vartheta)\pi i }{m}} d ^{-1} \big|^{2+\alpha}}+{\gamma_2}\sum_{k=1}^{m-1}\frac{1-e^{\frac{2k\pi i }{m}} }{\big|1 -e^{\frac{2k\pi i }{m}} \big|^{2+\alpha}}\Big),
\end{split}
\end{equation}
with $d:=d_2/d_1$.
By symmetry arguments, one may easily check that
\begin{gather*}
\sum_{k=0}^{m-1} 
  {e^{\frac{2k\pi i}{m}} } =0, \qquad \qquad\quad\; \sum_{k=1}^{m-1}\frac{1-e^{\frac{2k\pi i }{m}} }{\big|1 -e^{\frac{2k\pi i }{m}} \big|^{2+\alpha}}
=\frac12\sum_{k=1}^{m-1}{ \Big(2\sin\big({\frac{k\pi  }{m}}\big)  \Big)^{-\alpha}}
=:\frac{S_\alpha}{2},\\
\sum_{k=0}^{m-1}\frac{1 -  d ^{\pm 1} e^{\frac{(2k\pm\vartheta)\pi i}{m}}}{\big|1 -  d ^{\pm 1} e^{\frac{(2k\pm\vartheta)\pi i}{m}} \big|^{\alpha+2}}
=\sum_{k=0}^{m-1}\frac{1 -  d ^{\pm 1} \cos\big({\frac{(2k\pm\vartheta)\pi }{m}}\big)}{\big(1+( d ^{\pm 1})^2 - 2 d ^{\pm 1} \cos\big({\frac{(2k\pm\vartheta)\pi }{m}}\big) \big)^{\frac\alpha2+1}}
=:{T_\alpha^\pm( d ,\vartheta)}.\label{t}
 \end{gather*}
 Thus, the identities in  \eqref{p-gG2-012} become
  \begin{equation}\label{p-gG2}
 \begin{split}
{\mathcal{F}}_0^{\alpha}(\lambda)&=0,
\\
{\mathcal{F}}_1^{\alpha}(\lambda)&=\Omega d_1 -\frac{\widehat{C}_\alpha}{2d_1^{\alpha+1}}\Big[\gamma_0+\frac{\gamma_1}{2}S_\alpha +\gamma_2T_\alpha^+( d ,\vartheta)\Big],
\\
{\mathcal{F}}_2^{\alpha}(\lambda)&=\Omega d_2 -\frac{\widehat{C}_\alpha}{2d_2^{\alpha+1}}\Big[\gamma_0+\gamma_1T_\alpha^-( d ,\vartheta)+\frac{\gamma_2}{2}S_\alpha \Big].
\end{split}
\end{equation}
Moreover, the differential of the mapping ${\mathcal{Q}}^{\alpha}:=({\mathcal{Q}}_1^{\alpha},{\mathcal{Q}}_2^{\alpha})$ with respect to $\lambda=(\Omega,\gamma_2)$ is given by
\begin{equation}\label{jacob-mat}
D_{\lambda}{\mathcal{F}}^{\alpha}(\lambda)\begin{pmatrix} \dot\Omega\\ \dot \gamma_2\end{pmatrix} =\begin{pmatrix}
d_1 &  -\frac{\widehat{C}_\alpha }{2}\frac{ T_\alpha^+( d ,\vartheta)}{d_1^{1+\alpha}}\\
 d_2 &  -\frac{\widehat{C}_\alpha }{2} \frac{ S_\alpha}{2d_2^{1+\alpha}}
\end{pmatrix}
\begin{pmatrix} \dot\Omega\\ \dot \gamma_2\end{pmatrix}.
\end{equation}

If the Jacobian determinant is non-trivial, 
\begin{equation}\label{det-j-poly}
\det\big(D_\lambda \mathcal F^\alpha(\lambda)\big)=
-\frac{\widehat{C}_\alpha d_1}{2d_2^{\alpha+1}}\Big(\frac{ S_\alpha}{2}- d^{\alpha+2} T_\alpha^+( d ,\vartheta)\Big)\neq 0,
\end{equation}
then  the system \eqref{p-gG2} has a unique solution $\lambda^*=( \Omega^*,\gamma_2^*)$  given by  
\begin{equation}\label{gamma0-om0}
 \begin{split}
\gamma_2^*&:=\frac{\big( d ^{\alpha+2}-1\big){\gamma_0}+ \big(\frac12S_\alpha  d^{\alpha+2} - T_\alpha^-( d ,\vartheta)\big)\gamma_1}{\frac12S_\alpha- T_\alpha^+( d ,\vartheta) d ^{\alpha+2}},\\
\Omega^*&:=\frac{\widehat{C}_\alpha}{2(d_1^{\alpha+2}+d_2^{\alpha+2})}\Big(\gamma_0+\gamma_1\big(\tfrac12 S_\alpha+ T_\alpha^-( d ,\vartheta)\big)+\gamma_2^*\big(T_\alpha^+( d ,\vartheta)+\tfrac12 S_\alpha\big) \Big).
\end{split}
    \end{equation}
In order to  ensure that $\gamma_2$ is non-vanishing, one has to assume that $\gamma_0$ and $\gamma_1$ verify the condition 
 \begin{equation}\label{non-deg}
( d ^{\alpha+2}-1){\gamma_0}+ \big(\tfrac12S_\alpha  d ^{\alpha+2} - T_\alpha^-( d ,\vartheta)\big)\gamma_1\neq 0.
    \end{equation}

\begin{remark}
  For the Eulerian interaction $\alpha=0$, one may easily check that 
  $$
  S_\alpha=m-1\quad\textnormal{and}\quad T_\alpha^\pm( d ,\vartheta)=\frac{m}{1-(-1)^\vartheta  d ^{\pm m}}.
  $$ 
  It follows that \eqref{det-j-poly} and \eqref{non-deg} can be written as
  \begin{equation*}
    \frac{m-1}{2}+\frac{m  d ^2}{1-(-1)^\vartheta  d ^m}\neq 0, \quad  \big( d ^2-1\big){\gamma_0}+\Big(\frac{m-1}{2} d ^2- \frac{m  d ^{m}}{ d ^{m}-(-1)^\vartheta }\Big){\gamma_1}\neq 0.
  \end{equation*}
  This amounts to the study the of the roots of two polynomials of order $m$.  A detailed  analysis is given  in \cite[Pages 10--13]{Ce} and  \cite[Pages  18--22]{Aref}. 
\end{remark}

\begin{remark}
  While for general non-degenerate equilibria the differential $D_\lambda \mathcal P^\alpha(\lambda^*)$ was never onto, in this more symmetric setting the differential $D_\lambda \mathcal F^\alpha(\lambda^*)$ \emph{is} onto whenever \eqref{det-j-poly} holds. This will enable us to directly apply the implicit function theorem to the vortex patch equations, avoiding Lemma~\ref{abstract-lemma} and the use of integral identities.
\end{remark}
\subsection{Boundary equations}
Let $m$ be a positive integer  and   $\mathcal{O}_0^\varepsilon, \mathcal{O}_1^\varepsilon$, $\mathcal{O}_2^\varepsilon$   be three bounded simply connected domains containing  the origin and contained in the ball $B(0,2)$. Assume in addition  that $\mathcal{O}_0^\varepsilon$ is $m$-fold symmetric, that is
\begin{equation}\label{Nsymm}
e^{\frac{2\pi i}{m}}\mathcal{O}_0^\varepsilon=\mathcal{O}_0^\varepsilon,
\end{equation}
and $\mathcal{O}_0^\varepsilon,\mathcal{O}_1^\varepsilon$ and $\mathcal{O}_2^\varepsilon$ are  symmetric about the real axis.
Given  $b_0,b_1,b_2\in \R_+$ and $d_1,d_2\in \R_+$ and $\varepsilon\in(0,\varepsilon_0)$, with  $\varepsilon_0\ll 1$,  we define the domains
\begin{align}
\mathcal{D}_{00}^\varepsilon&:= \varepsilon b_0 \mathcal{O}_0^\varepsilon,\label{Dj0}\\
\mathcal{D}_{1j}^\varepsilon&:= e^{\frac{2j\pi i}{m}}\big(\varepsilon  b_1\mathcal{O}_1^\varepsilon+d_1\big),\quad\quad\quad\quad j=0,\ldots, m-1, \label{Dj1} \\ 
\mathcal{D}_{2j}^\varepsilon&:=e^{\frac{(2j+\vartheta)\pi i}{m}}\big(\varepsilon  b_2\mathcal{O}_2^\varepsilon+d_2\big),\quad\quad\;\,\, j=0,\ldots, m-1, \label{Dj2}
\end{align}
where we  recall that $\vartheta=0$ corresponds to
the aligned configuration    and  $\vartheta=1$ refers to   the staggered configuration. 
Let $\gamma_0,\gamma_1,\gamma_2\in \R\setminus\{0\}$ and  consider the initial vorticity  
\begin{equation}\label{omega0sym}
\omega_{0}^\varepsilon=\frac{\gamma_0}{\varepsilon^2 b_0^2}\chi_{\mathcal{D}_{00}^\varepsilon}+\frac{\gamma_1}{\varepsilon^2 b_1^2}\sum_{j=0}^{m-1}\chi_{\mathcal{D}_{1j^\varepsilon}}+\frac{\gamma_2}{\varepsilon^2 b_2^2}\sum_{j=0}^{m-1}\chi_{\mathcal{D}_{2j}^\varepsilon}.
\end{equation}
Now assume that the evolution of $\omega_{0}^\varepsilon$ is prescribed by the \eqref{eqn:boundry00} with $U=0$. While this initially gives $N=2m+1$ equations, one for each patch, using the fact that 
$$
\mathcal{D}_{1j}^\varepsilon=e^{\frac{2j\pi i}{m}}\mathcal{D}_{10}^\varepsilon,\quad\textnormal{and}\quad  \mathcal{D}_{2j}^\varepsilon=e^{\frac{2j\pi i}{m}}\mathcal{D}_{20}^\varepsilon,
$$
we shall show that this system can be reduced to a system of three equations, on the boundaries of  $\mathcal{O}_0^\varepsilon, \mathcal{O}_1^\varepsilon$ and  $\mathcal{O}_2^\varepsilon$.
\subsubsection{Euler equation}
From \eqref{eq2E0} one has
\begin{equation}\label{eqE0}
\begin{split}
&\Re\big\{\gamma_0\big(\Omega\overline{z}+V^\varepsilon(z)\big)z'\big\}=0,\quad \forall z\in \partial \mathcal{D}_{00}^\varepsilon,
\\ &\Re\big\{\gamma_1\big(\Omega\overline{z}+V^\varepsilon(z)\big)z'\big\}=0,\quad \forall z\in \partial \mathcal{D}_{1n}^\varepsilon, \quad n=0,\ldots,m-1,
\\ &\Re\big\{\gamma_2\big(\Omega\overline{z}+V^\varepsilon(z)\big)z'\big\}=0,\quad \forall z\in \partial \mathcal{D}_{2n}^\varepsilon, \quad n=0,\ldots,m-1,
\end{split}
\end{equation}
where  $z'$ denotes a tangent vector to the boundary  at the point $z$ and
 \begin{align*}
V^\varepsilon(z) &=\frac{\gamma_0}{2\varepsilon^2 b_0^2}\fint_{\partial \mathcal{D}_{00}^\varepsilon}\frac{\overline{\xi}-\overline{z}}{\xi-z}d\xi+\sum_{\ell=1}^2\frac{\gamma_\ell}{2\varepsilon^2 b_\ell^2}\sum_{k=0}^{m-1}\fint_{\partial \mathcal{D}_{\ell k}^\varepsilon}\frac{\overline{\xi}-\overline{z}}{\xi-z}d\xi
.
\end{align*}
In view of \eqref{Dj1} and \eqref{Dj2}, the change  of variables $\xi\mapsto e^{\frac{2k\pi i}{m}}\xi$  leads to
\begin{equation}\label{I}
V^\varepsilon(z)=\frac{\gamma_0}{2\varepsilon^2 b_0^2}\fint_{\partial \mathcal{D}_{00}^\varepsilon}\frac{\overline{\xi}-\overline{z}}{\xi-z}d\xi+\sum_{\ell=1}^2\frac{\gamma_\ell}{2\varepsilon^2 b_\ell^2}\sum_{k=0}^{m-1}\fint_{\partial \mathcal{D}_{\ell 0}^\varepsilon}\frac{\overline{\xi}-e^{\frac{2k\pi i}{m}}\overline{z}}{e^{\frac{2k\pi i}{m}}\xi- z}d\xi
.
\end{equation}
For any $n\in\{1,\ldots,m-1\}$ one has
\begin{align*}
V^\varepsilon\big(e^{\frac{2n\pi i}{m}}z\big)&=\frac{\gamma_0}{2\varepsilon^2 b_0^2}\fint_{\partial \mathcal{D}_{00}^\varepsilon}\frac{\overline{\xi}-e^{-\frac{2n\pi i}{m}}\overline{z}}{\xi-e^{\frac{2n\pi i}{m}}z}d\xi+e^{-\frac{2n\pi i}{m}}\sum_{\ell=1}^2\frac{\gamma_\ell}{2\varepsilon^2 b_\ell^2}\sum_{k=0}^{m-1}\fint_{\partial \mathcal{D}_{\ell 0}^\varepsilon}\frac{\overline{\xi}-e^{\frac{2(k-n)\pi i}{m}}\overline{z}}{e^{\frac{2(k-n)\pi i}{m}}\xi- z}d\xi.
\end{align*}
In view of \eqref{Nsymm}, making the change of variables $\xi\mapsto e^{\frac{2n\pi i}{m}}\xi$ in the first integral gives
\begin{equation}\label{Iej}
V^\varepsilon\big(e^{\frac{2n\pi i}{m}}z\big)=e^{-\frac{2n\pi i}{m}}V^\varepsilon(z).
\end{equation}
From \eqref{Dj1}, \eqref{Dj2} and \eqref{Iej} we conclude that if \eqref{eqE0} is satisfied for $n=0$, then it also satisfied for all $n\in\{1,\ldots,m-1\}$. Thus, the system  \eqref{eqE0} is reduced to
\begin{align}
&\Re\big\{\gamma_j\big(\Omega\overline{z}+V^\varepsilon(z)\big)z'\big\}=0,\quad \forall z\in \partial \mathcal{D}_{j0}^\varepsilon,\quad j=0,1,2.\label{eqEj}
\end{align}
We assume that the boundaries of the domains  $\mathcal{O}_j^\varepsilon$, $j=0,1,2$ in \eqref{Dj1}, \eqref{Dj2} are parametrized by conformal mappings $\phi_j\colon\mathbb{T}\to  \partial{\mathcal{O}_j^\varepsilon}$ satisfying
\begin{align*}
  \phi_j (w)= w+\varepsilon b_j f_j(w) \quad\textnormal{with}\quad f_j(w)=\sum_{m=1}^\infty\frac{a_m^j}{w^{m}},\quad a_m^j\in \R.
\end{align*}
Following the steps established in Section \ref{subsec:euler}, more precisely \eqref{rotn+100}, we may conclude that the dynamics the three boundaries is governed by the system
\begin{align} 
 \gamma_0\,\mathcal{G}_0^0(\varepsilon,f;\lambda)(w)&=  -\gamma_0\Im \Big\{
  \Big(
  \Omega\big(\varepsilon b_0\overline{w}+\varepsilon^2 b_0^2 f_0(\overline{w})\big)+{\gamma_0} \overline{\mathcal{I}^0}[\varepsilon, f_0](w)\notag
  \\
  &\qquad\qquad+\sum_{\ell=1}^2\gamma_\ell\sum_{k=0}^{m-1} \mathcal{K}_{k}^0[\varepsilon,f_\ell, f_0](w)
    \Big)
  w\big(1+\varepsilon b_0 f'_0(w)\big)
  -\frac{\gamma_0}{2} f_0'(w)\Big\}=0,\label{G0}\\
 \gamma_j\, \mathcal{G}_j^0(\varepsilon,f;\lambda)(w)&:=  -\gamma_j\Im \Big\{
  \Big(
  \Omega\big(\varepsilon b_j\overline{w}+\varepsilon^2 b_j^2 f_j(\overline{w})+d_j\big)+{\gamma_j} \overline{\mathcal{I}^0}[\varepsilon, f_j](w)+\gamma_0\, \mathcal{K}_{0}^0[\varepsilon,f_0, f_j](w)\notag
  \\
  &  \qquad\qquad+\sum_{\ell=1}^2\gamma_\ell\sum_{k=\delta_{\ell j}}^{m-1} \overline{\mathcal{K}_k^0}[\varepsilon,f_\ell, f_j](w)
    \Big)
  w\big(1+\varepsilon b_j f'_j(w)\big)
  -\frac{\gamma_j}{2} f_j'(w)\Big\}=0,\label{Gj}
\end{align}
for all $w\in\T$ and  $j=1,2$, where
\begin{align}\label{I0}
\overline{\mathcal{I}^0}[\varepsilon, f_j](w)&= \frac12\fint_{\mathbb{T}}\frac{\overline{w}-\overline{\tau}+\varepsilon b_j\big({f_j(\overline{\tau})}-{f_j(\overline{w})}\big)}{w-\tau+\varepsilon b_j\big(f_j(\tau)-f_j(w)\big)}f_j'(\tau)d\tau\notag
\\ 
&\qquad
+\fint_{\mathbb{T}}\frac{i\Im\big\{(w-\tau)\big({f_j(\overline{\tau})}-{f_j(\overline{w})}\big)\big\}}{(w-\tau)\big(w-\tau+\varepsilon b_j f_j(\tau)-\varepsilon b_j f_j(w)\big)}d\tau,
\\
\overline{\mathcal{K}_k^0}[\varepsilon,f_\ell, f_n](w)&:=\frac12\fint_\T \frac{
  \big(\overline{\tau}+\varepsilon b_\ell f_\ell(\overline{\tau})\big) \big(1+\varepsilon b_\ell f_\ell'({\tau})\big)}
  {\nu_{k\ell j}(\varepsilon b_\ell  {\tau}+\varepsilon^2 b_\ell^2 f_\ell({\tau}) +d_\ell) - \varepsilon b_n\big({w}+\varepsilon b_n f_n({w})\big) -  
  d_j} \, d\tau, \label{K}
   \end{align}
with the convention $d_0=0$ and 
$
\nu_{k\ell j}:=\exp\big(2k\pi i/m+(\delta_{2\ell}-\delta_{2j})\vartheta\pi i/m\big).
$
\subsubsection{gSQG equations}
From  \eqref{eqn:boundry00} and \eqref{vgsqg} one has
\begin{equation}\label{eqn:boundry2}
\begin{split}
&\Re\big\{\gamma_{0}\big(\Omega {z}+i{v^\varepsilon(z)}\big)\overline{z'}\big\}=0\quad  \forall  z\in\displaystyle\partial \mathcal{D}_{00}^\varepsilon, \\
&\Re\big\{\gamma_{1}\big(\Omega {z}+i{v^\varepsilon(z)}\big)\overline{z'}\big\}=0\quad  \forall z\in\displaystyle\partial \mathcal{D}_{1n}^\varepsilon, \quad  n=0,\ldots,m-1,\\
&\Re\big\{\gamma_{2}\big(\Omega {z}+i{v^\varepsilon(z)}\big)\overline{z'}\big\}=0\quad \forall z\in\displaystyle\partial \mathcal{D}_{2n}^\varepsilon, \quad  n=0,\ldots,m-1,
\end{split}
\end{equation}
where  $z'$ denotes a tangent vector to the boundary  at the point $z$ and 
\begin{equation*}
{v^\varepsilon(z)} =\frac{C_\alpha}{2\pi}\frac{\gamma_0}{\varepsilon^2 b_0^2}\int_{\partial\mathcal{D}_{00}^\varepsilon}\frac{d\xi}{|z-\xi|^\alpha}+\sum_{m=1}^2\frac{\gamma_\ell}{\varepsilon^2 b_\ell^2}\sum_{k=0}^{m-1}\frac{C_\alpha}{2\pi}\int_{\partial\mathcal{D}_{\ell k}^\varepsilon}\frac{d\xi}{|z-\xi|^\alpha}
\end{equation*}
for all $z\in\mathbb{C}$. In view of  \eqref{Dj1} and \eqref{Dj2}, a suitable change of variables gives
\begin{equation*}
{v^\varepsilon(z)} =\frac{C_\alpha}{2\pi}\frac{\gamma_0}{\varepsilon^2 b_0^2}\int_{\partial\mathcal{D}_{00}^\varepsilon}\frac{d\xi}{|z-\xi|^\alpha}+\sum_{\ell=1}^2\frac{\gamma_\ell}{\varepsilon^2 b_\ell^2}\sum_{k=0}^{m-1}\frac{C_\alpha}{2\pi}\int_{\partial\mathcal{D}_{\ell 0}^\varepsilon}\frac{e^{\frac{2\pi k i}{m}}d\xi}{|z-e^{\frac{2\pi k i}{m}}\xi|^\alpha}. 
\end{equation*}
Observe that for any $n\in\{1,\ldots,m-1\}$ one has
\begin{equation*}
{v^\varepsilon\big(e^{\frac{2\pi n i}{m}}z\big)} =\frac{C_\alpha}{2\pi}\frac{\gamma_0}{\varepsilon^2 b_0^2}\int_{\partial\mathcal{D}_{00}^\varepsilon}\frac{d\xi}{|e^{\frac{2\pi n i}{m}}z-\xi|^\alpha}+e^{\frac{2\pi n i}{m}}\sum_{\ell=1}^2\frac{\gamma_\ell}{\varepsilon^2 b_\ell^2}\sum_{k=0}^{m-1}\frac{C_\alpha}{2\pi}\int_{\partial\mathcal{D}_{\ell 0}^\varepsilon}\frac{e^{\frac{2\pi (k-n) i}{m}}}{|e^{-\frac{2\pi (k-n) i}{m}}z-\xi|^\alpha}d\xi.
\end{equation*}
From \eqref{Nsymm}, the change of variable $\xi \mapsto e^{\frac{2\pi n i}{m}} \xi $ in the first integral leads to
\begin{equation*}
{v^\varepsilon\big(e^{\frac{2\pi n i}{m}}z\big)} =e^{\frac{2\pi n i}{m}}v^\varepsilon\big(z\big).
\end{equation*}
From the last identity and by \eqref{Dj1} and \eqref{Dj2}, we conclude that the system \eqref{eqn:boundry2} of $2m+1$ equations can be reduced to a system of three equations, 
\begin{equation*}
\begin{split}
&\gamma_{j}\Re\big\{\big(\Omega {z}+i{v^\varepsilon(z)}\big)\overline{z'}\big\}=0\quad \textnormal{for all}\quad z\in\displaystyle\partial \mathcal{D}_{j0}^\varepsilon,\quad j=0,1,2. \\
\end{split}
\end{equation*}
Assume that the boundaries of the domains  $\mathcal{O}_j^\varepsilon$, $j=0,1,2$ in \eqref{Dj1}, \eqref{Dj2} are parametrized by the conformal mappings $\phi_j\colon\mathbb{T}\to  \partial{\mathcal{O}_j^\varepsilon}$ satisfying
\begin{align*}
  \phi_j (w)= w+\varepsilon|\varepsilon|^\alpha b_j^{1+\alpha} f_j(w) \quad\textnormal{with}\quad f_j(w)=\sum_{m=1}^\infty\frac{a_m^j}{w^{m}},\quad a_m^j\in \R.
\end{align*}
Then, from \eqref{Fjsqgalpha0} one may conclude that the dynamics of three boundaries is described by 
\begin{align} \label{sys2sqg1}
 \gamma_0\,\mathcal{G}_0^\alpha(\varepsilon,f;\lambda)(w)&:= \gamma_0\Im \Big\{
  \Big(
  \Omega\big(\varepsilon b_0 w+b_0^{\alpha+2}\varepsilon^2|\varepsilon|^\alpha f_0(w) \big)+\gamma_0\mathcal{I}^\alpha[\varepsilon, f_0](w) \notag 
  \\
  &\quad
  +\sum_{\ell=1}^{2} \gamma_\ell\sum_{k=0}^{m-1} \mathcal{K}^{\alpha}_{k}[\varepsilon,f_\ell,f_0](w)
    \Big)
 \overline{ w}\big(1+b_0^{\alpha+1}\varepsilon|\varepsilon|^\alpha f'_0(\overline{w})\big)
  -\gamma_0\mu_\alpha f_0'(\overline {w})\Big\}=0,
\\
\gamma_j\, \mathcal{G}_j^\alpha(\varepsilon,f;\lambda)(w)&:=\gamma_j\Im \Big\{
  \Big(
  \Omega\big(\varepsilon b_j w+b_j^{\alpha+2}\varepsilon^2|\varepsilon|^\alpha f_j(w)+d_j\big)+\gamma_j \mathcal{I}^\alpha[\varepsilon, f_j](w) +\gamma_0  \mathcal{K}^{\alpha}_{0}[\varepsilon,f_0,f_j](w)
 \notag \\& \quad +\sum_{\ell=1}^{2} \gamma_\ell\sum_{k=\delta_{\ell j}}^{m-1}  \mathcal{K}^{\alpha}_{k}[\varepsilon,f_\ell,f_j](w)  \Big)
 \overline{ w}
  \big(1+b_j^{\alpha+1}\varepsilon|\varepsilon|^\alpha f_j( \overline{ w})\big)
  -\gamma_j\mu_\alpha f_j'(\overline{w})\Big\}=0,\label{sys2sqg2}
\end{align}
for all   $w\in \T$ and $j=1,2$,
 where $\mu_\alpha$ is defined in \eqref{mu} and 
\begin{align}
\mathcal{I}^\alpha[\varepsilon, f_j](w)&= -C_{\alpha} \fint_\T 
  \frac{f_j'(\tau)}{|\phi_j(\tau)-\phi_j(w)|^\alpha}\, d\tau\nonumber
  \\ 
  &\quad\;
  +\alpha C_{\alpha} \fint_\T\int_0^1 
  \frac{\Re\big[\big(f_j(\tau)-f_j(w)\big)\big(
 \overline{\tau} - \overline{w}\big)\big]+\varepsilon|\varepsilon|^\alpha t|f_j(\tau)-f_j(w)|^2}{|w-\tau+t\varepsilon|\varepsilon|^\alpha\big(f_j(\tau)-f_j(w)\big)|^{2+\alpha}} dt\, d\tau, \label{Ialpha}
 \\
\mathcal{K}^\alpha_{k}[\varepsilon,f_\ell,f_j](w)&:=
 \frac{\alpha C_\alpha}{ b_ \ell}\bigg[ \fint_\T\int_0^1 \frac{
\Re\big[\big(\nu_{k\ell j}  d_\ell-d_j\big)\big( b_ \ell \overline{\nu_{k\ell j}} 
 \overline{\phi_\ell(\tau)} -  b_j \overline{\phi_j(w)}\big)\big] }
  { | \nu_{k\ell j}\big(
t\varepsilon b_\ell \phi_\ell(\tau)+d_\ell\big) - \big(t\varepsilon b_j \phi_j(w) +d_j\big)|^{\alpha+2}} \phi_\ell'(\tau)dt\, d\tau
\nonumber
\\
&  \qquad + \fint_\T\int_0^1\frac{
t \varepsilon | \nu_{k\ell j}
 b_ \ell \phi_ \ell(\tau) -  b_j \phi_j(w) |^2\phi_\ell'(\tau) }
  { |\nu_{k\ell j} \big(
t\varepsilon b_\ell \phi_\ell(\tau)+d_\ell\big) - \big(t\varepsilon b_j \phi_j(w) +d_j\big)|^{\alpha+2}} dt\, d\tau\bigg]\nu_{k\ell j} ,\label{Kalpha}
\end{align}
with the convention $d_0=0$ and 
$
\nu_{k\ell j}=\exp\big(2k\pi i/m+(\delta_{2\ell}-\delta_{2j})\vartheta\pi i/m\big).
$

\subsection{Existence of the nested polygonal vortex patch equilibria}
For any  $m\geq 2$ we define the Banach spaces 
\begin{equation}\label{Valpha}
\begin{split}
\mathcal{V}^\alpha :=V^\alpha_m\times V^\alpha_1\times V^\alpha_1 \;\;&\textnormal{and}\;\;  
\mathcal{W}^\alpha:=W^\alpha_m\times W^\alpha_1\times W^\alpha_1,\\
\textnormal{with}\quad   V_m^\alpha:=
        \big\{f\in V_1^\alpha: f\big(e^{\frac{i2\pi}{m}}w\big)=e^{\frac{i2\pi}{m}}f(w)\big\}\;\;&\textnormal{and}\;\;  W^\alpha_m:=
        \big\{g\in  W_1^\alpha : g\big(e^{\frac{i2\pi}{m}}z\big)=g(z)\big\}, 
\end{split}
\end{equation}
 where $V_1^\alpha$ and $W_1^\alpha$ are defined by
\begin{equation}\label{Valpha1}
   V_1^\alpha:=
        \big\{f\in  V_0^\alpha : \overline{f(w)}=f(\overline{w})\big\}\quad \textnormal{and}\quad  W^\alpha_1:=
        \big\{g\in  W_0^\alpha : \overline{g(w)}=-g(\overline w)\big\}
\end{equation}
and $V_0^\alpha$ and $W_0^\alpha$ are given by \eqref{Valpha1}.  Note that if $f_0\in V_m^\alpha$, the expansion of the associated conformal mapping is given by
$$
\phi_0(w)=w+\varepsilon |\varepsilon|^\alpha b_0^{1+\alpha} f_0(w)=w\Big(1+\varepsilon |\varepsilon|^\alpha b_0^{1+\alpha} \sum_{n=1}^{\infty}\frac{a_{nm-1}}{w^{nm}}\Big).
$$
which provides the $m$-fold symmetry of the associated patch. 
We denote by $B^\alpha$ the open unit ball in $\mathcal{V}^\alpha$. Define the mapping
$$
\mathcal{G}^\alpha(\varepsilon,f;\lambda):=\big(\mathcal{G}^\alpha_0(\varepsilon,f;\lambda),\mathcal{G}^\alpha_1(\varepsilon,f;\lambda), \mathcal{G}^\alpha_2(\varepsilon,f;\lambda)\big), 
$$
where $f=(f_0,f_1,f_2)$, $\lambda=(\Omega,\gamma_2)$, $\mathcal{G}^\alpha_j$ is given by \eqref{sys2sqg1}--\eqref{sys2sqg2}  for $\alpha\in(0,1)$ and   by  \eqref{G0}--\eqref{Gj} for $\alpha=0$.

\vspace{0.2cm}

The proof of the existence of the co-rotating nested polygons  follows from the next theorem, which
 gives the full statement of Theorem~\ref{thm:informal-1polygon}.

\begin{theorem}\label{thm:polygon} 
  Let $\alpha \in [0,1)$, $b_1,b_2,d_1,d_2\in (0,\infty)$ such that $d=d_2/d_1>0$ satisfies \eqref{det-j-poly}, and let $\gamma_0,\gamma_1\in \R\setminus\{0\}$ such that \eqref{non-deg} holds. Then 
\begin{enumerate}[label=\rm(\roman*)]
\item There exists $\varepsilon_0 > 0$ and a neighborhood $\Lambda$ of $\lambda^*$ in $\R^2$ such that  $\mathcal{G}^\alpha$ 
can be extended to a $C^1$ mapping $(-\varepsilon_0,\varepsilon_0)\times B^\alpha\times \Lambda \to \mathcal{W}^\alpha$.
\item $\mathcal{G}^\alpha(0,0;\lambda^*)=0,
$
where $\lambda^*=(\Omega^*,\gamma_2^*)$ is given by  \eqref{gamma0-om0}.
\item  The linear operator 
$D_{(f;\lambda)}\mathcal{G}^\alpha (0,0;\lambda^*)\colon  \mathcal{V}^\alpha\times\R^2 \to  \mathcal{W}^\alpha$ is an isomorphism.
\item There exists $\varepsilon_1>0$ and a unique $C^1$ function $(f,\lambda)\colon (-\varepsilon_1,\varepsilon_1)\to   B^\alpha\times\R^2$ such that 
\begin{equation*}\label{gf1f2-poly}
\mathcal{G}^\alpha \big(\varepsilon, f(\varepsilon);\lambda(\varepsilon)\big)=0,
\end{equation*}
with $\lambda(\varepsilon)=\lambda^*+o(\varepsilon)$ and 
\begin{gather}
f(\varepsilon)=\Xi_\alpha\,\Big(0,\frac{\varepsilon b_1\mathcal{Q}_1^\alpha}{\gamma_1d_1^{\alpha+2}}\overline{w} ,\frac{\varepsilon b_2\mathcal{Q}_2^\alpha}{\gamma_2^* d_2^{\alpha+2}}\overline{w}\Big)+o(\varepsilon),
\notag \\
\label{Qjalpha}
\mathcal{Q}_j^\alpha:=
\gamma_0+ \sum_{\ell=1}^2\gamma_\ell\sum_{k=\delta_{\ell j}}^{m-1}  \frac{\big(e^{\frac{2k\pi i}{m}+(\delta_{2\ell}-\delta_{2j})\frac{\vartheta\pi i}{m}} d_\ell /d_j-1\big)^2
}{ |e^{\frac{2k\pi i}{m}+(\delta_{2\ell}-\delta_{2j})\frac{\vartheta\pi i}{m}} d_\ell/d_j-1|^{\alpha+4}}, \quad \Xi_\alpha:=\frac{(\alpha+2)\Gamma(1-\frac\alpha2)\Gamma(3-\frac\alpha2)}{4\Gamma(2-\alpha)}.
\end{gather}
\item For all  $\varepsilon\in (-\varepsilon_1,\varepsilon_1)\setminus\{0\}$  the domains $\mathcal{O}_j^\varepsilon$, whose boundaries are given by the conformal parametrizations $\phi_j^\varepsilon=\operatorname{Id}+\varepsilon|\varepsilon|^\alpha b_j^{1+\alpha} f_j\colon\mathbb{T}\to \partial \mathcal{O}_j^\varepsilon$,   are
 strictly convex.

\end{enumerate}
\end{theorem}
\begin{proof}
 The regularity of the nonlinear operator $\mathcal{G}^\alpha $ follows from Proposition~\ref{proposition2}. In order to prove  the reflection symmetry property we shall assume that the Fourier coefficients  of $f_1,f_2$ are real,  that is
\begin{equation}\label{ref-sym4}
\overline{f_j(w)}=f_j(\overline{w}), 
\end{equation}
and prove that
\begin{equation}\label{ref-sym24}
\mathcal{G}_j^\alpha (\varepsilon,f;\lambda)(\overline{w})=-\mathcal{G}_j^\alpha (\varepsilon,f)({w}). 
\end{equation}
  It is obvious that if $f_1, f_2$ satisfy \eqref{ref-sym4}, then  
  $$w\mapsto  \Im \big\{
  \Omega\big(\varepsilon b_j w+\varepsilon^2|\varepsilon|^\alpha b_j^{2+\alpha}  f_j(w) +d_j \big)
 \overline{ w}\big(1+\varepsilon|\varepsilon|^\alpha b_1^{1+\alpha} \overline{f_j'(w)} \big)
  -\mu_\alpha\, \gamma_j \overline{f_j'(w)}\big\}
  $$
   satisfies \eqref{ref-sym24}. Moreover, using \eqref{Ialpha} and \eqref{Kalpha}, we check that the Fourier coefficients of  $\mathcal{I}^\alpha[\varepsilon, f_j](w)$ and $\mathcal{J}^\alpha_k[\varepsilon,f_k,f_j](w) $  are also real for every $f$ satisfying  \eqref{ref-sym}, namely,
$$
\overline{\mathcal{I}^\alpha[\varepsilon, f_j](w)}=\mathcal{I}^\alpha[\varepsilon, f_j](\overline{w}), \qquad\textnormal{and}\qquad \overline{\mathcal{K}^\alpha_k[\varepsilon,f_k,f_j;\lambda](w) }=\mathcal{K}^\alpha_k[\varepsilon,f_k,f_j;\lambda](\overline{w}). 
$$
Then using \eqref{sys2sqg1} we  conclude \eqref{ref-sym24}.
Thus, it remains to check the $m$-fold  symmetry  property of $\mathcal{G}_0^\alpha $, namely, that if
\begin{equation}\label{m-fold-poly}
{f_0(e^{\frac{2\pi i}{m}}w)}=e^{\frac{2\pi i}{m}} f_0({w}) \quad \forall w\in \mathbb{T},
\end{equation}
then
\begin{equation}\label{m-fold2-poly}
\mathcal{G}_0^\alpha (\varepsilon,f;\lambda)(e^{\frac{2\pi i}{m}}{w})=\mathcal{G}_0^\alpha (\varepsilon,f;\lambda)({w}) \quad \forall w\in \mathbb{T}.
\end{equation}
We shall give the  details of the proof in the case $\alpha\in(0,1)$; the  case $\alpha=0$ can be checked in a similar way.
From \eqref{I0} one has 
  \begin{align*}
\overline{\mathcal{I}^0}[\varepsilon, f_0](e^{\frac{2\pi i}{m}}w)&= \frac12\fint_{\mathbb{T}}\frac{e^{-\frac{2\pi i}{m}}\overline{w}-\overline{\tau}+\varepsilon b_0\big({f_0(\overline{\tau})}-{f_0(e^{-\frac{2\pi i}{m}}\overline{w})}\big)}{e^{\frac{2\pi i}{m}}w-\tau+\varepsilon b_0\big(f_0(\tau)-f_0(e^{\frac{2\pi i}{m}}w)\big)}f_0'(\tau)d\tau\notag\\ &\qquad+\fint_{\mathbb{T}}\frac{i\Im\big\{(e^{\frac{2\pi i}{m}}w-\tau)\big({f_0(\overline{\tau})}-{f_0(e^{-\frac{2\pi i}{m}}\overline{w})}\big)\big\}}{(e^{\frac{2\pi i}{m}}w-\tau)\big(e^{\frac{2\pi i}{m}}w-\tau+\varepsilon b_0 f_0(\tau)-\varepsilon b_0 f_0(e^{\frac{2\pi i}{m}}w)\big)}d\tau.
\end{align*}
Using the change of variables $\tau\mapsto e^{\frac{2\pi i}{m}}\tau$, we find
  \begin{align*}
\overline{\mathcal{I}^0}[\varepsilon, f_0](e^{\frac{2\pi i}{m}}w)&= \frac12\fint_{\mathbb{T}}\frac{e^{-\frac{2\pi i}{m}}\overline{w}-e^{-\frac{2\pi i}{m}}\overline{\tau}+\varepsilon b_0\big({f_0(e^{-\frac{2\pi i}{m}}\overline{\tau})}-{f_0(e^{-\frac{2\pi i}{m}}\overline{w})}\big)}{e^{\frac{2\pi i}{m}}w-e^{\frac{2\pi i}{m}}\tau+\varepsilon b_0\big(f_0(e^{\frac{2\pi i}{m}}\tau)-f_0(e^{\frac{2\pi i}{m}}w)\big)}f_0'(e^{\frac{2\pi i}{m}}\tau)e^{\frac{2\pi i}{m}}d\tau\notag\\ &\qquad+\fint_{\mathbb{T}}\frac{i\Im\big\{(e^{\frac{2\pi i}{m}}w-e^{\frac{2\pi i}{m}}\tau)\big({f_0(e^{-\frac{2\pi i}{m}}\overline{\tau})}-{f_0(e^{-\frac{2\pi i}{m}}\overline{w})}\big)\big\}}{(w-\tau)\big(e^{\frac{2\pi i}{m}}w-e^{\frac{2\pi i}{m}}\tau+\varepsilon b_0 f_0(e^{\frac{2\pi i}{m}}\tau)-\varepsilon b_0 f_0(e^{\frac{2\pi i}{m}}w)\big)}d\tau.
\end{align*}
Then, by  \eqref{m-fold-poly}, we deduce that
\begin{align*}
\overline{\mathcal{I}^0}[\varepsilon, f_0](e^{\frac{2\pi i}{m}}w)&= e^{-\frac{2\pi i}{m}}\overline{\mathcal{I}^0}[\varepsilon, f_0](w).
\end{align*}
 In view of \eqref{K} and \eqref{m-fold-poly} we have
\begin{equation*}
\overline{\mathcal{K}_k^0}[\varepsilon,f_\ell, f_0](e^{\frac{2\pi i}{m}}w)=\frac{e^{-\frac{2\pi i}{m}}}{2}\fint_\T \frac{
  \big(\overline{\tau}+\varepsilon b_\ell f_\ell(\overline{\tau})\big) \big(1+\varepsilon b_\ell f_\ell'({\tau})\big)}
  {e^{(k-1)\frac{2\pi i}{m}+\delta_{2\ell}\frac{\vartheta\pi i}{m}}(\varepsilon b_\ell {\tau}+\varepsilon^2 b_\ell^2 f_\ell({\tau}) +d_\ell) - \varepsilon b_0\big({w}+\varepsilon b_0 f_0({w})\big)} \, d\tau.
\end{equation*}
Summing over $k$ then gives
\begin{align*}
\sum_{k=0}^{m-1}\overline{\mathcal{K}_k^0}[\varepsilon,f_\ell, f_0](e^{\frac{2\pi i}{m}}w)&
 =e^{-\frac{2\pi i}{m}}\sum_{k=0}^{m-1}\overline{\mathcal{K}_k^0}[\varepsilon,f_\ell, f_0](w),
\end{align*}
concluding the proof of {\rm(i)}. 
 The proof of {\rm(ii)} follows immediately from Proposition~\ref{proposition2}(ii), \eqref{p-gG2} and \eqref{gamma0-om0}. 
In order to show {\rm(iii)} we  use  Proposition~\ref{proposition2}(ii)  and  \eqref{jacob-mat} to get,
 for all $h=(h_0,h_1,h_2)\in  \mathcal{V}^\alpha$ and  $(\dot\Omega,
\dot\gamma_2)\in\mathbb{R}^2$, 
\begin{equation*}
D_{(f;\lambda)}\mathcal{G}^\alpha(0,0;\lambda^*)\begin{pmatrix}
\dot\Omega\\
\dot\gamma_2\\
h
\end{pmatrix}(w)= -\begin{pmatrix}
0  & 0\\
d_1 &  \frac{\widehat{C}_\alpha }{2}\frac{ T_\alpha^+( d ,\vartheta)}{d_1^{1+\alpha}}\\
 d_2 &  \frac{\widehat{C}_\alpha }{2} \frac{ S_\alpha}{2d_2^{1+\alpha}}
\end{pmatrix}
\begin{pmatrix} \dot\Omega\\ \dot \gamma_2\end{pmatrix}\, \Im\{w\}
+\sum_{n\geq 1}M_n^\alpha\begin{pmatrix}
{\gamma_0}\, a_n^0 \\
{\gamma_1}\, a_n^1
\\
\gamma_2^*\, a_n^2
\end{pmatrix}\,\Im\{w^{n+1}\},
\end{equation*}
where $M_n^\alpha$ is given by \eqref{lambdan}. 
Proposition~\ref{proposition2}{\rm(iv)} and the assumption \eqref{det-j-poly} then imply {\rm(iii)}.

\vspace{0.2cm}

The existence and uniqueness in {\rm(iv)} follow form the implicit function theorem. In order to compute the asymptotic of the solution, we shall use the  formula 
\begin{align}\label{comp}
\partial_\varepsilon \big(f(\varepsilon),\lambda(\varepsilon)\big)\big|_{\varepsilon=0}=-D_{(f;\lambda)} \mathcal{G}^\alpha\big(0,0;\lambda^*)^{-1}\partial_\varepsilon \mathcal{G}^\alpha(0,0;\lambda^*).
\end{align}
For any $H\in \mathcal{W}^\alpha$ with the expansion
$$
H(w)=\sum_{n\geq 0}\begin{pmatrix}
A_n^0  \\
A_n^1\\
A_n^2
\end{pmatrix} \,\Im\{w^{n+1}\},
$$
with $A_n^0=0$ if $n$ is not a multiple of $m$, we have
\begin{align}\label{inv}
&D_{(f,\lambda)}\mathcal{G}^\alpha(0,0;\lambda^*)^{-1}H(w)=\\ &\Big(
\frac{1}{\gamma_0}\sum_{n\geq 1}\frac{A_n^0}{M_n^\alpha}\, \overline{w}^n,
\frac{1}{\gamma_1}\sum_{n\geq 1}\frac{A_n^1}{M_n^\alpha}\, \overline{w}^n,
\frac{1}{\gamma_2^*}\sum_{n\geq 1}\frac{A_n^2}{M_n^\alpha}\, \overline{w}^n\, ; \frac{ d ^{\alpha+2}T_\alpha^+( d ,\vartheta)A_0^2-\frac12{S_\alpha}A_0^1}{\det\big(D_\lambda \mathcal P^\alpha(\lambda^*)\big)},\frac{A_0^2d_1-A_0^1d_2}{\det\big(D_\lambda \mathcal P^\alpha(\lambda^*)\big)}
\Big),\notag
\end{align}
where $\det\big(D_\lambda \mathcal P^\alpha(\lambda^*)\big)$ was calculated in \eqref{det-j-poly}.
On the other hand, from \eqref{sys2sqg1}--\eqref{sys2sqg2}  and   \eqref{G0}--\eqref{Gj} we have
\begin{equation}\label{gfin}
\begin{split}
&\mathcal{G}^\alpha_{0}(\varepsilon,0;\lambda)(w)=\Im \Big\{\Big( \sum_{\ell=1}^2\gamma_\ell\sum_{k=0}^{m-1} {\mathcal{K}_{k}^\alpha}[\varepsilon,f_\ell, f_0](w)\Big)  \overline{ w}
\Big\},
\\
&\mathcal{G}^\alpha_{j}(\varepsilon,0;\lambda)(w)=\Im \Big\{
  \Omega{ d_j}\overline{ w}+
\Big(\gamma_0{\mathcal{K}_{0}^\alpha}[\varepsilon,f_0, f_j](w)+ \sum_{\ell=1}^2\gamma_\ell\sum_{k=\delta_{\ell j}}^{m-1} {\mathcal{K}_{k}^\alpha}[\varepsilon,f_\ell, f_j](w)\Big)\overline{ w} \Big\}, 
\end{split}
\end{equation}
with $ j=1,2$.
\paragraph{\bf Case $\alpha=0$.} Differentiating  \eqref{K} with respect to $\varepsilon $ gives 
\begin{align*}
\partial_\varepsilon \overline{\mathcal{K}_k^0}[\varepsilon,0, 0](w)|_{\varepsilon=0}
&=\frac{1}{2\big(\nu_{k\ell j}d_\ell-d_j\big)^{2}}\fint_\T \overline{\tau}\big(\nu_{k\ell j} b_\ell  {\tau} -  b_n{w}\big) \, d\tau\\ 
&=\frac{  b_n{w}}{2\big(\nu_{k\ell j}d_\ell-d_j\big)^{2}},
\end{align*}
with $\nu_{k\ell j}=\exp\big(2k\pi i/m+(\delta_{2\ell}-\delta_{2j})\vartheta\pi i/{m}\big)$. It follows from \eqref{gfin} that
\begin{equation*}
\begin{split}
&\partial_\varepsilon \mathcal{G}^0_{0}(\varepsilon,0;\lambda)|_{\varepsilon=0}(w)=\frac{b_0}{ 2 }\,\Im \Big\{ \sum_{\ell=1}^2 \frac{\gamma_\ell}{ d_\ell^{2} }\sum_{k=0}^{m-1}e^{\frac{4k\pi i}{m}+(\delta_{2\ell}-\delta_{2j})\frac{2\vartheta\pi i}{m}} \overline{ w}^2
\Big\}=0,
\\
&\partial_\varepsilon\mathcal{G}^0_{j}(\varepsilon,0;\lambda)|_{\varepsilon=0}(w)=\frac{b_j}{ 2 d_j^{2}}\, 
 \Big(\gamma_0+ \sum_{\ell=1}^2\sum_{k=\delta_{\ell j}}^{m-1}  \frac{\gamma_\ell
}{ \big(\overline{\nu_{k\ell j}} \frac{d_\ell}{d_j}-1\big)^{2}}\Big) \,\Im \{\overline{ w}^2 \}=\frac{b_j}{ 2 d_j^{2}}\, 
 \mathcal{Q}_j^0 \,\Im \{\overline{ w}^2 \}.
\end{split}
\end{equation*}
 Combining the two last identities with \eqref{comp} and \eqref{inv}  yields
\begin{align*}
\partial_\varepsilon \big( f(\varepsilon),\lambda(\varepsilon)\big)\big|_{\varepsilon=0}=\Big(0,\frac{b_1\mathcal{Q}_1^0}{\gamma_1d_1^{\alpha+2}}\overline{w} ,\frac{b_2\mathcal{Q}_2^0}{\gamma_2^* d_2^{\alpha+2}}\overline{w}\, ;0,0 \Big).
\end{align*} 
\paragraph{\bf Case $\alpha\in (0,1)$.}
From \eqref{Kalpha} we have
\begin{align*}
\mathcal{K}^\alpha_{k}[\varepsilon,0,0](w)
 &=\alpha C_\alpha \frac{\nu_{k\ell j} }{ b_ \ell} \bigg[\fint_\T\int_0^1\frac{
\Re\big[(\nu_{k\ell j}  d_\ell-d_j)(\overline{\nu_{k\ell j}}  
 b_ \ell \overline{\tau} -  b_j \overline{w})\big] }
  { | \nu_{k\ell j}\big(
t\varepsilon b_\ell \tau+d_\ell\big) - \big(t\varepsilon b_j w +d_j\big)|^{\alpha+2}}dt\,  d\tau 
\\ &\qquad\qquad\qquad + \fint_\T\int_0^1\frac{
t \varepsilon | \nu_{k\ell j}
 b_ \ell \tau -  b_j w |^2}
  { |\nu_{k\ell j} \big(
t\varepsilon b_\ell \tau+d_\ell\big) - \big(t\varepsilon b_j w +d_j\big)|^{\alpha+2}} dt\,  d\tau\bigg].
\end{align*}
with $\nu_{k\ell j}$ defined as for $\alpha=0$. Applying formula \eqref{taylor0} gives
\begin{equation*}
\begin{split}
\mathcal{K}^\alpha_{k}[\varepsilon,f_\ell,f_j](w)
&=\frac{\alpha C_\alpha}{ 4} \bigg[
\frac{2\big(\nu_{k\ell j} d_\ell-d_j\big)+\varepsilon b_j w  }{ | \nu_{k\ell j} d_\ell - d_j|^{\alpha+2}}
+(\alpha+2) \frac{\big(\nu_{k\ell j} d_\ell - d_j\big)^2
 \varepsilon b_j \overline{w} }{ |\nu_{k\ell j} d_\ell - d_j|^{\alpha+4}} \bigg]+o(\varepsilon).
\end{split}
\end{equation*}
Inserting the last identity into \eqref{gfin} and then differentiating with respect $\varepsilon$, we obtain  
\begin{equation*}
\begin{split}
&\partial_\varepsilon \mathcal{G}^\alpha_{0}(0,0;\lambda)(w)=\frac{(\alpha+2) \alpha C_\alpha}{ 4 }\,\Im \Big\{ b_0\overline{ w}^2 \sum_{\ell=1}^2 \frac{\gamma_\ell}{ d_\ell^{\alpha+2} }\sum_{k=0}^{m-1}e^{(2k+\sigma(\ell,j)\vartheta)\frac{2\pi i}{m}}   
\Big\}=0,
\\
&\partial_\varepsilon\mathcal{G}^\alpha_{j}(0,0;\lambda)(w)=\frac{(\alpha+2) \alpha C_\alpha}{ 4 d_j^{\alpha+2}}\, b_j
  \Big(\gamma_0+ \sum_{\ell=1}^2\gamma_\ell\sum_{k=\delta_{\ell j}}^{m-1}  \frac{\big(\nu_{k\ell j} \frac{d_\ell}{d_j}-1\big)^2
}{ |\nu_{k\ell j} \frac{d_\ell}{d_j}-1|^{\alpha+4}} \Big)\,\Im \{\overline{ w}^2\}.
\end{split}
\end{equation*}
Combining the two last identities with \eqref{comp}, \eqref{inv} and \eqref{gam1hatg},  we get 
\begin{align*}
\partial_\varepsilon \big( f(\varepsilon),\lambda(\varepsilon)\big)\big|_{\varepsilon=0}=\Xi_\alpha\Big(0,\frac{b_1\mathcal{Q}^\alpha_1}{\gamma_1d_1^{\alpha+2}}\overline{w} ,\frac{b_2\mathcal{Q}_2^\alpha}{\gamma_2^* d_2^{\alpha+2}}\overline{w}\, ; 0,0 \Big),
\end{align*} 
where $\mathcal{Q}_j^\alpha$ and $\Xi_\alpha$ are given by \eqref{Qjalpha}.

The convexity in {\rm (v)} is established in exactly the same was as in the proof of Theorem~\ref{prop:ift}. This ends the proof of the theorem.
\end{proof}
\begin{remark}
Note that, by setting $\gamma_2=0$ and $\lambda=\Omega$ in \eqref{intial-vort},  \eqref{sys2sqg1}--\eqref{sys2sqg2}  and \eqref{G0}--\eqref{Gj}, 
we can recover the existence and uniqueness result  of the body-centered polygonal configuration through Theorem~\ref{thm:polygon}. This remains equally true for the rotating vortex polygon by setting $\gamma_0=\gamma_2=0$ and $\lambda=\Omega$.
\end{remark}

\section*{Acknowledgments}
Miles H.~Wheeler was partially supported by NSF-DMS grant 1400926. The work of Z. Hassainia is supported by Tamkeen under the NYU Abu Dhabi Research Institute grant of the center SITE.

\vspace{0.3cm}

\bibliographystyle{plain}
\bibliography{multipole}

\noindent \textsc{Department of Mathematics, 
New York University in Abu Dhabi, 
Saadiyat Island, 
P.O. Box 129188, 
Abu Dhabi, 
United Arab Emirates}
,\\
\textit{E-mail address:} \texttt{zh14@nyu.edu}\\

\noindent \textsc{Department of Mathematical Sciences,
University of Bath,
Bath BA2 7AY, UK},\\
\textit{E-mail address:} \texttt{mw2319@bath.ac.uk}

\end{document}